\documentclass[12pt]{amsart}

\usepackage{etoolbox}

\makeatletter
\let\old@tocline\@tocline
\let\section@tocline\@tocline
\newcommand{\subsection@dotsep}{4.5}
\newcommand{\subsubsection@dotsep}{4.5}
\patchcmd{\@tocline}
  {\hfil}
  {\nobreak
     \leaders\hbox{$\m@th
        \mkern \subsection@dotsep mu\hbox{.}\mkern \subsection@dotsep mu$}\hfill
     \nobreak}{}{}
\let\subsection@tocline\@tocline
\let\@tocline\old@tocline

\patchcmd{\@tocline}
  {\hfil}
  {\nobreak
     \leaders\hbox{$\m@th
        \mkern \subsubsection@dotsep mu\hbox{.}\mkern \subsubsection@dotsep mu$}\hfill
     \nobreak}{}{}
\let\subsubsection@tocline\@tocline
\let\@tocline\old@tocline

\let\old@l@subsection\l@subsection
\let\old@l@subsubsection\l@subsubsection

\def\@tocwriteb#1#2#3{%
  \begingroup
    \@xp\def\csname #2@tocline\endcsname##1##2##3##4##5##6{%
      \ifnum##1>\c@tocdepth
      \else \sbox\z@{##5\let\indentlabel\@tochangmeasure##6}\fi}%
    \csname l@#2\endcsname{#1{\csname#2name\endcsname}{\@secnumber}{}}%
  \endgroup
  \addcontentsline{toc}{#2}%
    {\protect#1{\csname#2name\endcsname}{\@secnumber}{#3}}}%

\newlength{\@tocsectionindent}
\newlength{\@tocsubsectionindent}
\newlength{\@tocsubsubsectionindent}
\newlength{\@tocsectionnumwidth}
\newlength{\@tocsubsectionnumwidth}
\newlength{\@tocsubsubsectionnumwidth}
\newcommand{\settocsectionnumwidth}[1]{\setlength{\@tocsectionnumwidth}{#1}}
\newcommand{\settocsubsectionnumwidth}[1]{\setlength{\@tocsubsectionnumwidth}{#1}}
\newcommand{\settocsubsubsectionnumwidth}[1]{\setlength{\@tocsubsubsectionnumwidth}{#1}}
\newcommand{\settocsectionindent}[1]{\setlength{\@tocsectionindent}{#1}}
\newcommand{\settocsubsectionindent}[1]{\setlength{\@tocsubsectionindent}{#1}}
\newcommand{\settocsubsubsectionindent}[1]{\setlength{\@tocsubsubsectionindent}{#1}}

\renewcommand{\l@section}{\section@tocline{1}{\@tocsectionvskip}{\@tocsectionindent}{}{\@tocsectionformat}}%
\renewcommand{\l@subsection}{\subsection@tocline{2}{\@tocsubsectionvskip}{\@tocsubsectionindent}{}{\@tocsubsectionformat}}%
\renewcommand{\l@subsubsection}{\subsubsection@tocline{3}{\@tocsubsubsectionvskip}{\@tocsubsubsectionindent}{}{\@tocsubsubsectionformat}}%
\newcommand{\@tocsectionformat}{}
\newcommand{\@tocsubsectionformat}{}
\newcommand{\@tocsubsubsectionformat}{}
\expandafter\def\csname toc@1format\endcsname{\@tocsectionformat}
\expandafter\def\csname toc@2format\endcsname{\@tocsubsectionformat}
\expandafter\def\csname toc@3format\endcsname{\@tocsubsubsectionformat}
\newcommand{\settocsectionformat}[1]{\renewcommand{\@tocsectionformat}{#1}}
\newcommand{\settocsubsectionformat}[1]{\renewcommand{\@tocsubsectionformat}{#1}}
\newcommand{\settocsubsubsectionformat}[1]{\renewcommand{\@tocsubsubsectionformat}{#1}}
\newlength{\@tocsectionvskip}
\newcommand{\settocsectionvskip}[1]{\setlength{\@tocsectionvskip}{#1}}
\newlength{\@tocsubsectionvskip}
\newcommand{\settocsubsectionvskip}[1]{\setlength{\@tocsubsectionvskip}{#1}}
\newlength{\@tocsubsubsectionvskip}
\newcommand{\settocsubsubsectionvskip}[1]{\setlength{\@tocsubsubsectionvskip}{#1}}

\patchcmd{\tocsection}{\indentlabel}{\makebox[\@tocsectionnumwidth][l]}{}{}
\patchcmd{\tocsubsection}{\indentlabel}{\makebox[\@tocsubsectionnumwidth][l]}{}{}
\patchcmd{\tocsubsubsection}{\indentlabel}{\makebox[\@tocsubsubsectionnumwidth][l]}{}{}

\newcommand{\@sectypepnumformat}{}
\renewcommand{\contentsline}[1]{%
  \expandafter\let\expandafter\@sectypepnumformat\csname @toc#1pnumformat\endcsname%
  \csname l@#1\endcsname}
\newcommand{\@tocsectionpnumformat}{}
\newcommand{\@tocsubsectionpnumformat}{}
\newcommand{\@tocsubsubsectionpnumformat}{}
\newcommand{\setsectionpnumformat}[1]{\renewcommand{\@tocsectionpnumformat}{#1}}
\newcommand{\setsubsectionpnumformat}[1]{\renewcommand{\@tocsubsectionpnumformat}{#1}}
\newcommand{\setsubsubsectionpnumformat}[1]{\renewcommand{\@tocsubsubsectionpnumformat}{#1}}
\renewcommand{\@tocpagenum}[1]{%
  \hfill {\mdseries\@sectypepnumformat #1}}

\let\oldappendix\appendix
\renewcommand{\appendix}{%
  \leavevmode\oldappendix%
  \addtocontents{toc}{%
    \protect\settowidth{\protect\@tocsectionnumwidth}{\protect\@tocsectionformat\sectionname\space}%
    \protect\addtolength{\protect\@tocsectionnumwidth}{2em}}%
}
\makeatother

\makeatletter
\settocsectionnumwidth{2em}
\settocsubsectionnumwidth{2.5em}
\settocsubsubsectionnumwidth{3em}
\settocsectionindent{1pc}%
\settocsubsectionindent{\dimexpr\@tocsectionindent+\@tocsectionnumwidth}%
\settocsubsubsectionindent{\dimexpr\@tocsubsectionindent+\@tocsubsectionnumwidth}%
\makeatother

\settocsectionvskip{5pt}
\settocsubsectionvskip{0pt}
\settocsubsubsectionvskip{0pt}
    
\settocsectionformat{\bfseries}
\settocsubsectionformat{\mdseries}
\settocsubsubsectionformat{\mdseries}
\setsectionpnumformat{\bfseries}
\setsubsectionpnumformat{\mdseries}
\setsubsubsectionpnumformat{\mdseries}

\let\oldtableofcontents\tableofcontents
\renewcommand{\tableofcontents}{%
  \vspace*{-\linespacing}
  \oldtableofcontents}

\usepackage{amsmath, amssymb, amsthm, amsfonts, marginnote}
\usepackage{mathtools}
\usepackage{tikz-cd} 
\usepackage{comment}

\input xy
\xyoption{all}

\usepackage[bookmarks=true, bookmarksopen=true,%
bookmarksdepth=3,bookmarksopenlevel=2,%
colorlinks=true,%
linkcolor=blue,%
citecolor=blue,%
filecolor=blue,%
menucolor=blue,%
urlcolor=blue]{hyperref}

\theoremstyle{plain}
\newtheorem{theorem}{Theorem}[section]
\newtheorem{lemma}[theorem]{Lemma}
\newtheorem{prop}[theorem]{Proposition}
\newtheorem{proposition}[theorem]{Proposition}
\newtheorem{cor}[theorem]{Corollary}
\newtheorem{corollary}[theorem]{Corollary}

\theoremstyle{definition}

\newtheorem{definition}[theorem]{Definition}

\newtheorem{remark}[theorem]{Remark}
\newtheorem{example}[theorem]{Example}

\newtheorem{ex}[theorem]{Example}

\usepackage[margin=1in,marginparwidth=0.8in, marginparsep=0.1in]{geometry}

\def\CC{\mathbb{C}}
\def\DD{\mathbb{D}}

\def\NN{\mathbb{N}}

\def\RR{\mathbb{R}}
\def\R{\RR}

\def\ZZ{\mathbb{Z}}

\def\BR{\mathbb{R}}

\def\R{\mathbb{R}}
\def\Z{\mathbb{Z}}

\def\calU{\mathcal{U}}

\newcommand\cC{\mathcal{C}}

\newcommand\cE{\mathcal{E}}
\newcommand\cF{\mathcal{F}}
\newcommand\cG{\mathcal{G}}
\newcommand\cH{\mathcal{H}}

\newcommand\cK{\mathcal{K}}
\newcommand\cL{\mathcal{L}}

\newcommand\cN{\mathcal{N}}

\newcommand\cQ{\mathcal{Q}}

\newcommand\cT{\mathcal{T}}
\newcommand\cU{\mathcal{U}}

\newcommand\frF{\mathfrak{F}}

\newcommand\frL{\mathfrak{L}}
\newcommand\frM{\mathfrak{M}}

\newcommand\frc{\mathfrak{c}}

\newcommand\frf{\mathfrak{f}}
\newcommand\frg{\mathfrak{g}}

\newcommand\id{\textup{id}}

\newcommand{\Pic}{\textup{Pic}}

\newcommand{\Vect}{\textup{Vect}}

\newcommand\Hom{\textup{Hom}}

\newcommand{\isom}{\stackrel{\sim}{\to}}

\newcommand{\wt}[1]{\widetilde{#1}}

\newcommand\quash[1]{}

\newcommand\bs{\backslash}

\newcommand{\oo}{\infty}

\newcommand{\ssupp}{\mathit{ss}}
\newcommand\inthom{\mathcal{H}om}
\newcommand\ol{\overline}
\newcommand\on{\operatorname}

\newcommand\U{\mathsf U}

\newcommand\beq{\begin{equation}}
\newcommand\eeq{\end{equation}}

\usepackage{graphicx}

\newcommand\colim{\text{colim}}

\newcommand{\sh}{\mathit{sh}}
\newcommand{\mush}{\mathit{\mu sh}}
\newcommand\muhom{\mathit \mu hom}

\newcommand\Fun{\textup{Fun}}

\usepackage{hieroglf} 
\newcommand{\lunknot}{\prec\!\succ}

\usepackage{url}

\title{Sheaf quantization in Weinstein symplectic manifolds}

\dedicatory{}
\author{David Nadler}
\thanks{}
\address{Department of Mathematics, UC Berkeley, Evans Hall, Berkeley CA 94720, USA}
\email{nadler@math.berkeley.edu}
\author{Vivek Shende}
\thanks{}
\address{Centre for Quantum Mathematics, SDU, Campusvej 55, 5230 Odense M, Denmark $\qquad \qquad$ \& Department of Mathematics, UC Berkeley, Evans Hall, Berkeley CA 94720, USA}
\email{vivek.vijay.shende@gmail.com}
\date{}
\subjclass[2010]{}

\AtBeginDocument{%
   \def\MR#1{}
}

\begin{document}


\begin{abstract}
Using the microlocal theory of sheaves, we 
associate a category to a Weinstein 
manifold equipped with appropriate Maslov data.  By constructing a microlocal specialization functor, we show that exact Lagrangians
give objects in the category, and that the category is invariant under Weinstein homotopy. 
\end{abstract}


\maketitle
\thispagestyle{empty}

{\vspace{-8mm} \scriptsize
\tableofcontents
}

\newpage

\section{Introduction}

Since the groundbreaking work of Gromov and Floer, holomorphic curve methods have played a central role in 
symplectic topology.  The Fukaya category organizes many of the relevant structures: given a symplectic manifold $X$,
 Lagrangian submanifolds $L \subset X$ provide objects,  transverse intersection points $L \cap L'$ provide morphisms,
and counts of pseudoholomorphic disks with boundaries along the Lagrangian submanifolds provide the structure constants. 
Since these structure constants count solutions of a nonlinear partial differential equation, 
computing in Fukaya categories is highly nontrivial. 

Perhaps the simplest such calculation, made by Floer at the outset of the subject \cite{Floer-Witten},
is of the Floer complex $\Hom(L, L')$
when $L$ bounds no holomorphic disks, and $L'$ is a small displacement of $L$ by a Hamiltonian $H$.  In this case, $\Hom(L, L')$ is canonically identified with the Morse complex for $H|_L$.  As the Morse complex is quasi-isomorphic to
any other presentation of the cohomology of $L$,  it is natural to ask whether the Fukaya category can be constructed using 
only  topological methods.  

\vspace{2mm}

For cotangent bundles $X = T^*M$, the answer is known to be positive.  Indeed, 
the microlocal theory of sheaves on $M$, as introduced by Kashiwara-Schapira \cite{kashiwara-schapira}, has been used to investigate 
the symplectic topology of $X$, and has proven of comparable power \cite{tamarkin, guillermou-kashiwara-schapira, guillermou, 
stz, shende-conormal, chiu},
or indeed equivalent  \cite{nadlerzaslow, nadlermicrolocal, nrssz, gpsconstructible, viterbo-sheaf}, to Floer-theoretic methods.  
The calculations required in the sheaf-theoretic approach are topological and combinatorial (no differential equations
need be solved), and can be carried out in cases of central interest  \cite{FLTZ-Morelli, kuwagaki}.  Other benefits
of the sheaf-theoretic approach are 
 its direct connection with geometric representation theory
\cite{mirkovic-vilonen, rouquier-categorification, benzvi-nadler} 
 and the ability to work with general coefficients, e.g.~sheaves over 
the sphere spectrum \cite{jin-treumann, jin-BO, jin-J}.  

The purpose of the present article is to globalize the microlocal theory of sheaves from cotangent bundles to the more general class of exact symplectic manifolds called Weinstein manifolds, introduced in \cite{Weinstein-surgery, Eliashberg-Gromov-convex, Eliashberg-topological-Stein}. 
Notably, Weinstein manifolds include the symplectic manifolds determined by the K\"ahler form on a Stein manifold.  In fact, it is
known \cite{donaldson} that any compact symplectic manifold with integral symplectic form contains a symplectic divisor, realizing a multiple of the symplectic form, whose complement
is naturally Weinstein.  

Our main result -- see Theorem~\ref{thm: main} for a more  precise statement --  associates to any Weinstein manifold $W$, equipped with suitable topological `Maslov data',  a pre-triangulated differential graded (or more generally, stable  $\infty$-) category $\mathfrak{Sh}_\emptyset(W)$ enjoying  key geometric properties  of the Fukaya category. Namely, we show $\mathfrak{Sh}_\emptyset(W)$   is locally constant with respect to Weinstein deformations of $W$, and 
compact exact Lagrangian submanifolds $L \subset W$, similarly equipped with   suitable topological Maslov data, determine objects of $\mathfrak{Sh}_\emptyset(W)$ whose endormorphisms calculate the cohomology of $L$.

The considerations of the present article involve only elementary symplectic geometry and the microlocal theory of sheaves -- we do not 
rely on Floer theory at any stage.  However, it has been shown  that $\mathfrak{Sh}_\emptyset(W)$ is equivalent to 
the wrapped Fukaya category of $W$ \cite{gpsconstructible}, 
and thus may be put to the same rich range of applications.  Notably, the computability of
microsheaves has led to several recent advances in homological mirror symmetry
\cite{nadler-mirrorLG, nadler-wrapped, gammage-shende, bezrukavnikov-kapranov, shende-treumann-williams, gammage-webster-mcbreen, gammage-shende-2}. 
Moreover, in holomorphic exact symplectic settings, our approach allows perverse $t$-structures to be introduced \cite{perverse-microsheaves} and compared with Kashiwara's category \cite{CKNS2}. 
Additionally, our approach works  for general coefficients, e.g.~over any $E_\infty$-ring spectrum, with no additional difficulty. 

\vspace{2mm}

To state precise versions of our results, we next recall some basic ideas of the microlocal theory of sheaves.  We will write $sh(M)$ for the dg derived (or more generally, stable $\infty$-) category of sheaves 
on $M$ valued in $\Z$-modules (or more general stable coefficients).  Given an object $F \in sh(M)$,
which for simplicity we will call a sheaf, 
its microsupport $ss(F) \subset T^*M$ is the minimal closed subset whose complement ensures: 

\vspace{2mm}
\noindent {\bf Noncharacterstic propagation} \cite[Cor.~5.4.19]{kashiwara-schapira}, \cite{robalo-schapira}: 
for any $C^1$ function $\phi: M \to \R$ proper on the support of $F$, 
if $d\phi \notin ss(F)$ over $\phi^{-1}[a, b)$, then the following natural restriction maps are isomorphisms:
$$F(\phi^{-1}(-\infty, b)) \xrightarrow{\sim} F(\phi^{-1}(-\infty,a]) \xrightarrow{\sim} F(\phi^{-1}(-\infty, a)).$$ 

The microsupport of a sheaf is conic, i.e.~invariant under positive scaling of cotangent fibers, and  
a sheaf is locally constant if and only if its microsupport is contained in the zero-section. 
Some deeper facts that are cornerstones of the theory: the microsupport is always co-isotropic in an appropriate sense \cite[Thm.~6.5.4]{kashiwara-schapira}; 
a sheaf is constructible with respect to some subanalytic stratification if and only if its microsupport is a subanalytic Lagrangian \cite[Thm.~8.4.2]{kashiwara-schapira};
and the Riemann-Hilbert correspondence sends 
the characteristic cycle of a $D$-module to the microsupport of the sheaf of solutions \cite[Thm.~11.3.3]{kashiwara-schapira}.

\vspace{2mm}

At the most basic level, to study sheaves on $M$ `microlocally' means: given a point $p \in T^*M$, we  view all sheaves
whose microsupport is disjoint from $p$ as {\em trivial objects}.  As we vary the point $p \in T^*M$, we obtain a sheaf   of  categories  $\mush$ on $T^*M$ whose sections are called {\em microsheaves}. 
More precisely, for open $U \subset T^*M$, consider the full subcategory of sheaves $Null(U) \subset sh(M)$ 
whose microsupport is disjoint from $U$.
Then there is a natural presheaf of categories
$$\mush^{pre}(U) := sh(M) / Null(U)$$
and  one defines $\mush$ to be its sheafification. 
Some basic facts: the global sections $\mush(T^*M)$ are canonically equivalent to the original category of sheaves $\sh(M)$;
since microsupport is conic, $\mush$ is also conic, i.e.~invariant under positive scaling of cotangent fibers. 
In particular,  it is natural to view the restriction of $\mush$  away from the zero-section
as a sheaf of categories on  the  cosphere bundle $S^*M$. 

\vspace{2mm}

To go beyond cotangent bundles to more general exact symplectic manifolds, various authors have taken the strategy of 
embedding them into \cite{nadler-mirrorLG, nadler-wrapped, gammage-shende} or gluing them out of 
\cite{bezrukavnikov-kapranov, shende-treumann-williams, gammage-webster-mcbreen, gammage-shende-2}
subsets of cotangent bundles (or cosphere bundles in the contact case).  
In the present article, we will adopt the approach of \cite{shende-microlocal}, where high-codimensional embeddings lead to an unambigious notion of microsheaves on  contact manifolds. We will  recall and improve on 
these ideas below; for now, let us state the following result  for contact manifolds, which is the conceptual starting point for our investigation: 

\begin{theorem}[see Theorem~\ref{thm: descent}] \label{thm: contact microsheaves}
Let $(V, \xi)$ be a contact manifold, and $C$ a stable compactly generated symmetric monoidal $\infty$-category.  A null-homotopy
of the $C$-valued Maslov obstruction $V \to B^2 Pic(C)$ determines a sheaf of $C$-linear stable $\infty$-categories
$\mu sh$ on $V$, locally isomorphic to microsheaves on a cosphere bundle of the same dimension. 
\end{theorem}

\begin{remark}
We discuss the {\em Maslov obstruction} in detail in Section \ref{sec:descent}; we will call a choice of null
homotopy {\em Maslov data}. 
Our proof of Theorem \ref{thm: contact microsheaves} does not in fact depend on knowing 
what the Maslov obstruction is; applying the theorem in practice, however, does. 
When $C = mod-\Z$, the Maslov obstruction was computed by Guillermou \cite{guillermou};
in particular, in that case, Maslov data exists if and only if $2c_1(\xi) = 0$ 
for $\xi$ the contact
distribution, and homotopy classes of Maslov data form a torsor for 
$[V, BPic(mod-\Z)] = H^1(V, \Z) \oplus H^2(V, \Z/2\Z)$. This is the same  
homotopical data appropriate to defining Fukaya categories; see \cite[Sec. 5.3]{gpsconstructible} for some discussion.
When $C$ is the the category of modules over the sphere spectrum,
the Maslov obstruction  is characterized by works of Jin \cite{jin-BO, jin-J} in terms of the $J$-homomorphism.

A stable  polarization of the contact distribution  always gives Maslov data;
in fact, we prove Theorem \ref{thm: contact microsheaves} by first considering the case of a stable  polarization
(as already in \cite{shende-microlocal}), then
arguing by descent.
The same holds for many other results of the present article, and the main novelty 
can  already be appreciated  in the setting of a stable polarization.  In many applications,
a stable polarization is naturally present \cite{gammage-shende, gammage-webster-mcbreen, gammage-shende-2}. 
However, the general setup of Maslov data is a natural setting  for the construction of 
perverse microsheaves \cite{perverse-microsheaves}. 
\end{remark}

 Now let us return to our original goal of producing a categorical invariant of exact symplectic manifolds.
From the above theorem for contact manifolds,
we may produce a candidate
invariant as follows.    
First, for any subset $X$ of a contact manifold $V$, we may consider the microsheaves supported in $X$; this refines the theorem to 
provide a subsheaf of full subcategories $\mush_X \subset \mush$.  
Now consider an exact symplectic manifold $(W, \lambda)$, where we are careful to regard the symplectic primitive 1-form $\lambda$ as part of the structure. 
Consider the contactization $(W \times \RR, \xi)$ with contact structure $\xi$ given by the kernel
of the contact 1-form $\lambda+dt$.  
Then, to the exact symplectic manifold $(W, \lambda)$, we may associate 
the sheaf of categories $\mu sh_{(W, \lambda)} := \mu sh_{W \times 0}|_W$, i.e.~the sheaf of microsheaves on the contactization supported over the zero-section.

To a symplectic topologist,
there are some fundamental  shortcomings of the  sheaf of categories $\mu sh_{(W, \lambda)}$, stemming from   tensions between  properties of microsheaves and the traditional notions of Lagrangian submanifolds and equivalences
of exact symplectic manifolds.  

On the one hand, in (exact) symplectic topology, one is concerned with 
exact Lagrangian submanifolds $L$, i.e., those for which $\lambda|_L$ is the derivative of a function $f$.  Relatedly, one 
regards the exact symplectic manifold $(W, \lambda)$ as equivalent to $(W, \lambda + df)$ for reasonable functions $f$, i.e., one allows deformations of the primitive 1-form.  

On the other hand, 
 the support of any microsheaf is necessarily conic (see for example Corollary~\ref{cor: conicity}),
i.e.~invariant for the Liouville vector field $Z$  characterized by $\iota_Z d\lambda = \lambda$.  
Where said support is a  Lagrangian submanifold $L$, one must have $\lambda|_L = 0$, which
does not hold for most exact Lagrangian submanifolds. Thus 
 no microsheaf supported along  them exists, and we cannot hope to assign to them  a nonzero object.
 Relatedly, the trajectories of the Liouville vector fields 
associated to different primitive $1$-forms $\lambda$ and $\lambda + df$ are typically drastically different.
 So the sheaf of categories $\mush_{(W, \lambda)}$ cannot in general agree with (the pullback under any symplectomorphism, or indeed homeomorphism, of)   the sheaf of categories 
$\mush_{(W, \lambda+df)}$. In fact,  it is not difficult to show that  one can
extract the trajectories of the corresponding Liouville vector fields
from these sheaves of categories.
   
The above issues illustrate why the sheaf of categories $\mu sh_{(W, \lambda)}$ is not a good  
  invariant of exact symplectic manifolds.
But nothing  excludes the possibility that its {\em global sections} furnish an invariant, and this will be our main object of study
$$\mathfrak{Sh}(W, \lambda) := \Gamma(W, \mu sh_{(W, \lambda)})$$
At first, it is not clear that the passage to global sections makes any progress in addressing the above issues. Notably, 
 the contactomorphism invariance of microsheaves \cite[Chap. 6]{kashiwara-schapira}, \cite{guillermou-kashiwara-schapira} is insufficient to establish the exact symplectic invariance of $\mathfrak{Sh}(W, \lambda)$.  Indeed, in general, there is no contactomorphism carrying the defining
setup for $\mush_{(W, \lambda)}$ to that of $\mush_{(W, \lambda + df)}$. One can see this by noting the trajectories of the corresponding Liouville vector fields
 are an invariant of the germ along  $W \times 0$ of the contact structures of the respective contactizations.

\vspace{2mm} 
Our main results address the above difficulties.  To state them, 
recall that an exact symplectic manifold-with-boundary $(W, \lambda)$ is said to be a Liouville domain when the Liouville vector field $Z$
is outwardly transverse to the boundary.  In this case, for a subset $\Lambda \subset \partial W$, 
we consider the following full subcategory of microsheaves: 
$$\mathfrak{Sh}_\Lambda (W, \lambda) = \{F \, | \, \mathrm{support}(F) \cap \partial W \subset \Lambda \} \subset \mathfrak{Sh}(W, \lambda)$$
When $\Lambda = \emptyset$, the $Z$-invariance of microsheaves imples that such microsheaves are 
in fact supported in the core $\mathfrak{c}(W, \lambda) \subset W$, i.e.~the maximal locus 
of $W$ that is $Z$-invariant  and disjoint from $\partial W$.  More generally, for any $\Lambda \subset \partial W$,
by the same reasoning, such microsheaves are supported in the `relative core' $\mathfrak{c}(W, \Lambda, \lambda)$, i.e.~the union of the core $\mathfrak{c}(W, \lambda)$ and the $Z$-trajectories ending in $\Lambda$.

A basic observation underlying our approach is the following.  While there is no contactomorphism 
of $(W \times \R, \lambda + dt)$ relating the naturally embedded hypersurfaces $(W, \lambda)$ and 
$(W, \lambda + df)$, there is a sequence of contactomorphisms, given by a contact lift of the Liouville flow
\cite[Prop. 2.5]{eliashberg-weinstein},
that collapses $W \times \R$ to $\mathfrak{c}(W, \lambda)$.   Our main technical work uses this sequence of contactomorphisms to construct fully faithful functors 
$$\mu sh_X(W \times \R) \to \mu sh_{\mathfrak{c}(W, \lambda)}(W \times \R)$$
when a subset $X \subset W$ and the core $\mathfrak{c}(W, \lambda)$ 
are `sufficiently isotropic' as formalized in Def.~\ref{def: sufficiently isotropic}.
Any Whitney stratified subset which is isotropic at smooth points is sufficiently isotropic \cite[Lemma 3.2.2 and 3.2.8]{li-nadler-shende}.\footnote{In \cite{li-nadler-shende}, the term sufficiently isotropic is defined to also require a third condition -- self displacebility.  We use this property also here, but it is imposed separately through our `gapped' hypotheses.}
Thus  the core of any Liouville domain 
whose Liouville field is gradient-like and Morse-Smale is sufficiently isotropic.  Similarly,
we can consider homotopies of Liouville domains that are `sufficiently Weinstein', by which we mean that the  corresponding  Liouville 
domain underlying $W \times T^* \R$ has a sufficiently 
isotropic core.   The reasons we impose these isotropicity hypotheses are explored in Remark \ref{lament}.

What follows is an application-oriented statement collecting many of our main results, proved  in
Theorem \ref{thm: quantization}, Corollary \ref{cor: weinstein quantization}, Theorem \ref{thm: invariance},
 and Proposition~\ref{prop:gen version}:

\begin{theorem} \label{thm: main} For $(W, \lambda)$ a Liouville domain, equipped with Maslov data, the category 
$\mathfrak{Sh}_\emptyset(W, \lambda)$ of compactly supported microsheaves  satisfies:

\begin{enumerate}  
 \item {\bf Lagrangian objects.}  \label{lagrangian objects} If $(W, \lambda)$ is sufficiently Weinstein and
 $L \subset W$ is a smooth compact exact Lagrangian, then 
 secondary Maslov data on $L$ determines 
 a fully faithful functor from local systems to microsheaves
 $$\psi: Loc(L) \hookrightarrow  \mathfrak{Sh}_\emptyset(W, \lambda).$$  

 \item {\bf Invariance.} \label{invariance}  If $(W, \lambda)$ and $(W, \lambda' = \lambda + df)$ are sufficiently Weinstein, then there is an embedding  $\mathfrak{Sh}_\emptyset(W, \lambda) \hookrightarrow 
 \mathfrak{Sh}_\emptyset(W, \lambda')$. 
 If $\lambda$ and $\lambda'$ admit a sufficiently Weinstein homotopy,\footnote{In fact, 
 any Liouville homotopy of Weinstein domains  can be deformed to a sufficiently 
 Weinstein homotopy \cite[Prop. 2.42]{lazarev-sylvan-tanaka}.  This was pointed out to us by 
 Wenyuan Li, who has also shown~\cite{wenyuan}  by another argument that ``sufficiently Weinstein homotopy''
 can be weakened to ``Liouville homotopy'' in the theorem.  This means that the hypothesis is removed entirely, 
 since the linear interpolation $\lambda + sdf$ is such a homotopy.} 
 this embedding is an equivalence.
 \end{enumerate} 
\end{theorem}

\begin{remark}
We prove also a relative version of the theorem, i.e.~for the category $\mathfrak{Sh}_\Lambda(W, \lambda)$, when 
the relative core $\mathfrak{c}(W, \Lambda, \lambda)$ is sufficiently Weinstein (see  
Corollary \ref{cor: relative weinstein quantization}).  It is suitable for applications to Weinstein pairs, 
and in particular for producing objects from eventually conic Lagrangians. 
\end{remark}

\begin{remark}
In the theorem, {\em secondary Maslov data} is a homotopy between the restriction of the ambient Maslov data on $X$ 
and the Maslov data corresponding to the fiber polarization of $T^*L$.  It is the same topological information
required for a Lagrangian to define an object in the Fukaya category; see~\cite[Sec. 5.3]{gpsconstructible}.  
\end{remark}

Some new technical ingredients in the proof of Theorem \ref{thm: main} include: a criterion
for full faithfulness of nearby cycles (Theorem \ref{thm: gapped specialization is fully faithful}) and a 
 result  showing that microsheaf
 categories can be realized as localizations of categories of sheaves (Theorem \ref{thm: relative antimicrolocalization}). 
These are combined into a construction of microlocal specialization (Theorem \ref{thm:specialization}), which is 
then applied in Corollary \ref{cor: weinstein quantization} to prove (\ref{lagrangian objects})
and in Theorem \ref{thm: invariance} to prove (\ref{invariance}).  
In fact, the preceding arguments take place in the stably polarizable case;  we extend to the general case by descent from $U/O$
to $B Pic(C)$ in \S\ref{sec:descent}.

As alluded to earlier, one can summarize Theorem \ref{thm: main} by saying that  the category
$\mathfrak{Sh}_\emptyset(X, \lambda)$ enjoys key geometric properties  of the Fukaya category.   On the one hand, this is no accident: 
in \cite{gpsconstructible} it is shown that for stably polarized Weinstein manifolds $(X, \lambda)$, the category 
$\mathfrak{Sh}_\emptyset(X)$ is equivalent to the ind-completion of the wrapped Fukaya category.\footnote{The restriction to the 
stably polarized case is because the descent arguments of \S\ref{sec:descent} have not been
set up on the Fukaya side.}  On the other hand, the proof in  \cite{gpsconstructible} uses in a crucial way
the antimicrolocalization of the present paper
(Theorem \ref{thm: relative antimicrolocalization}) to reduce to the case of cotangent bundles.

\vspace{2mm}
{\bf Sheaf conventions} 
For a manifold $M$, we write $\sh(M)$ for the $\infty$-category of sheaves valued in the symmetric monoidal stable compactly generated
$\infty$-category $\mathcal{C}$ of the reader's choice.  As a basic example, one can take $\sh(M)$ to be the 
dg derived category of complexes of sheaves of abelian groups on $M$ localized along the acyclic sheaves.  
Another possibility for $\mathcal{C}$ is the stable $\infty$-category of spectra.  

We appeal often to results from the foundational reference on microlocal sheaf theory 
\cite{kashiwara-schapira}.  This work 
was written in the then-current terminology of bounded derived categories.  The hypothesis of
boundedness can typically be removed, though one needs in some cases different arguments, 
such as for proper base change, where to work in the unbounded setting
one needs to use \cite{spaltenstein}, and for the noncharacteristic deformation
lemma, where one needs \cite{robalo-schapira}.  We avoid
any constructions where boundedness is fundamental (e.g. asserting that applying Verdier duality twice gives the identity).
Passing to the dg or stable $\infty$ setting is solely a matter of having
adequate foundations; in the dg setting one can use \cite{drinfelddgquotient, toenmorita} and in
general \cite{luriehttpub, lurieha}.  The latter also provide foundations for working with sheaves
of categories more generally, which we will require as well.  

For a sheaf $F \in \sh(M)$, we write $ss(F) \subset T^*M$ for its microsupport, i.e.~the locus of codirections along which the
space of local sections is nonconstant.  For $X \subset T^*M$, we write $\sh_X(M)$ for sheaves microsupported in $X$.  

\vspace{2mm}
{\bf Acknowledgements}  We thank Yuichi Ike, Chris Kuo, Tatsuki Kuwagaki, and Wenyuan Li for  many helpful  remarks on previous drafts. 
We thank German Stefenich and  Nick Rozenblyum for their comments about sheaf foundations, and 
Peter Haine for his guidance with the stable homotopy theory
of  Section~\ref{sec:descent}.   
We thank an anonymous referee for many comments.

D.N. was supported by the grants NSF DMS-1802373, and in addition NSF DMS-1440140 while in residence at MSRI during the Spring 2020 semester.  V.S. was supported by the grant NSF CAREER DMS-1654545, and by the Simons Foundation and the Centre de Recherches Math\'ematiques, through the Simons-CRM scholar-in-residence program. 

\section{Propagation and displacement}

We begin by restating the fundamental result on noncharacteristic propagation: 

\begin{proposition}\cite[Cor.~5.4.19]{kashiwara-schapira}, \cite{robalo-schapira}
 \label{lem: propagation}
For any $C^1$ function $\phi: M \to \R$ proper on the support of $F$, 
if $d\phi \notin ss(F)$ over $\phi^{-1}[a, b)$, then the following natural restriction maps are isomorphisms:
$$F(\phi^{-1}(-\infty, b)) \xrightarrow{\sim} F(\phi^{-1}(-\infty,a]) \xrightarrow{\sim} F(\phi^{-1}(-\infty, a)).$$ 
\end{proposition} 

The covector $d \phi$ is the outward conormal to a sublevel set of $\phi$.  In this section, we give conditions guaranteeing
that certain such conormals can be displaced from microsupports.

\subsection{Positively displaceable from fibers} 

\begin{definition} \label{def: pdff}
A closed subset $\Lambda$ of a cosphere bundle $S^*M$  is {\em strictly positively displaceable from fibers}
(spdff) if for any cosphere fiber $S^*_p M$, 
there is a positive contact isotopy $\phi_t$, $t \in \R$, so that
for $t$ in some punctured neighborhood of zero $(-\epsilon(p), \epsilon(p)) \setminus 0$, we have 
$\phi_t(S^*_p M) \cap V = \emptyset$. 

We say $V$ is  positively displaceable from fibers (pdff) if there is some ambient contactomorphism $f: S^*M \to S^*M$ so that $f(V)$ is spdff. 

We use the same terminology for closed conic subsets of the cotangent bundle, applying the notions to their spherical projectivizations.
\end{definition}

To explain the term ``positive'' in the above definition,
we 
briefly recall the notion of positive contact vector fields.  On a contact manifold $(V, \xi)$, a 
contact vector field is by definition one whose flow preserves the distribution $\xi$; they are 
in natural bijection with sections of $TV / \xi$.  When $\xi$ is co-oriented, it makes
sense to discuss positive vector fields; a contact form $\alpha$ identifies sections of $TV / \xi$ with
functions and positive vector fields with positive functions.  The Reeb vector field is the contact
vector field associated to the constant function $1$; in other words it is characterized
by $\alpha(v_{Reeb}) = 1$ and $d\alpha(v_{Reeb}, \cdot) = 0$.  In fact any positive vector field
is a Reeb vector field, namely the vector field corresponding to $f$ is the Reeb field for 
the contact form $f^{-1} \alpha$.  

Another way of expressing the above is in terms of the symplectization $\hat V = T^+V^{\perp \xi}$,
the space of covectors $(n, \alpha) \in T^*V$ with $\alpha|_\xi = 0$ and $\alpha = c\lambda$, with $c>0$, for any cooriented local contact form $\alpha$. For example, if $V = S^*X$ is a cosphere bundle, then $\hat V \simeq T^\circ X = T^* X \setminus X$ is the corresponding punctured cotangent bundle.  A notation for going in the opposite direction: for a conic (or conic outside a compact set) subset $Z \subset T^*X$, we write $Z^\infty := Z \cap S^*X$ for the corresponding subset of the cosphere bundle. 

We view $\hat V \to V$ as an $\R_{>0}$-bundle with $\R_{>0}$-action given by scaling $\alpha$. 
An $\R_{>0}$-equivariant trivialization $\hat N \simeq N \times \R_{>0}$ over $N$ is the same data as an $\R_{>0}$-equivariant  function $f:\hat V \to \R_{>0}$. Its Hamiltonian vector field $v_f$ is $\R_{>0}$-invariant, preserves the level-sets $f^{-1}(c) \simeq N$, for all $c\in \R_{>0}$, and the corresponding Hamiltonian flow is a contact isotopy on each. For example, if $N = S^*X$, so that $\hat N = T^\circ X$, a metric on $X$ provides such a trivialization $T^\circ X \simeq S^*X \times \R_{>0}$ with the function $|\xi|:T^\circ X \to \R_{>0}$ given by the length of covectors and  positive contact isotopy given by normalized geodesic flow.

By a positive contact isotopy, we mean the flow of a (possibly time dependent) positive
contact vector field.  The 
typical example is the geodesic flow, which corresponds to the Reeb vector
field for the form $p dq$ restricted to the cosphere bundle in a given metric.  The 
image of a given cosphere fiber under this flow is the boundary of some round ball.  
We recall that any positive contact vector field also locally determines balls.  

A family of Legendrians is positive if it is the flow of a Legendrian under a positive
contact isotopy.  (By Gray's theorem, this can be checked just on the Legendrians without
need to explicitly construct an ambient positive isotopy.) 
An important example is the geodesic flow of a cosphere.  For time positive and smaller
than the injectivity radius, the image of the cosphere is the outward conormal to a ball; 
for time negative it will be the inward conormal.  Generalizing this we have: 

\begin{lemma} \label{lem: positive perturbation of cosphere}
Let $\eta_t: S^*M \to S^*M$ be the time $t$ contactomorphism of some positive contact flow. 
For small time $t < T(x)$, 
the cosphere $\eta_t(S^*_x M)$ is the outward conormal to the boundary 
of some open topological ball $B_t(x)$, with closure $\overline{B}_t(x)$.  For 
$t < t'$ one has $B_t(x) \subset B_{t'}(x)$.

For negative time, $\eta_t(S^*_x M)$ will be similarly an inward conormal.   
\end{lemma} \label{lem: generalized balls}
\begin{proof}
Because the projection of the cosphere to $M$ is degenerate, it is convenient to change
coordinates before making a transversality argument.  
We may as well assume $M = \R^n$.  Consider the map 
\begin{eqnarray*} \R \times S^{n-1} & \to & \R^n \\
(r, s) & \mapsto & r s 
\end{eqnarray*} 
Viewing the domain and codomain as the images of the front projections 
$J^1 S^{n-1} \to  \R \times S^{n-1}$ and 
$S^* \R^n \to \R^n$, said map lifts uniquely 
to a contactomorphism $J^1 S^{n-1} \to S^* \R^n$. 
This contactomorphism carries the zero section of $J^1 S^{n-1}$ to the cosphere over $0$ in $S^* \R^n$.

Embeddedness of the front is a generic condition, so any sufficiently small isotopy of the 
zero section in $J^1 S^{n-1}$ will leave the front in $\R \times S^{n-1}$ embedded.  A positive isotopy
will carry the front to a subset of $\R_{> 0} \times S^{n-1}$, hence its image under the embedding
$\R_{> 0} \times S^{n-1} \hookrightarrow \R^n$ will remain embedded. 
\end{proof} 

\begin{definition} \label{def: neighborhood defining function}
In the situation of Lemma \ref{lem: positive perturbation of cosphere}, 
it is always possible to choose a 
 function $f_x: M \to \RR_{\ge 0}$ with $f_x^{-1}(t)$ the 
front projection of $\eta_t(S^*_x M)$ for small $t$.  We term such a function a {\em neighborhood defining function}
at $x$ for the flow $\eta_t$.  

Note that a similar argument to Lemma \ref{lem: positive perturbation of cosphere} shows that the positive perturbation
of the conormal to a submanifold will be the outward conormal to its boundary.  Correspondingly, we may also speak of 
the neighborhood defining functions for a submanifold.  

By the {\em injectivity radius} at a point (or submanifold) we mean the first time $T > 0$ at which $\eta_T$ applied to the conormal 
fails to have an embedded projection.  The neighborhood defining function is well defined for $0 < t < T$, and we use it to 
give sense to the notion of distance to the point or submanifold.  (For the Reeb flow for a metric, these are of course the usual
notions of injectivity radius and distance.) 
\end{definition}

\begin{remark}
    A similar argument to that of Lemma \ref{lem: generalized balls} shows 
that for any (closed) submanifold $N \subset M$, any small positive contact 
perturbation of $S^*_N M$ gives the outward conormal to a neighborhood of $N$.  We also use the 
notion of neighborhood defining function in this context.
\end{remark}

The spdff condition ensures the following: 

\begin{lemma} \label{lem: reasonable stalks}  
For $\cF\in \sh(M)$, if $\ssupp(\cF) \subset T^*M$ is spdff, then for any $x\in M$, there exists $r(x)>0$ so that for any $0<r<r(x)$, the natural restrictions are isomorphisms
$$
\xymatrix{
\Gamma(\ol B_r(x), \cF) \ar[r]^-\sim & \Gamma(B_r(x), \cF) \ar[r]^-\sim & \cF_x
}
$$   
Here, $B_r(x)$ and $\overline B_r(x)$ are the open and closed balls associated by Lemma \ref{lem: positive perturbation of cosphere} to a contact isotopy witnessing the spdff property of $\ssupp(\cF)$ at $x$.
\end{lemma} 
\begin{proof}
Consider the associated neigborhood defining function $f_x$; then $df_x |_{f_x^{-1}(r)}$ is in the direction of the outward conormals
of  $B_r(x)$.  These conormals are the images under the flow of the conormal sphere at $x$, hence
by the spdff hypothesis, are disjoint from $ss(\cF)$ for all $r < r(x)$, for some $r(x)$.  
We have 
$B_r(x) := \phi^{-1}((-\infty,r))$ and $\overline{B}_r(x) = \phi^{-1}((-\infty,r])$.  
Now by noncharacteristic propagation (Proposition \ref{lem: propagation}), we find 
$$\Gamma(\ol B_r(x), \cF) \xrightarrow{\sim} \Gamma(B_r(x), \cF) \xrightarrow{\sim} \Gamma(B_{r'}(x), \cF)$$
for any $r' < r < r(x)$.  As balls are cofinal amongst open neighborhoods, we conclude 
$\Gamma(B_r(x), \cF) \xrightarrow{\sim} \cF_x$.  
\end{proof}

\begin{remark} 
The spdff condition functions as a kind of constructibility hypothesis.  Indeed, recall from \cite[Chap. 8]{kashiwara-schapira}
that a sheaf $\cF \in \sh(M)$ is ``weakly constructible'' when it is constant when restricted to each stratum of a locally subanalytic stratification
of $M$, and has perfect stalks; one says it is ``constructible'' when the stratification is globally subanalytic (in particular
with finitely many strata) and the stalks $\cF_x$ are perfect.  Weak constructibility in this sense is equivalent
to requiring that $ss(\cF)$ is (locally) a subanalytic Lagrangian; this is well known to imply that $ss(\cF)$ is spdff, for instance with the positive displacement provided by geodesic flow.  

On the other hand recall that ``weak cohomological constructibility'' is the following representability condition: 
$$\lim_{x \in U} \Hom(G, \cF(U)) \xrightarrow{\sim} \Hom(G, \cF_x)$$
as well as a similar condition for costalks.  Cohomological constructibility requires in addition that
$\cF_x$ is perfect \cite[Def. 3.4.1]{kashiwara-schapira}.  
It follows from Lemma \ref{lem: reasonable stalks} (and its costalk version, which holds
because we formulated the spdff hypothesis including negative times) that if $ss(\cF)$ is spdff then
$\cF$ is weakly cohomologically constructible.  

Thus we see that spdff sits between weak constructibility and weak cohomological constructibility.  
A sheaf on the line $\cF \in sh(\R)$  has spdff microsupport iff it is locally constant when restricted to a locally finite stratification.
Over higher dimensional manifolds, because of the subanalyticity requirement, it is easy to produce examples of sheaves which are 
not weakly constructible but have spdff microsupport.  On the other hand we do not actually know 
an example of a sheaf which is weakly cohomologically constructible without its microsupport being 
spdff.
\end{remark}

The following   
variant of Lemma~\ref{lem: reasonable stalks} will be used to study stalks of external products of sheaves:

\begin{lemma} \label{lem: reasonable stalks of prods}  
For $i=1,2$, suppose $\cF_i\in \sh(M_i)$ with $\ssupp(\cF_i) \subset T^*M_i$ spdff. Then for any $x_i\in M_i$, for $i=1,2$, there exists $r(x_i)>0$ so that for any $0<r_i<r(x_i)$, the natural restrictions are isomorphisms
$$
\xymatrix{
\Gamma(\ol B_{r_1}(x_1) \times \ol B_{r_2}(x_2), \cF_1\boxtimes \cF_2) \ar[r]^-\sim & 
\Gamma( B_{r_1}(x_1) \times  B_{r_2}(x_2), \cF_1\boxtimes \cF_2)
\ar[r]^-\sim & (\cF_1 \boxtimes \cF_2)_{(x_1, x_2)}
}
$$   
where the notation is as in 
Lemma~\ref{lem: reasonable stalks}.
\end{lemma}

\begin{proof}
The proof is a direct adaptation of the proof of  Lemma~\ref{lem: reasonable stalks}. We must
check that $\cF_1 \boxtimes \cF_2$ is non-characteristic with respect to the outward conormal of sufficiently small polyballs 
$B_{r_1}(x_1) \times  B_{r_2}(x_2)$. 

The   outward conormal of $B_{r_1}(x_1) \times  B_{r_2}(x_2)$ consists of three types of covectors: (i) 
along the side
 $\partial B_{r_1}(x_1) \times   B_{r_2}(x_2)$, covectors $(\xi_1, 0)$ with $\xi_1\not =0$ in the
 outward conormal of $\partial B_{r_1}(x_1)$ ; (ii)
  along the side $ B_{r_1}(x_1) \times  \partial B_{r_2}(x_2)$,
  covectors
  $(0, \xi_2)$  with $\xi_2 \not = 0$ in the
 outward conormal of $\partial B_{r_2}(x_2)$; and (iii)  along the corner
$\partial B_{r_1}(x_1) \times  \partial B_{r_2}(x_2)$, covectors
$(\xi_1, \xi_2)$ with at least one of $\xi_1, \xi_2$ non-zero, and when non-zero in the respective 
 outward conormal of  
 $\partial B_{r_1}(x_1),  \partial B_{r_2}(x_2)$. 

Now by \cite[Proposition 5.4.1]{kashiwara-schapira}, 
$\ssupp(\cF_1 \boxtimes \cF_2) \subset \ssupp(\cF_1) \times \ssupp(\cF_2)$. By assumption, $\ssupp(\cF_1) \subset T^*M_1$ 
and $\ssupp(\cF_2) \subset T^*M_2$ are  individually spdff, 
hence 
$\ssupp(\cF_1)$, $\ssupp(\cF_2)$ are disjoint from the respective 
outward conormal of  
 $\partial B_{r_1}(x_1),  \partial B_{r_2}(x_2)$.
 Thus
$\ssupp(\cF_1) \times \ssupp(\cF_2)$ is  disjoint from the 
 outward conormal of $B_{r_1}(x_1) \times  B_{r_2}(x_2)$ as 
 catalogued above.
\end{proof}

In the context of the Lemma, let us note that
the authors do not in fact know whether the external product of sheaves with spdff microsupport  again has spdff microsupport.

\subsection{Displaceability, Reeb chords, and gappedness} 

We will want to formulate displaceability conditions in terms of chord lengths. 

\begin{definition}
Fix a co-oriented contact manifold $(V, \xi)$ and positive contact isotopy $\eta_t$.  
For any subset $Y \subset V$ we write $Y[s] := \eta_s(Y)$.  

Given $Y, Y' \subset V$ we define the {\em chord length spectrum} of the pair to be the set of lengths of 
Reeb trajectories from $Y$ to $Y'$: 
$$cls(Y \to Y') = \{s \in \RR \,|\, Y[s] \cap Y' \ne \emptyset\}$$
We term $cls(Y) := cls(Y \to Y)$ the chord length spectrum of $Y$. 
\end{definition}

Borrowing terminology from quantum mechanics or functional analysis, we say a subset of $\mathbb{R}$ is {\em gapped} if its intersection with some interval $(-\epsilon, \epsilon)$ contains
only $0$.  We will say a family of subsets of a contact manifold is gapped when the corresponding chord length spectrum is gapped uniformly over the family: 

\begin{definition} \label{def:gapped}
Given a parameterized family of pairs $(Y_b, Y'_b)$ in some $V$ over $b \in B$, we say it is gapped if 
there is some $\epsilon > 0$ such that for all $b \in B$, 
 $(-\epsilon, \epsilon) \cap cls(Y_b \to Y'_b) \subset \{0\}$.  
 In case $Y = Y'$, we simply
say $Y$ is gapped.  
\end{definition}

\begin{example} \label{ex: reasonable from gapped} 
Def. \ref{def: pdff} can be reformulated as: $\Lambda$ is spdff if, for every cosphere fiber $L$, there is some positive contact isotopy $\phi_t$ with respect to which  $(L, \Lambda)$ is gapped. 
\end{example}

Note that in the notion of gapped, some positive contact isotopy is assumed given and fixed.  We will
often specify this explicitly and say that the family is ``gapped with respect to $\eta_t$''.  In the typical case
where $\eta_t$ is generated by a time independent positive Hamiltonian and hence is Reeb flow for some 
fixed contact form $\alpha$, we may say that the family is ``gapped with respect to $\alpha$''.  

We also say: 

\begin{enumerate}
\item 
$Y$ is {\em $\epsilon$-chordless} if $(0, \epsilon] \cap cls(Y) = \emptyset$. 
\item  
$Y$ is {\em chordless} if $cls(Y) = \{0\}$. 
\end{enumerate}

\begin{example}
Smooth Legendrians are $\epsilon$-chordless for any positive isotopy.  This may be seen by an argument similar to 
the proof of Lemma \ref{lem: positive perturbation of cosphere}.  
\end{example}

\begin{example} \label{ex: subanalytic locally chordless}
The curve selection lemma can be used to show that a subanalytic subset which is Legendrian at 
all smooth points is $\epsilon$-chordless for the Reeb flow for an analytic metric.  
\end{example}

\begin{example}
An example of a singular Legendrian which is not $\epsilon$-chordless for the geodesic flow: 
consider the Legendrian in $S^* \RR^2$
whose front projection is the union of the $x$-axis and the graph of $e^{-1/x^2} \sin(1/x)$. 
\end{example}

\begin{example}
A submanifold $W \subset V$ is said to be a (exact) 
symplectic hypersurface if, 
for some choice of contact form $\lambda$, the restriction $d\lambda|_W$ provides a symplectic form.  Such a hypersurface is  $\epsilon$-chordless, since 
the Reeb flow is along the kernel of $d \lambda$.  

Such a hypersurface is said to be Liouville if $\lambda|_W$ gives $W$ the structure of a Liouville
domain or manifold.   Liouville domains themselves are the subject of much inquiry; in particular
the subclass of Weinstein domains, see \cite{cieliebak-eliashberg}.  
The elementary contact geometry of
such Liouville hypersurfaces is studied in e.g. \cite{avdek, eliashberg-weinstein}.  
\end{example}

\begin{example}
In particular, if $\Lambda \subset V$ is a compact subset for which there exists a Liouville hypersurface $W \subset V$ with $\Lambda = Core(W)$ (``$\Lambda$ admits a ribbon''), 
then $\Lambda$ is $\epsilon$-chordless. 
\end{example}

We note some properties and examples of positive flows.
Generalizing the relationship between jet and cotangent bundles, we have: 

\begin{lemma} \label{lem: partial cosphere}
Let $M, N$ be manifolds.  
Then $T^*M \times S^*N \subset T^*(M \times N)$ is a contact level,
which is identified (by real projectivization) 
with $S^*(M \times N) \setminus 0_{T^*N}$. 

Given any positive contact isotopy $\phi_t$ on $S^*N$, the product $1 \times \phi_t$ on 
$T^*M \times S^*N$ remains positive.  In particular, the 
Reeb vector field on $T^*M \times S^*N$ is just the pullback of the Reeb vector field 
on $S^*N$.  
\end{lemma} 
\begin{proof}
Pull back the contact Hamiltonian by the projection map. 
\end{proof}

\begin{remark}
The most typical appearance of this fact is when $N = \R$, in which case it provides
the contactomorphism embedding $J^1 M \hookrightarrow S^*(M \times \R)$.  Note that 
the Reeb direction in $J^1 M$ happens only along the $\R$ factor; from this embedding
we can see that along a certain subset of $S^*(M\times \R)$, just moving in the $\R$ direction
is a positive contact isotopy. 

We will use the lemma later in the opposite case $M = \R$, for the purpose of learning similarly
that using the ``Reeb flow in the $M$ direction'' is positive along a certain subset of $S^*(M \times \R)$. 
\end{remark}

Next, consider 
manifolds $X, Y$ and positive contact  isotopies $\phi_t$ on $S^*X$, $\eta_t$ on $S^*Y$ corresponding 
in the symplectization to $\R_{>0}$-equivariant   
Hamiltonians $x: T^\circ X \to \R_{>0}$, $y: T^\circ Y \to \R$.   Observe we may define 
an $\R_{>0}$-equivariant  $C^1$-Hamiltonian by the formula $$x \# y := (x^2 + y^2)^{1/2}: T^\circ (X \times Y)\to \R_{>0}$$
Although $x, y$ are only defined on $T^\circ X, T^\circ Y$, since they are $\R_{>0}$-equivariant, their squares $x^2, y^2$  $C^1$-extend
to $T^* X, T^* Y$.
Let us write $(\phi \# \eta)_t$ for the flow  on $S^*(X \times Y)$ generated by  $x \# y $, and note it is  a 
 positive contact  $C^0$-isotopy.

\begin{example}
If $x$, $y$ are the length of covectors for some metrics
on $X, Y$, so that $\phi_t$, $\eta_t$ are normalized geodesic flow,  then $x \# y$ is the   length of covectors for the product metric on $X \times Y$, and $(\phi \# \eta)_t$ is again normalized geodesic flow.
\end{example}

\begin{remark} \label{rem:away from corners}
In what follows, we will only make use of this construction in the invariant open locus $T^\circ X \times T^\circ Y\subset T^\circ (X \times Y)$  where $x \# y$  is as smooth as $x$ and $y$. 
\end{remark}

\begin{lemma} \label{lem:three ways of looking at a reeb chord}
Let $V, W \subset T^\circ X$ be conic subsets, and $\phi_t$ the positive contact isotopy on $S^*X$
generated by a time-independent   $\R_{>0}$-equivariant  Hamiltonian $f:T^\circ X \to \R_{>0}$.

Then the following are in length-preserving bijection:
\begin{itemize}
\item Chords for $\phi_t$ from $V$ to $W$
\item Chords for $(\phi \# \phi)_t$ from the conormal of the diagonal to $-V \times W$. 
\end{itemize}

Let us write $\rho$ for the $\phi_t$ distance, and $\rho^\#$ for the $\phi_t \# \phi_t$ distance, and $\Delta$
for the diagonal.  When both are defined, $\rho(x_1, x_2) = \sqrt{2} \rho^{\#}((x_1, x_2), \Delta)$.   If $X$ is compact (or the contact Hamiltonian is bounded below), there is some
$\epsilon \in \R$ such that both distances are always defined in the $\epsilon$ neighborhood of the diagonal, 
hence $B_\epsilon(\Delta)$. 

Chords of length smaller than $\epsilon$ are then additionally in bijection with 
intersections in $T^*(B_\epsilon(\Delta))$ between the graph $d \rho(x_1, x_2)$ 
and $-V \times W$, with the chord length matching the value of $\rho$ at the intersection. 
\end{lemma}

\begin{proof}
Let $v_f = \omega^{-1}(df)$ denote the vector field generating $\phi_t$. Let $t\mapsto \gamma(x, \xi)_t$ denote the integral curve of $v_f$ through $(x, \xi) \in T^\circ X$.

Write $(x_1, \xi_1, x_2, \xi_2)$ for a point of $T^\circ(X\times X)$, and 
set $h:=f \# f = (f^2(x_1, \xi_1) + f^2(x_2, \xi_2))^{1/2}$.
By direct calculation,  $(\phi \# \phi)_t$ is generated by the vector field $v_{f\# f} = h^{-1}(f(x_1, \xi_1) v_f \oplus f(x_2, \xi_2) v_f)$.
Although $v_f$ is only defined on $T^\circ X$,  the scaling $f v_f$   $C^0$-extends to $T^*X$, vanishing on the 0-section. Indeed, it is the Hamiltonian vector field for $\frac{1}{2} f^2$ and its integral curve  through $(x, \xi) \in T^* X$ is given by
$t\mapsto \gamma(x, \xi)_{t'}$ where $t' = f(x, \xi) t$.

Thus the integral curve  of $v_{f\# f}$ 
 through $(x_1, \xi_1, x_2, \xi_2),  \in T^\circ(X \times X)$ 
 is given by $t\mapsto \gamma(x_1, \xi_1)_{t_1} \times \gamma(x_1, \xi_2)_{t_2}$
where $t_1 = h^{-1} f(x_1, \xi_1) t$, $t_2 = h^{-1} f(x_2, \xi_2) t$. Note that  $h^{-1} f(x_1, \xi_1)$, $h^{-1} f(x_2, \xi_2)$
are constant along the integral curve, so give a linear reparametrization of time. In particular, if we have
 $f(x_1, \xi_1) = f(x_2, \xi_2) = 1$,  then the reparametrization constant is $h^{-1} = \sqrt{2}/2$.

With the above formulas, all of the assertions are elementary to check. 

For example, suppose $\gamma(x_1, \xi_1)_t =
(x_2, \xi_2)$ at $t= \ell$, i.e.~ provides a chord of $\rho$-length $\ell$ from $(x_1, \xi_1)$ to $(x_2, \xi_2)$. Suppose without loss of generality that $f(x_1, \xi_1) = f(x_2, \xi_2) = 1$ so that $h = \sqrt{2}$. Set $(x', \xi') = \gamma(x_1, \xi_1)_{\ell/2}$. Then the integral curve $t\mapsto \gamma(x', -\xi')_{t_1} \times \gamma(x', \xi')_{t_2}$ lies in the conormal to the diagonal at $t= 0$, and equals $(x_1, -\xi_1, x_2, \xi_2)$ at $t = \frac{\ell}{\sqrt{2}}$,  i.e.~ provides a chord of $\rho^\#$-length $\frac{\ell}{\sqrt{2}}$ from the conormal to the diagonal to $(x_1, -\xi_1, x_2, \xi_2)$. 

Conversely,  given an integral curve $t\mapsto \gamma(x', -\xi')_{t_1} \times \gamma(x', \xi')_{t_2}$, the integral curve 
$t\mapsto  \gamma(x', \xi')_{t - \ell/2}$ provides the inverse construction.

Finally, note that the level-sets of $\rho$ are the fronts of the flow of the conormal to the diagonal. Thus $d\rho$ lies in a subset precisely when the  flow of the conormal  intersects the subset.
\end{proof}

\begin{remark}
In particular, Lemma \ref{lem:three ways of looking at a reeb chord} allows the chord length spectrum
and the gapped condition to be 
reformulated in terms of the graph of the derivative of the distance from the diagonal.
\end{remark}

\begin{remark}
We will want to use the above lemma (and remark) in connection with the ``gapped'' condition of Def.~\ref{def:gapped}, 
and so henceforth we will only consider gappedness with respect to a fixed contact form.  This is solely due 
to the above lemma; if a similar result could be shown for time dependent flows then we could use them throughout.  
\end{remark}

\subsection{Microlocal criterion for comparing mixed restrictions}

Let $X$ be a manifold, and $Y_1, Y_2 \subset X$ transverse submanifolds. Let $Y = Y_1 \cap Y_2$ so we have a Cartesian diagram
\beq\label{eq:key square abstract}
\xymatrix{
\ar[d]_-{i_2} Y \ar[r]^-{i_1}  & Y_1 \ar[d]^-{y_1}\\
Y_2 \ar[r]^-{y_2} & X
}
\eeq
There is a natural map of functors
\beq\label{eq:bc}
\xymatrix{
i_1^* y_1^! \ar[r] &  i_2^! y_2^*
}
\eeq

We write $p:Y \to pt$ for the map to a point, so that after taking global sections, \eqref{eq:bc} becomes: 
 \beq\label{eq:glob bc}
\xymatrix{
p_* i_1^* y_1^! \ar[r] &  p_* i_2^! y_2^*
}
\eeq

We are interested in knowing when \eqref{eq:glob bc} is an isomorphism. We first note that  \eqref{eq:glob bc} can be an isomorphism even when \eqref{eq:bc} fails to be.

\begin{example}
Take $X = \RR^2_{\geq 0} \times \RR$, $Y_1 = \RR_{\geq 0} \times \{0\} \times \RR$, $Y_2 = \{0\}\times \RR_{\geq 0} \times \RR$, $Y =  \{(0,0)\} \times \RR_{\geq 0}$. Let $k_{[0,1)}$ be the constant sheaf on the open-closed interval, so the constant sheaf $k_{(0,1)}$  on the open interval $*$-extended across $0$ and $!$-extended across $1$.  For  any sheaf $\cF$ on $\RR^2_{\geq 0}$, consider the external product $\cF \boxtimes k_{[0,1)}$ on $X$. Then no matter what form \eqref{eq:bc} takes, the terms of \eqref{eq:glob bc} will be zero, hence \eqref{eq:glob bc} will be an isomorphism.  
\end{example}

\begin{remark} 
The functors and  natural transformation in \eqref{eq:bc}, \eqref{eq:glob bc} were classically of interest in the case when $Y$ is the fixed locus
of some `hyperbolic' flow or endomorphism \cite{goresky-macpherson-lefschetz}.  The prototypical case
is that of the gradient flow, where $Y_1$ and $Y_2$ are taken to be the unstable and stable manifolds, respectively.  
Assuming all normal bundles are orientable, we have 
$i_1^* y_1^! \Z = i_2^! y_2^* \Z = \Z[-\mathrm{index}(Y)]$.   
Taking global sections and summing over all fixed loci gives the structure
on which the Morse differential acts (but not the differential).  

A result with many applications in geometric
representation theory is that on a (normal) algebraic
variety with $\mathbb{G}_m$ action, the morphism \eqref{eq:bc} is an isomorphism for equivariant sheaves, and 
moreover the corresponding `Morse differential' vanishes \cite{braden}. 
\end{remark}

\begin{example}
Take $X = \RR^2_{\geq 0}$, $Y_1 = \RR_{\geq 0} \times \{0\}$, $Y_2 = \{0\}\times \RR_{\geq 0}$, $Y = \{(0, 0)\} $. Take $\cF = k_{\Delta_{>0}*}$ to be the standard extension of the constant sheaf on the diagonal $\Delta_{>0} = \{(t, t) | t\in \RR_{> 0}\}$. Then $i_1^* y_1^!\cF \simeq 0$, while  $i_2^! y_2^*\cF \simeq k_Y$.
\end{example}

The following is a natural ``correction" of the issue arising in the prior example.

\begin{lemma}\label{eq: univ lemma}
Take $X = \RR^2_{\geq 0}$, $Y_1 = \RR_{\geq 0} \times \{0\}$, $Y_2 = \{0\}\times \RR_{\geq 0}$, $Y = \{(0, 0)\} $. 

Suppose $\cH\in \sh(X)$ is constant on an open submanifold of  $\mathring X = \RR_{>0}^2$
near to $Y$.
Then the base-change map is an equivalence
\beq
\xymatrix{
i_1^* y_1^!\cH \ar[r]^-\sim &  i_2^! y_2^*\cH
}
\eeq
\end{lemma}

\begin{proof}
It is elementary to check  the  base-change map is an equivalence for $\cH$ supported on $Y_1$ or $Y_2$.
By devissage we are reduced to the case that $\cH$ is the standard extension ($*$-pushforward) of 
$\cH|_{\RR_{>0}^2}$.  By hypothesis this sheaf is constant, and hence so is its extension, making the result obvious. 
 (Compare \cite[Sec. 9]{goresky-macpherson-lefschetz}.) 
\end{proof}

Let us reduce the general situation of  \eqref{eq:glob bc} to the model situation of the lemma.

\begin{prop} \label{prop:hyperbolic basechange} 
Let $X$ be a manifold, and $Y_1, Y_2 \subset X$ transverse submanifolds. Let $Y = Y_1 \cap Y_2$,
and denote maps: 
\beq
\xymatrix{
\ar[d]_-{i_2} Y \ar[r]^-{i_1}  & Y_1 \ar[d]^-{y_1}\\
Y_2 \ar[r]^-{y_2} & X
}
\eeq
Assume that $X$ is compact,
or more generally the support of a sheaf $\cH \in \sh(X)$  is compact. 

Fix a positive contact flow on $S^* X$.  
For $i = 1, 2$, let  $f_i:X \to \RR_{\geq 0}$ be associated neigborhood defining
functions (Def. \ref{def: neighborhood defining function}) for $Y_i$. 
Consider the map $f = f_1 \times f_2: X \to \RR_{\geq 0}^2$.  

Suppose for some $\epsilon$, over $0 < f_1, f_2 < \epsilon$, the intersection 
$$\on{span}(df_1, df_2) \cap \ssupp(\cH) \subset T^*X
$$
 lies in the zero-section.

Then the natural map on global sections is an equivalence 
\beq
\xymatrix{
\Gamma(Y,  i_1^* y_1^!\cH) \ar[r]^-\sim & \Gamma(Y, i_2^! y_2^*\cH)
}
\eeq
\end{prop}\label{prop: bc}

\begin{proof}
We compute the map on global sections after first pushing forward along $f$.  We assumed
$supp(\cH)$ compact; in particular $f$ is proper on it, hence we 
may invoke proper base-change to express \eqref{eq:glob bc} in the form 
 \beq\label{eq:factored bc}
\xymatrix{
 i_1^* y_1^!f_*  \ar[r] & i_2^! y_2^* f_* 
}
\eeq

Properness of $\cH$, the hypothesis on $\on{span}(df_1, df_2)$, and the compatibility of microsupport with pushforwards \cite[Prop. 5.4.4]{kashiwara-schapira} 
ensures that microsupport $\ssupp(f_* \cH) \subset T^* \RR^{2}_{\ge 0}$ in fact lies in the zero-section over $\RR^2_{>0}$.  
This implies local constancy of this pushforward in this region.   
We now apply the prior lemma to $f_* \cH$ and conclude that \eqref{eq:factored bc}, and hence \eqref{eq:glob bc}, evaluated on $\cH$  is an equivalence.
\end{proof}

Note that as we have only been interested in the geometry of $X$ near $Y$,  one is free anywhere to replace $X$ by a small neighborhood of $Y$.

\section{Relative microsupport}

Let $\pi: E \to B$ be a smooth fiber bundle.  We write $T^*\pi$ for the relative cotangent bundle, i.e.
the bundle over $E$ defined by the short exact sequence
$$0 \to \pi^* T^*B \to T^*E \xrightarrow{\Pi} T^* \pi \to 0$$
The fibers of $T^*\pi$ are the relative cotangent spaces: for $e \in E$, 
$$(T^*\pi)_e = T^*_e (E_{\pi(e)})$$
We may also view $T^* \pi$ as
a fiber bundle over $B$, with fibers the cotangent bundles to the fibers of $\pi$: for $b \in B$, 
$$(T^*\pi)_b =  T^*(E_b)$$

In this section, we present some lemmas concerning the microlocal 
theory of sheaves on $E$, on the fibers $E_b$, and on the base $B$. 

Regarding the comparison with the base, we have the following consequence
of the standard estimate on microsupport of a pushforward (\cite[Proposition 5.4.4]{kashiwara-schapira}): 

\begin{definition} \label{def:relatively noncharacteristic} 
We say $\Lambda \subset T^* E$ is $\pi$-noncharacteristic, or synonymously $B$-noncharacteristic, 
if $\Lambda \cap \pi^* T^*B \subset 0_{T^*X}$. 
We say a sheaf $\cF$ on $E$ is
$\pi$-noncharacteristic if its microsupport $\ssupp(\cF)$ is
$\pi$-noncharacteristic, and $\pi$ is proper on the support of $\cF$.  
\end{definition}

\begin{lemma}
If $\cF$ is $\pi$-noncharacteristic, then $\pi_* \cF$ has microsupport
contained in the zero section of $B$ and hence is locally constant. 
\end{lemma}

The interaction with the fiber will require more subtle
microsupport estimates, recalled in the next subsection.

\subsection{Microsupport estimates} \label{app:hats}

We recall from \cite[Chap. 6]{kashiwara-schapira} the standard estimates
on how microsupport interacts with various functors. 

These are expressed in terms a certain operations on conical
subsets.  
A special case is the $\hat{+}$ construction.  It is an operation
on conical subsets of $T^*X$, the result of which is larger than the naive sum.   It is defined in \cite[Chap. 6.2]{kashiwara-schapira} in 
terms of a normal cone
construction.  In practice, one uses the equivalent 
pointwise characterization 
\cite[6.2.8.ii]{kashiwara-schapira}, which we will just take as the definition.   

\begin{definition}
Given conical $A, B \subset T^*X$, 
a point $(x, \xi) \in A \hat{+} B$ if, in local coordinates, 
there are sequences 
$(a_n, \zeta_n) \in A$ and $(b_n, \eta_n) \in B$ with 
$a_n, b_n \to x$ and $\eta_n + \zeta_n \to \xi$, satisfying
the estimate $|a_n - b_n| |\eta_n| \to 0$ (or equivalently
$|a_n -  b_n||\zeta_n| \to 0$).
\end{definition}

Evidently $a_n, b_n \to x$ implies
that $|a_n - b_n| \to 0$.  So certainly sequences with
bounded $\zeta_n$ satisfy the estimate.  Of course,
in this case passing
to a subsequence gives $(a_n, \zeta_n)$ convergent
in $A$, hence the element $(x, \xi)$ would be already in 
$A + B$.  The additional points of $A \hat{+} B$ arise when 
$|\zeta_n| \to \infty$, in a manner controlled by 
$|a_n - b_n| |\zeta_n| \to 0$.  

\begin{remark}
Omitting the condition $|a_n - b_n| |\zeta_n| \to 0$ would result in a larger $\hat{+}$, for
which all the below estimates would be true, but weaker --
the term with the $\hat{+}$ is always used as an upper bound.  However, these 
weaker statements would still be  strong enough for our purposes here: we never actually use
the constraint $|a_n -  b_n||\zeta_n| \to 0$.
\end{remark}

\begin{example}
Consider the loci $A$ and $B$ given by the conormals to 
the $x$ and $y$ axes in $\RR^2_{x,y}$.  Then $A \hat{+} B
= A + B$ consists of the union of these conormals and the conormal
to the origin.  
\end{example}

\begin{example}
Consider the loci $A$ and $B$ given by the conormals to 
$\{y = 0\}$ and $\{y = x^2\}$ in 
$\RR^2_{x,y}$.  Then $A + B$ is just the union of these 
conormals, whereas $A \hat{+} B$ is the union of these 
with the conormal to the origin.  
\end{example}

A related operation arises in the context of a map $f: Y \to X$.  
Denote the natural maps 
$T^*Y \xleftarrow{df} f^* T^*X \xrightarrow{f_\pi} T^* X$. 

\begin{definition}
Given conic $A \subset T^*X$ and $B \subset T^*Y$, 
then $(y, \eta) \in f^{\#}(A, B)$ if (in coordinates) there is a sequence 
$(x_n, \xi_n) \times (y_n, \eta_n) \in A \times B$ with
$y_n \to y$, $x_n \to f(y)$, and 
$df_{y_n}(\xi_n) - \eta_n \to \eta$, while respecting
the estimate $|x_n - f(y_n)| |\xi_n| \to 0$.  

One writes $f^\#(A) := f^\#(A, T_Y^*Y)$.  The $\hat{+}$ construction
is the special case $A \hat{+} B = id^\#(A, -B)$. 
\end{definition}

Just as $A + B \subset A \hat{+} B$, we have 
$- B + df(f_{\pi}^{-1}(A)) \subset f^\#(A, B)$ as the locus where the 
$\xi_n$ remain bounded.  

\begin{example}
Of particularly frequent use is the case when $f$ is (locally) a closed
embedding.  We take coordinates 
$(z, y, \zeta, \eta)$ on $T^*X$, with $y$ the coordinates on $Y$, 
$\zeta$ the directions conormal to $Y$, and $\eta$ the directions
cotangent to $Y$.  Then 
a point $(y, \eta)$ is in $f^\#(A) \subset T^*Y$ if there 
is a sequence $(z_n, y_n, \zeta_n, \eta_n) \in A \subset T^*X$ with 
$(y_n, \eta_n) \to (y, \eta)$ while $z_n \to 0$ and 
$|z_n||\zeta_n| \to 0$.  
\end{example}

These constructions are useful to describe how microsupport
is affected by various functors.  

\begin{lemma} \label{standardestimates} Some microsupport estimates: 
\begin{itemize}
\item \cite[Theorem 6.3.1]{kashiwara-schapira} For $j: U \to X$ an open inclusion, and $\cF \in \sh(U)$
$$ss(j_* \cF) \subset ss(\cF) \hat{+} N^+ U \qquad \qquad 
ss(j_! \cF) \subset ss(\cF) \hat{+} N^- U$$
In particular, if $U$ is the complement of a closed submanifold
$Y \subset X$, then \cite[Proposition 6.3.2]{kashiwara-schapira}
$$ss(j_* \cF)|_Y \subset ss(\cF) \hat{+} T_Y^* X \qquad \qquad 
ss(j_! \cF)|_Y \subset ss(\cF) \hat{+} T_Y^* X$$
\item \cite[Cor. 6.4.4]{kashiwara-schapira} For $f: Y \to X$, 
$$ss(f^* \cF) \subset f^\#(ss(\cF)) \qquad \qquad ss(f^! \cF) \subset 
f^\#(ss(\cF))$$
\item \cite[Cor 6.4.5]{kashiwara-schapira} For sheaves $\cF, \cG$
$$ss(\cF \otimes \cG) \subset ss(\cF) \hat{+} ss(\cG) \qquad \qquad 
ss(\inthom(\cF, \cG)) \subset -ss(\cF) \hat{+} ss(\cG)$$
\end{itemize}
\end{lemma}

\subsection{Relative microsupport} \label{sec: relative microsupport}

We return to our 
smooth fiber bundle $\pi: E \to B$.  Recall we write 
$\Pi: T^*E \to T^*\pi$ for the projection 
to the relative cotangent bundle.  

\begin{definition}
For $\cF \in \sh(E)$, we define 
the relative microsupport to be the conical locus in $T^* \pi$ given by 
$$ss_{\pi}(\cF) := \overline{\Pi(ss(\cF))}$$
\end{definition}

The definition is motivated by the following connection with the $\hat{+}$ construction. 

\begin{lemma} \label{lem:addrelcotpushpull}
Let $\pi: E \to B$ be a fiber bundle, and let $\Lambda \subset T^*E$
be conical.  Let $\Pi: T^*E \to T^* \pi$ be the natural projection to
the relative cotangent bundle.  Then 
$$\Lambda \hat{+} \pi^* T^*B = \Pi^{-1}(\overline{\Pi(\Lambda)})$$
\end{lemma} 
\begin{proof}
The assertion is local on $E$; choose local coordinates $(x, e)$ 
with $x$ the coordinates along $B$ and $e$ the bundle coordinates. 
Let $(x,e,\xi, \eta)$ be corresponding
coordinates on $T^*E$.

Then $(x,e,\xi, \eta) \in \Lambda \hat{+} \pi^* T^* B$ iff there is 
a sequence $(x_n, e_n, \xi_n, \eta_n) \in \Lambda$ and some
$(x'_n, e'_n, \xi'_n, 0)$ such that $x_n, x'_n \to x$ and
$e_n, e'_n \to e$ and $\xi_n + \xi'_n \to \xi$ and $\eta_n \to \eta$,
and some estimate holds.  But we may as well take $x_n = x'_n$ and
$e_n = e'_n$, so the estimate is vacuous.  The condition on 
$\xi_n$ is vacuous as well, since $\xi'_n$ can be chosen arbitrarily.  
Thus the condition with content is $\eta_n \to \eta$.  I.e.,
$(x, e, \eta)$ is a limit point of $\Pi(\Lambda)$. 
\end{proof}

Suppose given in addition some submanifold $\alpha: A \subset B$.
We write $\pi_A: E_A \to A$ for the restricted bundle. 
By abuse of notation we write  $\alpha: E_A \subset E_B$.
We denote by $\Pi_A: T^*E_A \to T^* \pi_A$ for the relative cotangent bundle.
Note the natural identification $T^*\pi_A = T^*\pi|_{E_A}$. 
In this setting, we may write $\pi_B := \pi$ and $\Pi_B:= \Pi$ for
clarity.  

We have the following estimate on images in the relative
cotangent: 

\begin{lemma}
For $\Lambda \subset T^*E$ conical, we have
$\Pi_A(\alpha ^\# \Lambda) \subset \overline{\Pi_B(\Lambda)}|_{E_A}$. 
\end{lemma}
\begin{proof}
We take local coordinates  $(y,z,e,\gamma, \zeta, \eta)$ on $T^*E$,
where $y$ are coordinates on $A$, the $z$ are coordinates 
on $B$ in the normal directions to $A$, the $e$ are bundle coordinates,
and the $\gamma, \zeta, \eta$ are corresponding cotangent
coordinates. We use corresponding notations for coordinates on related spaces.

Consider a point 
$(y, e, \gamma, \eta) \in \alpha^\# \Lambda \subset T^*E_A$.  
By definition, there
must be a sequence $(y_n, z_n, e_n, \gamma_n, \zeta_n, \eta_n)
\in \Lambda$
with 
$(y_n, e_n, \gamma_n, \eta_n) \to (y, e, \gamma, \eta)$ and
$z_n \to 0$ (and also $|z_n||\zeta_n| \to 0$, though we won't use it). 

The image of $(y, e, \gamma, \eta)$ in $T^* \pi_A$ is 
$(y, e, \eta)$.  We must show this point is already in 
$\overline{\Pi_B(\Lambda)}$.  Thus consider the image 
of the above sequence in $T^*\pi_B$, i.e.  
$(y_n, z_n, e_n, \eta_n)$.  By the above hypotheses, this 
converges to $(y, 0, e, \eta)$.  
\end{proof} 

This lemma implies that the relative microsupport interacts well with restriction to submanifolds
of the base. 

\begin{lemma} \label{lem:bundlerestrictionestimate}
For $\cF \in \sh(E)$, we have $ss_\pi(\alpha^* \cF) \subset ss_\pi(\cF)|_{E_A}$, 
and $ss_\pi(\alpha^! \cF) \subset ss_\pi(\cF)|_{E_A}$.  
\end{lemma}
\begin{proof}
What is being asserted is that 
$\Pi_A(ss(\alpha^* \cF)) \subset \overline{\Pi_B(ss(\cF))}|_{E_A}$
and $\Pi_A(ss(\alpha^! \cF)) \subset \overline{\Pi_B(ss(\cF))}|_{E_A}$.  This
follows from the previous lemma and the 
estimate $ss(\alpha^* \cF) \subset \alpha^\# ss(\cF)$. 
\end{proof}

\begin{corollary} \label{pointrestrictionestimate}
In particular, if $P \in B$ is a point, 
 $ss(\cF|_{E_P}) \subset ss_\pi(\cF)|_{E_P}$. 
\end{corollary}

\begin{corollary} \label{cor:relativehomestimate}
Let $\pi_1, \pi_2: E \times_B E \to E$ denote the projections to the factors.  For 
sheaves $\cF_1, \cF_2$ on $E$, consider the relative external hom
$$\inthom_{E \times_B E}(\pi_1^* \cF_1, \pi_2^! \cF_2) \in \sh(E \times_B E)$$
Let $\tilde{\pi}: E \times_B E \to B$ be the structure map, and 
$\widetilde{\Pi}: T^*(E \times_B E) \to T^*\tilde{\pi}$ the projection
to the relative cotangent.   Then

$$ss_{\tilde{\pi}}(\inthom_{E \times_B E}(\pi_1^* \cF_1, \pi_2^! \cF_2)) \subset 
- ss_\pi(\cF_1) \boxplus ss_\pi(\cF_2)$$
\end{corollary}
\begin{proof}

This follows by applying Lemma \ref{lem:bundlerestrictionestimate} to $\delta: B \to B \times B$. 
(Here we use the general identity 
$f^! \inthom (F, G) = \inthom (f^* F, f^! G)$.  This is Prop \cite[3.1.13]{kashiwara-schapira}, which is there
stated under some boundedness hypotheses on $F, G$, but the proof is a series of adjunctions which hold
in general.)
\end{proof}

\subsection{Microsupport of nearby cycles}  \label{sec: nearby microsupport} 

The nearby cycles functor in sheaf theory is a useful notion for taking limits of families of sheaves.  We recall 
its construction.  Let $X$ be a topological space and $\pi: X \to \R_{\ge 0}$ a continuous map.  We consider the diagram
$$
\xymatrix{
\ar[d]_{\pi} X_{> 0}  \ar@{^(->}[r]^-j & \ar[d]_{\pi} X & \ar@{_(->}[l]_-i X_0 \ar[d]_{\pi} \\
 \R_{> 0}  \ar@{^(->}[r]^-j &  \R_{\ge 0} & \ar@{_(->}[l]_-i  \{0\}
}
$$
where $X_{> 0} := \pi^{-1}(\R_{> 0})$ and $X_0 = \pi^{-1}(0)$.  
The nearby cycles functor is by definition
$$
\xymatrix{
\psi = i^* j_*: \sh(X_{> 0})  \ar[r] & \sh (X_0)
}
$$
By precomposing with $j^*$, we may also consider $\psi = i^* j_* j^* : \sh(X)  \to \sh (X_0)$.  
In this setting we also have the vanishing cycles functor, $\phi := Cone(i^* \to i^* j_* j^*)$.   

Consider the case when $X_0$ is a manifold, 
and $X = X_0 \times \R_{\ge 0}$ 
and $\pi$ is the projection to the second factor. 
$$
\xymatrix{
\ar[d]_{\pi} X_0 \times \R_{> 0}  \ar@{^(->}[r]^-j & \ar[d]_{\pi} X_0 \times \R_{\ge 0} & \ar@{_(->}[l]_-i X_0 \ar[d]_{\pi} \\
 \R_{> 0}  \ar@{^(->}[r]^-j &  \R_{\ge 0} & \ar@{_(->}[l]_-i  \{0\}
}
$$
In this case it makes sense to ask about the microsupport
of the nearby cycles. 

Recall our notation for the cotangent sequence
$$
\xymatrix{
0 \ar[r] & \pi^*T^*\RR_{\geq 0}\ar[r] & T^*(X_0 \times \RR_{\geq 0})\ar[r]^-\Pi & T^*\pi\ar[r] & 0
}
$$
Note the natural identification $T^*\pi|_0 \simeq T^*X_0$.

\begin{lemma} \label{nearbyestimate}
For $\cF \in \sh(X_0 \times \RR_{>0})$, we have
$$
\xymatrix{
\ssupp(\psi(\cF)) \subset \overline{\Pi(\ssupp(\cF))} \cap T^*X_0
}
$$
\end{lemma}
\begin{proof}

Follows from Lemma \ref{standardestimates} and Lemma \ref{lem:addrelcotpushpull}. 
\end{proof}

\begin{definition}
Given a subset $\Lambda \subset  T^*(X_0 \times \RR_{> 0})$, 
we define its {\em nearby subset} as
$$\psi(\Lambda) := \overline{\Pi(\Lambda)} \cap T^*X_0$$
so as to write the statement of the previous lemma as
$$ss(\psi( \cF)) \subset \psi(ss(\cF))$$
\end{definition}

\section{Nearby cycles and relative external Hom} 

For sheaves $\cF \in sh(X)$ and $\cG \in sh(Y)$, we write 
$$\inthom^{\boxtimes}(\cF, \cG) := 
\inthom_{X \times Y}(\pi_1^* \cF, \pi_2^! \cG) \in sh(X \times Y)$$
Recall that when $X = Y$ and $\delta_X: X \to X \times X$ is the diagonal map, then 
$$\delta^!_X \inthom^{\boxtimes}(\cF, \cG) = \inthom_X(\cF, \cG) \in sh(X)$$

Consider now the space $M \times \R$, which we will think of a space over $\R$.  
We introduce a `relative external Hom' taking values in $\sh(M \times M \times \R)$.
If we write $\pi_1, \pi_2: M \times M \times \R \to M \times \R$ for the projections
to the first and second $M$ factors, respectively, then 
$$\inthom^{\boxtimes/\R}(\cF, \cG) := \inthom_{M \times M \times \R}(\pi_1^* \cF, \pi_2^! \cG)$$

In terms of the map 
$\delta_\R: M \times M \times \R \to M \times \R \times M \times \R$
taking the $\R$ coordinate to the diagonal, we have 
$$\inthom^{\boxtimes/\R}(\cF, \cG) = \delta_\R^! \inthom^{\boxtimes}(\cF, \cG)$$

Let $i: M \times M \to M \times M \times \R$ be the inclusion of the 0 fiber.  
In this section we are interested in the 
natural morphism
\begin{equation} \label{eq: rel box hom restricted}i^* : i^* \inthom^{\boxtimes/\R} (\cF, \cG) \to \inthom^{\boxtimes}(i^*\cF, i^*\cG) \end{equation} 

\begin{ex} \label{ex: why only nearbycycles} 
Suppose $M$ is a point.  Then we are studying sheaves on $\R$,
and in \eqref{eq: rel box hom restricted}, 
$\inthom^{\boxtimes/\R}$ is simply $\inthom_\R$ and $\inthom^\boxtimes = \inthom_{\mathrm{point}}$.  
We are asking when the map 
$$i^*: i^* \inthom_\R(\cF, \cG) \to \inthom_{\mathrm{point}}(i^* \cF, i^* \cG)$$ is an 
isomorphism.   Let us check three cases: 
\begin{enumerate}
\item
$\cF= k_0$, $\cG = k_\RR$. 
$$
\xymatrix{
 i^*\inthom_{\R}(k_0, k_\RR) \simeq k_0[-1] &  \inthom_{\mathrm{point}}(i^*  k_0, i^*  k_\R)\simeq k_0
}
$$
\item $\cF = k_{\RR_{\ge 0}}$, $\cG = k_{\RR_{\le 0}}$.  
$$
\xymatrix{
 i^*\inthom_{\R}(k_{\RR_{\ge 0}}, k_{\RR_{\le 0}}) \simeq 0 &  \inthom_{\mathrm{point}}(i^*  k_{\RR_{\ge 0}}, i^*  k_{\RR_{\le 0}})\simeq k_0
}
$$
\item \label{goodcase} $\cF = k_{\RR_{\ge 0}}$, $\cG = k_{\RR_{\ge 0}}$
$$
\xymatrix{
i^*\inthom_{\R}(k_{\RR_{\ge 0}}, k_{\RR_{\ge 0}}) \simeq k_0 &  \inthom_{\mathrm{point}}(i^*  k_{\RR_{\ge 0}}, i^*  k_{\RR_{\ge 0}})\simeq k_0 
}
$$
In fact in this case \eqref{eq: rel box hom restricted} is an isomorphism (and the same holds for any constant sheaves 
$M_{\RR_{\ge 0}}, N_{\RR_{\ge 0}}$). 
\end{enumerate}
\end{ex}

Let us write $j: \R_{> 0} \to \R$ to be the inclusion; we use the same notation for 
any maps base changed from $j$.  Generalizing Example \ref{ex: why only nearbycycles} \eqref{goodcase}, we will find: 

\begin{theorem} \label{thm: relative box hom basechange}
Let $\cF, \cG$ be sheaves on $M \times \RR_{> 0}$ which are $\RR_{> 0}$-noncharacteristic.
Assume also 
$\psi(\ssupp(\cF))$ and $\psi(\ssupp(\cG))$ are spdff.  Then the natural map 
$$i^*: i^*\inthom^{\boxtimes / \R} (j_* \cF, j_* \cG)  \to \inthom^{\boxtimes}(i^*  j_*\cF, i^* j_* \cG)$$
is an isomorphism. 
\end{theorem} 
\begin{proof}
We will study the following diagram: 

\beq\label{eq: lower square}
\xymatrix{
\ar[d]_-{i \times \id } M \times M  \ar[r]^-{i }  & (M \times M) \times \RR \ar[d]^-{\delta_{\RR}}\\
(M \times \RR) \times M \ar[r]^-{\id \times i } & (M \times \RR) \times (M \times \RR)
}
\eeq
We factor our map
into: 
\begin{eqnarray*}
i^*\inthom^{\boxtimes / \R} (j_* \cF, j_* \cG) & = & i^* \delta_\R^!  \inthom^\boxtimes (j_* \cF, j_* \cG)  \\
& \to & (i \times \mathrm{id})^! (\mathrm{id} \times i)^* \inthom^\boxtimes (j_* \cF, j_* \cG) \\
& \to &   (i \times \mathrm{id})^! \inthom^\boxtimes(j_* \cF, i^* j_* \cG) \\
& = & \inthom^\boxtimes(i^* j_* \cF, i^* j_* \cG) 
\end{eqnarray*}
Both arrows are natural base change morphisms which are not, in general, isomorphism.  The microsupport
conditions imposed here will ensure that they are in the cases at hand.  We show this for the first
arrow in Lemma \ref{lem:lower square}, and for the second arrow in Lemma \ref{cor: box hom star pullback reasonable}.  
In both cases the arguments involve computing stalks of nearby cycles in terms of sections over a neighborhood; 
we first develop in Section \ref{sec: nearby cycle stalks} criteria for when this is valid. 
\end{proof}

\subsection{Stalks of nearby cycles}  \label{sec: nearby cycle stalks}
Recall that stalks of constructible sheaves can be computed as sections over a sufficiently small open or closed ball;
per Lemma \ref{lem: reasonable stalks}, constructibility can be weakened to the condition that the microsupport is spdff.  

One may expect that, similarly, stalks of nearby cycles can be computed as sections over an appropriate cylinder. 
Here we show this is true, assuming the microsupport of the limit is spdff and the microsupport 
of the original sheaf is noncharacteristic with respect to the projection to $\R$. 

\vspace{2mm}

 For a given positive flow on $S^*W$, 
for $x\in W$ and $r>0$, let $B_r(x) \subset W$, $\ol B_r(x) \subset W$, $S_r(x) \subset W$ denote the respective
open ball, closed ball and sphere around $x$ of size $r$ in the sense of Lemma \ref{lem: positive perturbation of cosphere}.   

Consider the inclusions of cylinders
$$
\xymatrix{
c:C_{r, \delta}(x) = B_r(x) \times (-\delta, \delta) \ar[r] &  W \times \RR 
&
\ol c:\ol C_{r, \delta}(x) = \ol B_r(x) \times (-\delta, \delta) \ar[r] &  W \times \RR 
}
$$  
Note is $c$ is an open embedding.

 \begin{lemma}\label{lem: reasonable}
 Suppose 
 \begin{itemize}
 \item $\Lambda \subset T^*(W \times \R_{> 0})$ is conic and $\RR_{> 0}$-noncharacteristic
 \item  $\psi(\Lambda)^\infty \subset T^\infty(W \times 0)$ is spdff.  
 \end{itemize}

Fix $x\in W$, and a positive flow displacing $S^*_x W$ from $\psi(\Lambda)$.   Inside $T^\infty (W \times \R)$,
consider the Legendrians $\Lambda^\infty$ and the outward conormal to the cylinder, $N^+ C_{r, \delta}(x)$. 

Then there exists $r(x)>0$, and for each $0<r < r(x)$, there is a $\delta(x, r) >0$, so that for all $0<\delta<\delta(x, r)$, 

$$\Lambda^\infty \cap N^+ C_{r, \delta}(x) = \emptyset$$
\end{lemma}

\begin{proof}
Note the closure $\ol C_{r, \delta}(x)$ is a manifold with corners and its outward conormal Legendrian is a union of several  pieces: there are the two codimension one faces
$$
\xymatrix{
\partial_r = S_r(x) \times (-\delta, \delta) & \partial_\delta = B_r(x) \times \{\pm \delta\}  
}
$$ 
and the corner 
$$
\xymatrix{
\partial_{r, \delta} = S_r(x) \times \{\pm \delta\}  
}
$$ 
The assertion for $ \partial_\delta$ is immediate from the $\RR_{\not = 0}$-noncharacteristic hypothesis.

We write $\Pi: T^*(W \times \R) \to (T^*W) \times \R$ for the projection to the relative cotangent. 
For fixed $r>0$, if $\Lambda^\oo$ intersects the outward conormal along $\partial_r$, any point in the intersection defines a point in $\Pi(\Lambda)^\oo$. Similarly, by the $\RR_{\not = 0}$-noncharacteristic hypothesis,
if $\Lambda^\oo$ intersects the outward conormal along $\partial_{r, \delta}$, any point in the intersection also defines a point in $\Pi(\Lambda)^\oo$.  Taking the limit of such points as $\delta \to 0$ gives an intersection point of  $\psi(\Lambda)^\oo$ with the outward conormal of $B_r(x)$.  

But these conormals were chosen disjoint from $\psi(\Lambda)^\infty$, this having been possible because 
this locus was assumed spdff. 
\end{proof}

Consider the projections 
$$
\xymatrix{
q:C_{r, \delta}(x) = B_r(x) \times (-\delta, \delta) \ar[r] &   (-\delta, \delta)
&
\ol q:\ol C_{r, \delta}(x) = \ol B_r(x) \times (-\delta, \delta) \ar[r] &   (-\delta, \delta)
}
$$  
Note $\ol q$ is proper.

 \begin{cor}\label{cor: reasonable}
 Suppose $\cG \in \sh(W \times \R_{> 0})$ is $\RR_{> 0}$-noncharacteristic
 and $\psi(\ssupp(\cG))$ is spdff. 
 
Then for any $x\in W$, there exists $r(x)>0$ and $\delta(x, r(x))>0$ so that for all $0<r < r(x)$, $0<\delta<\delta(x, r(x)))$,  the 
following restriction maps are isomorphisms. 
$$\ol q_* \ol c^*  j_* \cG \xrightarrow{\sim} q_* c^* j_* \cG $$
$$
\xymatrix{
\Gamma(\ol C_{r, \delta}(x), j_*\cG) \ar[r]^-\sim & \Gamma(C_{r, \delta}(x), j_* \cG) \ar[r]^-\sim & (j_* \cG)_x
}
$$  
\end{cor}
\begin{proof}
Noncharacteristic propagation.
\end{proof}

\begin{remark}\label{rem: reasonable}
Note that Corollary \ref{cor: reasonable} implies the functor $q_*c^*$ commutes with standard operations: it  commutes with  $!$-pullbacks and $*$-pushforwards (since $c$ is smooth) and  $\ol q_*\ol c^*$  commutes with $!$-pushforwards and $*$-pullbacks  (since $\ol q$ is proper). 
\end{remark}

\subsection{The first arrow}
\begin{lemma} \label{lem:lower square}
Let $\cF, \cG$ be sheaves on $M \times \RR_{> 0}$, which are  $\RR_{> 0}$-noncharacteristic.
Assume that 
 $\psi(\ssupp(\cF)) \cup \psi(\ssupp(\cG))$ is pdff. 
Then the natural map 
\beq\label{eq: bc for base}
\xymatrix{
i^* \delta_\RR^!  \inthom^\boxtimes (j_* \cF, j_* \cG)  \ar[r] &  (i \times \mathrm{id})^! (\mathrm{id} \times i)^* \inthom^\boxtimes (j_* \cF, j_* \cG)
}
\eeq
is an isomorphism.  
\end{lemma}
\begin{proof}
 
 Note that the hypothesis (that  $\psi(\ssupp(\cF)) \cup \psi(\ssupp(\cG))$ is pdff)
and  the conclusion (by \cite{guillermou-kashiwara-schapira})
are invariant under contactomorphism of $S^*M$.
    Thus we may assume   $\psi(\ssupp(\cF)) \cup \psi(\ssupp(\cG))$ is spdff, and so in particular that $\psi(\ssupp(\cF))$ and $\psi(\ssupp(\cG))$ are spdff.

 We will check \eqref{eq: bc for base} is an isomorphism on the stalk at a point $(x_1, x_2)\in M \times M$. 
By Lemma~\ref{lem: reasonable stalks of prods}, we can calculate  \eqref{eq: bc for base} on  stalks at $(x_1, x_2)$ by taking sections over 
the polyball $B_r(x_1) \times B_r(x_2)$ which is the special fiber of the polycylinder $C_{r, \delta}(x_1) \times C_{r, \delta}(x_2)$. 
We use the notation from  above Lemma \ref{lem: reasonable}  for inclusion and projection
of cylinders. 

Writing this in terms of standard operations, we seek to show the induced map 
$$
\xymatrix{
(q\times q)_*(c \times c)^*i^* \delta_\RR^!  \inthom^\boxtimes (j_* \cF, j_* \cG)  \ar[r] &  
(q \times q)_*(c \times c)^*(i \times \mathrm{id})^! (\mathrm{id} \times i)^* \inthom^\boxtimes (j_* \cF, j_* \cG)
}$$
is an  isomorphism.

For the right hand side, applying  Lemma~\ref{lem: reasonable} and Corollary~\ref{cor: reasonable}, in view of 
the identities of Remark~\ref{rem: reasonable}, we find 

\begin{eqnarray*}
& \phantom{=} & (q\times q)_*(c \times c)^*(i \times \mathrm{id})^! (\mathrm{id} \times i)^* \inthom^\boxtimes (j_* \cF, j_* \cG) \\
& = & (i \times \mathrm{id})^! (\mathrm{id} \times i)^* (q\times q)_*(c \times c)^* \inthom^\boxtimes (j_* \cF, j_* \cG) \\
& = & (i \times \mathrm{id})^! (\mathrm{id} \times i)^* \inthom_{\R \times \R}^\boxtimes ( q_! c^! j_* \cF, q_* c^* j_* \cG) \\
& = & (i \times \mathrm{id})^! (\mathrm{id} \times i)^* \inthom_{\R \times \R}^\boxtimes( j_*  q_! c^! \cF,  j_* q_* c^*  \cG)
\end{eqnarray*}

For the left hand side, 
note that the analogues of Lemma~\ref{lem: reasonable}, Corollary~\ref{cor: reasonable}, and Remark~\ref{rem: reasonable} hold for
polycylinders (as opposed to just cylinders).  Using these, we find:
$$
(q\times q)_*(c \times c)^*i^* \delta_\RR^! \inthom^\boxtimes (j_* \cF, j_* \cG) =
i^* \delta_\RR^! \inthom_{\R \times \R}^\boxtimes ( j_*  q_! c^! \cF,  j_* q_* c^*  \cG)
$$

Thus we seek to show the natural map 
$$
\xymatrix{
i^* \delta_\RR^! \inthom_{\R \times \R}^\boxtimes ( j_*  q_! c^! \cF,  j_* q_* c^*  \cG) \ar[r] & 
(i \times \mathrm{id})^! (\mathrm{id} \times i)^* \inthom_{\R \times \R}^\boxtimes ( j_*  q_! c^! \cF,  j_* q_* c^*  \cG)
}$$
is an  isomorphism. 

A final application of Corollary~\ref{cor: reasonable}, to replace $q$ with the proper map $\ol q$,  shows by noncharacteristic propagation that 
the sheaves $q_!c^!\cF$, $q_*c^*\cG$ on $\R$ are locally constant near $0$. 
Thus it remains to verify the assertion in the case when $M$ is a point. This is an elementary exercise 
(essentially Example \ref{ex: why only nearbycycles} \eqref{goodcase} above). 
\end{proof}

\subsection{The second arrow}

Consider spaces $V, W$. 
We write $\pi_V$ and $\pi_W$ for the operations of projecting out the $V$ or $W$ factor in various 
products.  

\begin{lemma} \label{lem: box hom star pullback} 
Assume $V$ and $W$ are compact. 
For $\cF \in sh(V)$ and $\cG \in sh(W \times \R)$, suppose that the directed system $\cG(W \times (-\epsilon, \epsilon))$ 
is essentially constant, i.e. there exists some sequence $\epsilon_i \to 0$ such whenever $j > i$ we have
$\cG(W \times (-\epsilon_i, \epsilon_i)) \xrightarrow{\sim} \cG(W \times (-\epsilon_j, \epsilon_j))$.
Then the natural map
$$i^*: i^* \inthom^\boxtimes ( \cF,  \cG ) \to 
\inthom^\boxtimes ( \cF,  i^* \cG)$$
induces an isomorphism on global sections. 
\end{lemma} 
\begin{proof}
Taking global sections means applying $\pi_{V*}$ and $\pi_{W*}$; as $V$ and $W$ are compact, 
we may commute any pullbacks past these operations by proper base change.  After applying the natural adjunctions
$$f_* \inthom ( f^* \cdot, \cdot) = \inthom(\cdot, f_* \cdot) \qquad \qquad f_* \inthom (  \cdot, f^! \cdot) = \inthom(f_! \cdot, \cdot)$$ 
and various base changes, we are reduced to studying  
$$ i^* \inthom_{\R}(\pi_{\R}^* \pi_{V!} \cF, \pi_{W*} \cG) \to \Hom( \pi_{V!} \cF,  i^* \pi_{W*} \cG)$$
Let us write 
$\mathbf{F} = \pi_{V!} \cF$ for the relevant module and $\mathbf{F}_\R$ for the constant sheaf on $\R$ with this
stalk.  Then we are asking when the following map is an isomorphism 
$$\inthom_{\R}(\mathbf{F}_\R, \pi_{W*} \cG)_0 \to \Hom( \mathbf{F},  ( \pi_{W*} \cG)_0)$$
By hypothesis, the directed system $( \pi_{W*} \cG)(-\epsilon, \epsilon)$ is essentially constant, 
which ensures the above map is an isomorphism for any module $\mathbf{F}$. 
\end{proof} 

Recall we write $i: 0 \to \R$ for the inclusion, or anything base changed from it.  

\begin{corollary} \label{cor: box hom star pullback reasonable}
For  $\cG \in \sh(W \times \R_{> 0})$ which is $\RR_{>  0}$-noncharacteristic and such that $\psi(\ssupp(\cG))$ is spdff,  
and any $\cF \in \sh(V)$,  the natural map
$$i^*: i^* \inthom^\boxtimes ( \cF,  \cG ) \to 
\inthom^\boxtimes (\cF, i^* \cG)$$
is an isomorphism. 
\end{corollary} 
\begin{proof}
We check on stalks, by using the closed neighborhoods in Cor. \ref{cor: reasonable} to
verify the hypothesis of Lemma \ref{lem: box hom star pullback}. 
\end{proof} 

\begin{remark}
The conclusion of Cor. \ref{cor: box hom star pullback reasonable} also holds if  $\cF$ is (not only weakly) cohomologically 
constructible.  Indeed, in this case we may use Verdier duality: 
\begin{eqnarray*} 
(1_V \boxtimes i^*) \inthom^\boxtimes (\cF , \cG) 
& =  & (1_V \boxtimes i^*) (\DD \cF \boxtimes \cG) \\ 
& = & \DD \cF \boxtimes i^* \cG \\ 
& = & \inthom^\boxtimes ( \cF ,  i^* \cG)
\end{eqnarray*}
Let us explain how this compares to the present argument.
Looking back through the logic of this section, recall at some moment we were to consider
$\inthom_{\R}(\mathbf{F}_\R, \pi_{W*} \cG)_0 \to \Hom( \mathbf{F},  ( \pi_{W*} \cG)_0)$. 
Previously we argued it was an isomorphism because of some constancy of $\pi_{W*} \cG$ near zero. 
However, we could have also concluded that it was an isomorphism if $\mathbf{F}$ were a compact
object, i.e. if we had the cohomological constructibility needed for the Verdier duality argument above.  
We do {\em not} assume this: indeed, we will encounter infinite rank stalks whenever we
wish to consider the category of `wrapped' microsheaves. 
\end{remark}

\section{On the full faithfulness of nearby cycles}

Let $M$ be a  manifold, and consider the diagram:

$$
\xymatrix{
\ar[d]_{\pi} U = M \times \R_{> 0}  \ar@{^(->}[r]^-j & \ar[d]_{\pi} N = M \times \R & \ar@{_(->}[l]_-i M \ar[d]_{\pi} \\
 \R_{> 0}  \ar@{^(->}[r]^-j &  \R & \ar@{_(->}[l]_-i  \{0\}
}
$$

Consider the nearby cycle functor\footnote{
Strictly speaking, this is nearby cycles if we restrict attention to sheaves supported always over $\RR_{\ge 0}$,
which we will do in all applications.  We have chosen the present setup solely to avoid writing ``$\RR_{\ge 0}$'' in subscripts.  
It is evident that all results of Sect.~\ref{sec: nearby microsupport} 
hold in the present situation.
}
$$
\xymatrix{
\psi = i^* j_* :\sh(U) \ar[r] &  \sh(M)
}$$

Given $\cF, \cG\in \sh(U)$, there is an induced map between the following sheaves on $M$:  
\beq\label{eq: nearby hom}
\xymatrix{
\psi: \psi \inthom (\cF, \cG) \ar[r] & \inthom(\psi\cF, \psi\cG)
}
\eeq
We are interested to know when \eqref{eq: nearby hom} induces an isomorphism on global sections.  
(Note we may replace $\RR$ with any neighborhood of $0$.) 

We factor \eqref{eq: nearby hom} as:
\beq\label{eq: factored nearby hom}
\xymatrix{
i^* j_* \inthom (\cF, \cG) \ar[r]^-{j_*}  &  i^* \cH om(j_*\cF, j_*\cG)  \ar[r]^-{i^*}    & \inthom(i^*j_* \cF, i^*j_* \cG)
}
\eeq
Since $j$ is an open embedding, the counit of adjunction is an equivalence $j^* j_* \stackrel{\sim}{\to}  \id $, and hence the arrow above labelled $j_*$ is always an isomorphism.  Thus it remains to study when 
\beq\label{eq: local hom restricted}
\xymatrix{
i^*:  i^* \inthom(j_*\cF, j_*\cG)  \ar[r]    & \inthom(i^* j_* \cF, i^* j_* \cG)
}
\eeq
induces an isomorphism on global sections. 

\begin{theorem} \label{thm: gapped specialization is fully faithful}
Let $M$ be a compact manifold (without boundary).
 Let $\cF, \cG$ be sheaves on $M \times \RR_{> 0}$. 
 Assume: 
\begin{itemize}
\item  $\ssupp(\cF)$ and $\ssupp(\cG)$ are $\RR_{>0}$-noncharacteristic;
\item   $\psi(\ssupp(\cF)) \cup \psi(\ssupp(\cG))$ is pdff;
\item  The family of pairs in $S^*M$ determined by $(\ssupp_\pi(\cF), \ssupp_\pi(\cG))$ 
is gapped for some fixed contact form on $S^*M$.
\end{itemize}

Then the natural map 
$$i^*:\Gamma(M, i^*\inthom_{M \times \R}(j_* \cF, j_* \cG))  \to \Hom_{M}(i^*  j_*\cF, i^* j_* \cG)$$
is an isomorphism. 
\end{theorem}
\begin{proof}
 Note that the hypothesis (that  $\psi(\ssupp(\cF)) \cup \psi(\ssupp(\cG))$ is pdff)
and  the conclusion (by \cite{guillermou-kashiwara-schapira})
are invariant under contactomorphism of $S^*M$.
    Thus we may assume   $\psi(\ssupp(\cF)) \cup \psi(\ssupp(\cG))$ is spdff, and so in particular that $\psi(\ssupp(\cF))$ and $\psi(\ssupp(\cG))$ are spdff.

Let $\delta_M: M \to M \times M$ be the diagonal
embedding.  We write $i$ for a base-change of $i:\{ 0\} \to \RR$.    Consider the diagram:  
\beq\label{eq: key square split}
\xymatrix{
\ar[d]_-{\delta_M} M \ar[r]^-{i}  & M \times \RR \ar[d]^-{\delta_M}\\
M \times M  \ar[r]^-{i }  & (M \times M) \times \RR 
}
\eeq
and the map $\delta_\RR: (M \times M) \times \RR \to (M \times \R) \times (M \times \R)$.

Consider the sheaf 
\beq \label{eq: h sheaf} \inthom^\boxtimes (j_* \cF,  j_*\cG) 
\in \sh(M \times \R \times M \times \R) \eeq
and recall the identifications (the first by definition) 
 \begin{eqnarray*}
\inthom^{\boxtimes/\RR} (j_* \cF, j_*\cG) & = & \delta_{\R}^!  \inthom^\boxtimes (j_* \cF,  j_*\cG)  \\
\inthom(j_* \cF, j_* \cG) & = & \delta_M^! \delta_{\R}^! \inthom^\boxtimes (j_* \cF,  j_*\cG)
\end{eqnarray*}

A diagram chase shows we may factor \eqref{eq: local hom restricted} as: 

\begin{eqnarray*}
\Gamma(M, i^*\inthom_{M \times \R}(j_* \cF, j_* \cG)) & = &  \Gamma(M, i^* \delta_M^! \delta_\R^! \inthom^\boxtimes (j_* \cF,  j_*\cG)  )  \\
& \rightarrow & \Gamma(M, \delta_M^! i^*  \delta_\R^! \inthom^\boxtimes (j_* \cF,  j_*\cG)  ) \\
& = & \Gamma(M, \delta_M^! i^* \inthom^{\boxtimes/\RR} (j_* \cF,  j_*\cG) ) \\
& \xrightarrow{\sim} & \Gamma (M, \delta_M^! \inthom^\boxtimes(i^* j_* \cF, i^* j_* \cG)) \\
& = & \Hom_M(i^* j_* \cF, i^* j_* \cG)
\end{eqnarray*}
Above, the ``$\xrightarrow{\sim}$'' is from Theorem \ref{thm: relative box hom basechange}, 
whose hypotheses are included amongst those of the present theorem. 

The ``$\rightarrow$'' is the natural map associated to the square \eqref{eq: key square split}, 
and it remains only to show it is an isomorphism.  This is the content of Lemma \ref{lem:upper square} below. 
\end{proof}

For conceptual clarity, we replace $i: 0 \hookrightarrow \R$ with
the inclusion of a point in a manifold of any dimension, $i: 0 \hookrightarrow B$.  Similarly
we replace $j: \R_{>  0} \hookrightarrow \R$ with $j: (B \setminus 0) \hookrightarrow B$, etcetera. 
For the following result, we do not need any pdff conditions.

\begin{lemma} \label{lem:upper square}
Let $\cF, \cG \in sh(M \times (B\setminus 0))$ be $(B \setminus 0)$-noncharacteristic,
and assume the pair $(ss_{\pi}(\cF), ss_{\pi}(\cG))$ is gapped (Def. \ref{def:gapped}) over 
$B \setminus 0$.  Then the natural map 
$$\Gamma(M, i^* \delta_M^! (\delta_B^! \inthom^\boxtimes (j_* \cF,  j_*\cG))) \to \Gamma(M, \delta_M^! i^*  (\delta_B^! \inthom^\boxtimes (j_* \cF,  j_*\cG)))$$
associated to
\beq\label{eq: top key square}
\xymatrix{
\ar[d]_-{\delta_M} M \ar[r]^-{i}  & M \times B \ar[d]^-{\delta_M}\\
 M \times M  \ar[r]^-{i }  & (M \times M) \times B
}
\eeq
is an isomorphism. 
\end{lemma}
\begin{proof}
We will apply Proposition \ref{prop:hyperbolic basechange} to the sheaf $\delta_B^! \inthom^\boxtimes (j_* \cF,  j_*\cG)$.  
In the notation there we should consider 
$$X = (M \times M) \times B \qquad \qquad Y_1 = M \times B \qquad \qquad Y_2 = M \times M \qquad \qquad Y = M$$

Let $m: T^\circ M  \to \R$ be the  
$\R_{>0}$-equivariant Hamiltonian for the positive flow exhibiting gappedness (recall this Hamiltonian is assumed time independent).  
Fix a metric on $B$ and let $b$ be the $\R_{>0}$-equivariant Hamiltonian for Reeb flow on $T^*B \setminus B$
(i.e. the norm of the cotangent coordinate).  On $T^\circ(M \times M \times B)$ we consider the Hamiltonian
$h := (m^2 \oplus m^2 \oplus b^2)^{1/2}$;  it is $\R_{>0}$-equivariant so determines a contact flow on $S^*(M \times M \times B)$. 
With respect to this flow,  let $f_1$ and $f_2$ be neighborhood defining functions for $M \times B$ and $M \times M$ respectively. 
 Note in what follows our concerns will fall within Remark~\ref{rem:away from corners} so the constructions are suitably smooth.

To verify the hypotheses of Proposition \ref{prop:hyperbolic basechange}, we must show  that for some 
$\epsilon$,  above $\{f_1, f_2 \in (0, \epsilon)\}$, the locus $ss( \delta_B^! \inthom^\boxtimes (j_* \cF,  j_*\cG) )$ 
is disjoint from the span of $df_1, df_2$.  Note the condition $f_1, f_2 \ne 0$ implies that 
we are working in the complement of the preimage of $0 \in B$. 

Evidently $df_2$ is pulled back from the cotangent to $B$.  Thus the disjointness from $df_2$ follows because 
$\cF, \cG$ are $(B \setminus 0)$-noncharacteristic.

Because $df_2$ is contained in $T^*B$, 
it now suffices to check that the images in  $T^*X / T^*B$
of $ss(\delta_B^! \inthom^\boxtimes (j_* \cF,  j_*\cG))$ and $df_1$ are disjoint in some neighborhood $f_1^{-1}(0, \epsilon)$. 
That is, we should study the relative microsupport $ss_{\tilde{\pi}}(\delta_B^! \inthom^\boxtimes (j_* \cF,  j_*\cG))$.  

The relevant 
estimate is Cor. \ref{cor:relativehomestimate}, which tells us that in the fiber over $0 \in B$, i.e. inside $(T^*X / T^*B)_0 = T^*(M \times M)$, we have 

$$ss_{\tilde{\pi}}(\delta_B^! \inthom^\boxtimes (j_* \cF,  j_*\cG))_0 \subset - ss_{\pi}(\cF)_0  \boxplus  ss_{\pi}(\cG)_0$$

We therefore study intersections of $df_1$ with the RHS above. 
Per Lemma \ref{lem:three ways of looking at a reeb chord}, these correspond to chords 
whose length is the value of $f_1$ at the intersection point. 
From the gapped hypothesis, we may once and for all choose $\epsilon$ small enough that there are no such in $f_1^{-1}(0, \epsilon)$. 
\end{proof}

\section{Microsheaves} \label{sec: mush} 

Let $M$ be a manifold, and $\cC$ a symmetric  monoidal $\infty$-category.
 (The reader will not learn less from this article by taking $\mathcal{C}$ to be the dg derived category of modules 
over a commutative ring, say $\mathbb{Z}$.) When we wish to make assertions regarding presentability
or compact generation, we assume that $\mathcal{C}$ has these properties. 
   
  The category of sheaves on $M$ valued in $\cC$ {\em microlocalizes} over the cotangent
bundle $T^*M$ in the sense that it is the global sections of a sheaf of categories on $T^*M$.  
This sheaf of categories is defined as follows.  Recall that for $V \subset T^*M$, we write 
$sh_V(M)$ for the category of sheaves on $M$ microsupported in $V$.  For an open subset 
$\mathcal{U} \subset T^*M$, we set
$$
\mush^{pre}(\mathcal{U}) := \sh(M) / \sh_{M \setminus \mathcal{U}}(M)
$$

For  $\mathcal{U} \subset V$, we evidently have  
$\sh_{M \setminus \mathcal{U}}(M) \supset \sh_{M \setminus V}(M)$, and 
thus there are restriction maps $\mush^{pre}(\mathcal{U}) \to \mush^{pre}(V)$; 
it is easy to see that these make $\mush^{pre}$ into a presheaf of  stable  $\infty$-categories.  
We write $\mush$ for its sheafification (see Rem.~\ref{rem:sheafification for cats} for a discussion of this construction
in the present context).

While $\mush$ is sensible in the usual topology on $T^*M$, it is in fact pulled back from the conic
topology, in which the open sets are all invariant under the $\R_{>0}$-action.  We will often be interested
in its restriction to the the complement of the zero section $T^*M \setminus M$, where it is 
pulled back from a sheaf on the cosphere bundle $(T^*M \setminus M)/\R_{> 0} = S^*M$, which we 
also denote by $\mush$. 

For $F \in \mush(\mathcal{U})$, there is a well defined microsupport, $ss(F) \subset \mathcal{U}$.   For $\Lambda \subset M$ 
we write $\mush_\Lambda(\mathcal{U})$ for the full subcategory of objects microsupported in $\Lambda \cap \mathcal{U}$.
Note that 
$\mush_\Lambda$ is a subsheaf of $\mush$, and the pushforward of a sheaf on $\Lambda$. 

\begin{remark}\label{rem:sheafification for cats}
Let us give some technical remarks regarding sheaves of $\infty$-categories, in particular sheafification; 
throughout, we rely  upon the foundations provided by~\cite{luriehttpub, lurieha}.  To ease the exposition, 
when possible, we will say category in place of $\infty$-category.

When discussing categories, it is necessary
to specify in which category of categories we are working; specific natural choices include the category $Cat$ of all
categories, and also the categories $Pr^L$ and $Pr^R$ of presentable categories with continuous or cocontinuous morphisms. 
To determine whether a presheaf is a sheaf, this is immaterial: limits in $Pr^L$ or $Pr^R$ exist
and are computed by the corresponding limits in the category of categories \cite[Chap. 5.5]{luriehttpub}.  
However, 
given a presheaf of presentable categories for which all restriction maps are continuous and cocontinuous,
we do not know whether (or when) its sheafifications
in $Pr^L$ and $Pr^R$ and $Cat$ agree.  

In fact this is the situation in which we find ourselves.  
As the microsupport of a limit or colimit is contained in the union of the microsupports of the 
terms, the full subcategory $\sh_{M \setminus \mathcal{U}}(M) \subset \sh(M)$ is closed under limits and colimits.  
It follows (see e.g. \cite[Prop A.8.20, Rem A.8.19]{lurieha}) that the quotient map 
$\sh(M) \to \sh(M) / \sh_{M \setminus \mathcal{U}}(M)$ is continuous and cocontinuous.  Similarly, all restriction
maps in $\mush^{pre}$ are continuous and cocontinuous.  Thus it would be natural to try and work in
any of the above categories. 

In fact we always sheafify in $Cat$.  The reason is that, because $Cat$ embeds in simplicial spaces, we may
import from that latter category the fact that a morphism of sheaves is an isomorphism iff it is an isomorphism
on stalks.  This is {\em false} for sheaves in $Pr^L$ or $Pr^R$.  In particular,  colimits
in these categories can be much smaller than in $Cat$.
For example, let us write $sh_M$ for the sheaf of categories $U \mapsto sh(U)$ on some manifold $M$.  The restriction maps are continuous and co-continuous; this is a sheaf of categories in $Cat$, $Pr^L$, and $Pr^R$. To compute the stalk in $Pr^L$ or $Pr^R$ we may take $\colim_{U \ni x} sh(U) = \lim_{U \ni x} sh(U)$, where the morphisms in the limit diagram are the right or left adjoints or the restriction, i.e. the shriek or star pushforward.  In either case, the limit consists only of the skyscraper sheaves at $x$.  (By contrast, the stalk in the category of categories cannot similarly be computed by trading colimits for limits, and is much larger.)  Moreover, the same calculation is valid for any separated topological space.  So if $M^\delta$ is $M$ with the discrete topology, then the map $sh_{M^\delta} \to \sh_M$ (induced by pullback along $M^\delta \to M$) is a morphism of sheaves of categories which is an isomorphism on $Pr^L$ or $Pr^R$ stalks.\footnote{This paragraph is extracted from discussions
with Chris Kuo, German Stefenich, and  Nick Rozenblyum, who corrected some of our previous misconceptions.} 

We do not know whether, for $\mush^{pre}$, the sheafification in $Cat$ agrees with the sheafification in
$Pr^L$ or $Pr^R$; we also do not know whether the sections of the (sheafified in Cat) $\mush$ are in fact
presentable categories.  One reason to care about this is that the existence of adjoints to the restriction
maps (which are continuous and cocontinuous) would follow from presentability.  

That being said, there is a trivially sufficient condition which will guarantee that the sheafification of a presheaf 
of stable categories agrees in all three categories: if, at each point, there is a cofinal sequence in the 
neighborhood basis on which the restriction maps are isomorphisms.  It is not difficult to see that this holds for
subcategories $\mush_\Lambda^{pre}$ of microsheaves microsupported on some fixed sufficiently tame
Lagrangian $\Lambda$. 

Finally, one uncomplicated aspect of sheafification is that since all categories in sight are stable and all functors exact, the sheafification is automatically a sheaf of stable categories. 
\end{remark}

Next, let us recall some properties of $\mush$ available in the literature.
While $\mush$ is not explicitly considered in \cite{kashiwara-schapira},  $\mush^{pre}$ appears in \cite[Sec. 6.1]{kashiwara-schapira},
where $\mush^{pre}(\mathcal{U})$ (with boundedness assumptions) is called $D^b(M; \mathcal{U})$. The stalks
of $\mush^{pre}$, which are also the stalks of $\mush$, also appear  (again with boundedness assumptions) in \cite{kashiwara-schapira} under the name $D^b(M, p)$.  
In constructing a morphism of sheaves,
it is enough to do so for their corresponding presheaves; to check properties of a morphism (in particular, when
it is fully faithful or an isomorphism) it is enough to check on stalks.  Thus the results
of \cite{kashiwara-schapira} serve well for these purposes.  
Most fundamentally, the $\muhom$ functor of \cite{kashiwara-schapira}
gives the sheaf of Homs of objects in $\mush$.  (Indeed, in \cite[Theorem 6.1.2]{kashiwara-schapira}, it is shown that there is an
isomorphism on stalks, and the proof first constructs a natural morphism on presheaves.) 
The map $\mush^{pre}(\mathcal{U}) \to \mush(\mathcal{U})$ is not generally an isomorphism, but in case $\mathcal{U} = T^*U$ one has
$sh(U) = \mu  sh^{pre}(T^*U) = \mu sh(T^*U)$.\footnote{Indeed, let $\pi: T^* U \to U$ be the projection, and $s: U \to T^* U$.  As for any conic sheaf, one has an equivalence $s^* \mu sh \cong \pi_* \mu sh$.  It thus suffices to check that $sh \to s^*\mu sh$ is an equivalence, which we can moreover check on stalks.  But the stalks of $s^* \mu sh$, which are the stalks of $\mu sh$ along the zero section, which are the stalks of $\mu sh^{pre}$ along the zero section, which are the stalks of $\sh$.}

A basic tool to study $\mush$ is the microlocal theory of quantized contact transformations developed 
in \cite[Chap. 7]{kashiwara-schapira}.  A key result is that a contactomorphism induces local isomorphisms on $\mush$. 
More precisely, a contactomorphism $\phi$ between a germ of $x \in S^*M$ and of $y \in S^*N$ induces 
an isomorphism,\footnote{The isomorphism is unique up to a choice of invertible object in the coefficient category.}
respecting microsupports, of $\mush_x$ and $\phi_* \mush_y$ \cite[Cor 7.2.2]{kashiwara-schapira}.  

In order to globalize this, we note the evident: 

\begin{lemma} \label{lem: obvious constancy}
Suppose $\widetilde{\Lambda} \subset T^*(M \times \R^n)$ is the product of $\Lambda \subset T^*M$ and 
$T_0^*\RR^n$.  Then pullback along $i: M \times 0 \to M \times \RR^n$ induces an isomorphism
$i^*: i^*\mush_{\widetilde{\Lambda}} \to \mush_{\Lambda}$, and pullback along $\pi: M \times \RR^n \to M$ 
induces an isomorphism $\pi^*: \mush_{\widetilde{\Lambda}} \to \pi^*\mush_{\Lambda}$.  $\square$
\end{lemma} 

By contact transformation we have: 

\begin{lemma} \label{lem: relative constancy} 
Suppose the germ of $\widetilde{\Lambda} \subset S^*M$ is contactomorphic
to the germ of 
$\Lambda \times N \subset \mathcal{U} \times T^*N$ for some contact $\mathcal{U}$, by 
a map restricting to $f:  \widetilde{\Lambda}  \cong \Lambda \times N$.  

Let $\pi: \Lambda \times N \to \Lambda$ be the projection. 
Then for $\lambda \in \widetilde{\Lambda}$, there is an isomorphism 
$$\mush_{\widetilde{\Lambda}}|_{\lambda} \cong \mush_\Lambda|_{\pi \circ f(\lambda)}$$
\end{lemma}
\begin{proof}
The statement is local on $\widetilde{\Lambda}$, so by contact transformation we are reduced
to Lemma~\ref{lem: obvious constancy}. 
\end{proof}

That is, $\mush_{\widetilde{\Lambda}}$ is locally constant in the $N$ direction.  In particular,
by taking $\Lambda$ a point, one has: 

\begin{corollary} \label{cor: locally local systems} 
Let $X \subset S^*M$ be a smooth Legendrian.  Then $\mush_X$ is locally isomorphic to the category of 
local systems on $X$. 
\end{corollary}

That is, $\mush_X$ is a sheaf of categories of twisted local systems.  The twistings are related to the Maslov
obstruction and similar homotopical considerations, and are studied in \cite{guillermou, jin-treumann, jin-BO, jin-J}.  We will also consider them below. 

We now consider contact isotopies.  Recall that a hamilonian isotopy $\phi_t$ on $S^*M$ determines 
a Lagrangian $\Phi \subset T^*M \times T^*M \times T^*\R$.  Using $\Phi$ as a correspondence gives
a map (we also denote it $\Phi$) 
from subset of $T^*M$ to subsets of $T^*M \times T^* \R$, and $\Phi$ is characterized by the property
$\phi_t(X)$ is the symplectic reduction of $\Phi(X)$ over $t \in \R$.  

\begin{definition}
For a contact isotopy on $S^*M$, we obtain similarly a map $\Phi$
from subsets of $S^*M$ to subsets of $S^*M \times T^* \R$ (e.g. by viewing it as a conic Hamiltonian isotopy).   
For $X \subset S^*M$ we term $\Phi(X) \subset S^*M \times T^* \R \subset S^*(M \times \R)$ 
as the {\em contact movie} of $X$.   
\end{definition}

\begin{lemma} \label{lem: isotopy constancy}
 Let $\phi_t$ be a contact isotopy on $S^*M$, and $\Lambda \subset S^*M$ any subset.  
Then there is a canonical isomorphism $\phi_{t*} \mush_{\Lambda} \cong \mush_{\phi_{t}(\Lambda)}$.  
\end{lemma}
\begin{proof}
Note there is a contactomorphism 
$(\Phi(X), Nbd(\Phi(X))) \cong (X \times \R, Nbd(X) \times T^*\R)$, so we may apply Lemma \ref{lem: relative constancy}
to conclude that $\mush_{\Phi(\Lambda)}$ is constant
in the $\R$ direction.  As $\Phi(\Lambda)$ 
is noncharacteristic for the inclusion of any $M \times t$, pullback along such an inclusion 
induces an isomorphism $\mush_{\Phi(\Lambda)}|_{M \times t} \to \mush_{\phi_t(\Lambda)}$. 
\end{proof}

We recall the stronger result:

\begin{theorem} \cite{guillermou-kashiwara-schapira} For any contact isotopy $\phi_t$ of $S^*M$ 
there is a unique sheaf $K_{\Phi} \in sh_{\Phi}(M \times M \times \R)$ such that $K_{\Phi}|_{M \times M \times 0}$ 
is the constant sheaf on the diagonal.  
\end{theorem}

The relation of this theorem to the above discussion is that $K_{\Phi}|_{M \times M \times t}$ gives an 
integral kernel which on microsupports away from the zero section applies the contact transformation
$\phi_t$.  Its real strength has to do with the fact
that one obtains equivalences of categories of sheaves, not just (as in Lemma \ref{lem: isotopy constancy}) 
microsheaves away from the zero section.  

As this result will be important to us, let us sketch the proof.  Uniqueness can be seen as follows: consider the functor
$ sh_{\Phi}(M \times M \times \R) \to \mush_{\Phi}(\Phi)$.  Note the latter is a category of local systems. 
Thus if two candidate $K_{\Phi}$ are isomorphic at $M \times 0$, hence microlocally isomorphic 
along $\Phi|_0$, they must be microlocally isomorphic everywhere away from the zero section.  
Thus the cone between them is a local system; as it vanishes at $M \times 0$ it must be trivial.   To show
existence it suffices to show existence for small positive and negative isotopies (and then convolve the corresponding kernels). 
For a small positive or negative isotopy, arguing as Lemma \ref{lem: positive perturbation of cosphere} shows that
the symplectic reduction of $\Phi$ at $t$ (for $t$ small) is a conormal to the boundary of a neighborhood of the diagonal.  
The corresponding $K_{\Phi}$ can be taken as the constant sheaf on this 
(open or closed according as the isotopy is positive or negative) neighborhood. 

We will mainly use this result through its following consequence:

\begin{corollary}  \label{cor: gks constancy}  
Let $\eta: M \times \R \to \R$ be the projection.  
Fix a contact isotopy $\phi_t: S^*M \to S^*M$ and any $X \subset S^*M$.   
Then the sheaf of categories
(on $\R$)  given by $\eta_* \sh_{\Phi(X)}$ is locally constant.  Pullback to $t \in \R$ induces an equivalence 
$(\eta_* \sh_{\Phi(X)})_t \cong \sh_{\phi_t(X)}(M)$. 
\end{corollary}


\section{Antimicrolocalization} \label{sec:anti}

Because $\mush$ is a quotient of sheaves of categories, hence suffers in its definition a {\em sheafification}, 
it is nontrivial to compute in $\mush$ directly.  In particular, for $X \subset S^*M$, 
it is not generally true that the map $\sh_X(M) \to \mush_X(X)$ is a quotient.  
One can often nevertheless reduce problems of microsheaf theory to problems of sheaf theory 
by finding some larger 
$X' \subset S^*M$ for which the natural map $\sh_{X'}(M) \to \mush_X(X)$ has a 
right inverse.  We term such an inverse an {\em antimicrolocalization}.  

When $X$ projects finitely to $M$, a local version of this problem can be solved directly using the  ``refined microlocal cutoff" 
of \cite{kashiwara-schapira}; we give an account in Section \ref{sec:cutoff}; similar results
can be found in \cite{waschkies-microperverse, guillermou}. 

When $\Lambda$ is a smooth Legendrian in a jet bundle, $\Lambda \subset J^1 M \subset T^*(M \times \R)$, 
Viterbo observed that $\Lambda \cup \Lambda^{\epsilon}$ (the latter being a small Reeb pushoff) 
should provide an antimicrolocalization \cite{viterbo-notes, viterbo-sheaf}.  
His argument was Floer-theoretic: there is an (exact) Lagrangian 
$L \cong \Lambda \times \R$ with $\partial L = \Lambda \cup \Lambda^{\epsilon}$; now $\mush_L(L)$ is 
local systems on $L$, and the map $\mush_{L} \to sh_{\Lambda \cup \Lambda^{\epsilon}}(M \times \R)$ is
obtained by sending a given local system $\mathcal{L}$ to the sheaf organizing
the Floer theory of $(L, \mathcal{L})$ with cotangent fibers.  

A direct sheaf-theoretical construction was later given by Guillermou \cite{guillermou}.  He proceeds by
making a local construction which he then
proves glues.  A similar construction will work for singular $\Lambda$ in a jet bundle, 
but a difficulty of adapting this to the case of an arbitrary positive flow is that aside from the Reeb flow in 
the jet bundle, it is not clear how to construct a cover on the base $M$ compatible with the effect of the contact flow on $S^*M$.
Instead we formulate Guillermou's construction in a global manner, where the needed local properties are 
checked by the (naturally contact invariant) noncharacteristic propagation principle.  
In this way we prove an analogous result for a singular Legendrian $\Lambda$ in an arbitrary cosphere bundle displaced 
by an arbitrary positive flow, and a generalization to the case when $\Lambda$ is  only locally closed.   

\begin{remark}
The fact that antimicrolocalization should exist in this generality is partially motivated by the results on ``stop doubling"
of \cite[Ex. 8.6]{gpsdescent}.  Indeed, these results show that given a Weinstein manifold $W$ with skeleton $\Lambda$, 
and an embedding as an exact hypersurface $W \to S^*M$, then there is a fully faithful functor 
$Fuk(W) \hookrightarrow Fuk(T^*M; \Lambda \cup \Lambda^{\epsilon})$.  In \cite{gpsconstructible} it is shown
that $Fuk(T^*M; \Lambda \cup \Lambda^{\epsilon}) \cong sh_{\Lambda \cup \Lambda^{\epsilon}}(M)$, and one 
could imagine running an analogue of Viterbo's construction above -- if one knew that 
$\mush(\Lambda) \cong Fuk(W)$.   In fact \cite{gpsconstructible} uses this idea in the reverse direction, applying
the antimicrolocalization in order to reduce the general problem of showing $\mush(\Lambda) \cong Fuk(W)$ to the case
of cotangent bundles.  This result was originally conjectured in 
\cite{nadler-wrapped}.
\end{remark}

\begin{remark}
Note that antimicrolocalization is the special case of gapped specialization corresponding to taking the natural 
$(\Lambda \times \R) \subset T^*M$ whose boundary is $\Lambda \cup \Lambda^{\epsilon}$, and flowing down
by the Liouville flow. 
\end{remark}

\subsection{Local antimicrolocalization} \label{sec:cutoff}

\newcommand\image{\operatorname{im}}

Let us recall from \cite{kashiwara-schapira} the refined microlocal cutoff. 
The assertions are local so from the start we will work with $X= \R^n$ a vector space and focus on  the origin $x_0 = 0$.  Let $X^* \simeq \R^n$ denote the dual vector space so that $T^*X  \simeq X \times X^*$.
 
\begin{proposition} (\cite[Proposition 6.1.4]{kashiwara-schapira})
 Let $K \subset X^*$  be a proper closed convex cone, and $U \subset K$ an open cone. Fix  $\cF \in sh(X)$ and $W\subset X^*$ a conic neighborhood 
of $K \cap \ssupp(\cF)|_0 \setminus 0$. Then there exists $\cF'  \in sh(X)$ and a map $u:\cF'\to \cF$ along with
a neighborhood $B\subset X$  of $0 \in X$
such that
\begin{enumerate}
\item $\ssupp(\cF')|_B \subset B \times \overline{U}$ and $\ssupp(\cF')|_0 \subset W \cup 0$.
\item $u$ induces an isomorphism in $\mush^{pre}(B \times U)$.\footnote{In \cite{kashiwara-schapira} it is stated 
``induces an isomorphism on $U$''.  However from the proof, it is clear that the present assertion is what is meant.}  
\end{enumerate}
\end{proposition}

For further developments of the microlocal cutoff, in particular to treat the complex setting, see also~\cite{dagnolo-cutoff, waschkies-microperverse}.

\begin{corollary}
In the situation above, after possibly shrinking $B$, we have 
\begin{enumerate}
\item[($1'$)] $\ssupp(\cF')|_B \subset B \times (W \cup 0)$.
\end{enumerate}
\end{corollary}
\begin{proof}
 If not, there exists a sequence $(x_i, \xi_i) \in \ssupp(\cF')$, with $x_i \to 0$, but $\xi_i \not \in W \cup 0$. 
Since $W$ and $\ssupp(\cF')$ are conic, we may assume $|\xi_i| = 1$ and the sequence converges $\xi_i \to \xi_\infty$.
Since $\ssupp(\cF')$ is closed, we have $(0, \xi_\infty) \in \ssupp(\cF')|_0 \subset W \cup 0$ with $ |\xi_\infty| = 1$,
hence $\xi_\infty \in W$. 
But $W$ is open so ultimately $\xi_i \in W$.
\end{proof}

The proof given in \cite[Proposition 6.1.4]{kashiwara-schapira} involves $\cF \in \sh(X)$ only through its microsupport $\ssupp(\cF) \subset T^*X$, and the sheaf $\cF'$ is constructed functorially.  In addition the cutoff functor makes sense for any $\cF$; only its properties depend
on the asserted condition on the microsupport of $\cF$.  In other words what their argument actually shows is:

\begin{proposition} \label{prop: functorial cutoff}
Suppose given
\begin{itemize}
 \item $K \subset X^*$  a proper closed convex cone
 \item $U \subset K$ an open cone
 \item $\Lambda \subset T^*X$ a closed conic subset 
 \item $W\subset X^*$ a conic neighborhood of $K \cap \Lambda|_0 \setminus 0$. 
\end{itemize}
Then there is a map $\phi: sh(X) \to sh(X)$ with a natural transformation $u: \phi \to id$, 
such that for a small enough neighborhood $B$ of $0 \in X$
\begin{enumerate}
\item For any $\cF$, $ss(\phi(\cF))|_B \subset  B \times (W \cup 0)$.
\item If $ss(\cF) \cap (B \times U) = \Lambda \cap (B \times U)$, then $u: \phi(\cF) \to \cF$ induces an isomorphism in $\mush^{pre}(B \times U)$. 
\end{enumerate}
\end{proposition}

\begin{corollary} \label{cor: cutoff when W is in U} 
In the situation of Proposition \ref{prop: functorial cutoff}, suppose in addition $W \subset U$.  Then 
for any $0 \in A \subset X$, and for any small enough $B \ni 0$, the microlocal cutoff induces a functor 
$$
 \xymatrix{
 \phi:\mush_\Lambda^{pre}(A \times U) \ar[r] &  sh_{\Lambda}(B)/sh_{T_B^*B}(B) 
 }
 $$
 commuting with the natural projections on both sides to $\mush_\Lambda^{pre}(B \times U)$. 
\end{corollary}
\begin{proof}
We apply Proposition \ref{prop: functorial cutoff} with $X$ replaced by its subset $A$.   
As promised in (1) of said proposition, for any $\cF$,  
we have 
$$ss^\infty(\phi(\cF))|_B \subset B \times W^\infty \subset B \times U^\infty$$
If in addition $ss^\infty(\cF) \cap (B \times U^\infty) = \Lambda \cap (B \times U^\infty),$ then per property (2),
$$ss^\infty(\phi(\cF)) \cap (B \times U^\infty) = ss^\infty(\cF) \cap (B \times U^\infty).$$

In particular if $ss^\infty(\cF) \cap (A \times U^\infty) = \emptyset$, then $ss^\infty (\phi(\cF))|_B = \emptyset$, so $\phi$ factors through
$\mush^{pre}(A \times U)$ as stated.  Additionally if 
$ss^\infty(\cF) \cap (A \times W) \subset \Lambda$ then  $ss^\infty(\phi(\cF))|_B \subset B \times \Lambda$, giving the 
microsupport condition on the image. 

Finally, the assertion that $\phi$ commutes with the projections follows from the fact that $u: \phi \to id$ induces an equivalence
on $B \times U$. 
\end{proof}
 
Now let us use this to prove the following. (See \cite[Sect.~5.1]{waschkies-microperverse} for similar arguments in the complex setting.)

\begin{lemma} \label{lem: local antimicrolocalization}
Fix a closed subset $\Lambda \subset S^*X$. Suppose for a point $x \in X$,   the fiber $\Lambda|_x =  \Lambda \cap S^*_x X$ is a single point $\lambda$. Then 
the natural map of sheaves of categories
$$
\xymatrix{
q:sh_{\Lambda}/sh_{T_X^*X}\ar[r] &  \pi_* \mush_\Lambda 
}$$ 
is an isomorphism at $x$. 
\end{lemma}

\begin{proof}
The assertion is local so we may assume  $X= \R^n$, $x = 0$. Thus $T^*X \simeq X \times X^*$ and $S^* X = X \times X^\oo$
where $X^*$ is the dual of $X$, and $X^\oo = (X^* \setminus 0)/\R_{>0}$ is its projectivization.

We will apply Proposition \ref{prop: functorial cutoff} with $U = W \subset X^*$ an open  convex conic neighborhood of 
 the coray $\R_{>0} \cdot \lambda \subset X^*$, and $K = \ol U  = \ol W \subset X^*$ the closure.
 Note we are able to make these choices precisely due to the assumption that  
  $\Lambda|_x = \lambda$.  
  
For any open neighborhood $A \subset X$ of $0\in X$, and small enough   
 open neighborhood $B \subset X$ of $0\in X$, we have from Cor. \ref{cor: cutoff when W is in U}
 a functor 
 $$
 \xymatrix{
 \phi:\mush_\Lambda^{pre}(A \times U) \ar[r] &  sh_{\Lambda}(B)/sh_{T_X^*X}(B) 
 }
 $$
 
The rest of the argument follows the complex setting detailed in~\cite[Sect.~5.1]{waschkies-microperverse}. 
Next, 
 as in \cite[Lemma 5.1.4]{waschkies-microperverse}, 
 one observes the composed functor
$$
 \xymatrix{
 \mush_\Lambda^{pre}(A \times U) \ar[r]^-\phi &  sh_{\Lambda}(B)/sh_{T_X^*X}(B) \ar[r] &  (sh_{\Lambda}/sh_{T_X^*X})|_x
 }
 $$
factors through the natural map
$$
 \xymatrix{
 \mush_\Lambda^{pre}(A \times U) \ar[r] &  
 \pi_*\mush_\Lambda^{pre}|_x
 }
 $$
 For this, let us write $\phi_{U, A}$ to convey its dependence on the choices.  Given smaller $U' \subset U$, $A' \subset A$, by construction there is a natural map 
 $\phi_{U', A'} \to \phi_{U, A}$ which induces an isomorphism in $(sh_{\Lambda}/sh_{T_X^*X})|_x$  since its cone  must have microsupport in the zero-section.
 
 Finally, one checks the resulting map
 $$
 \xymatrix{
\phi_x:  \pi_*\mush_\Lambda^{pre}|_x \ar[r] &  (sh_{\Lambda}/sh_{T_X^*X})|_x
 }
 $$
 is an inverse to $q_x$. First, we have $q_x\circ \phi_x \simeq \id$ thanks to property (2). Second, we have $ \phi_x \circ q_x \simeq \id$ since the cone of the natural morphism relating them must have microsupport in the zero-section.  
\end{proof}

\begin{remark}
The $\Lambda$ to which this is applied in practice will typically have enough tameness that 
$\mush_\Lambda^{pre}|_x$ is computed by some particular $\mush_\Lambda^{pre}(B \times U)$, i.e. one need not pass
to the limit. 
\end{remark}

Now let us use the lemma to prove the following generalization.

\begin{prop} \label{prop: local antimicrolocalization}
Fix a closed subset $\Lambda \subset S^*X$. Suppose for a point $x \in X$,   the fiber $\Lambda|_x =  \Lambda \cap S^*_x X$ is a finite set of points $\lambda_1, \ldots, \lambda_k$. Then 
the natural map of sheaves of categories
$$
\xymatrix{
q:sh_{\Lambda}/sh_{T_X^*X}\ar[r] &  \pi_* \mush_\Lambda 
}$$ 
is an isomorphism at $x$. 
\end{prop}

\begin{proof}
Let $\Lambda_1, \ldots, \Lambda_k \subset \Lambda$ be the connected components through  
the respective points  $\lambda_1, \ldots, \lambda_k \in \Lambda|_x$. Note the evident direct sum decomposition
$$
\xymatrix{
\pi_* \mush_\Lambda|_x \simeq \oplus_{i = 1}^k \pi_* \mush_{\Lambda_i}|_x
}
$$
Thus by the prior lemma, it suffices to show the natural map
$$
\xymatrix{
\oplus_{i = 1}^k (sh_{\Lambda_i}/sh_{T_X^*X}|_x)   \ar[r] & sh_{\Lambda}/sh_{T_X^*X} |_x
}
$$
is an equivalence.

Set $\Lambda' = \cup_{i = 1}^{k-1} \Lambda_i$, $\Lambda''= \Lambda_k$. By induction, it suffices to show
the natural map
$$
\xymatrix{
i' \oplus i'':sh_{\Lambda'}/sh_{T_X^*X}|_x \oplus sh_{\Lambda''}/sh_{T_X^*X}|_x   \ar[r] & sh_{\Lambda}/sh_{T_X^*X} |_x
}
$$
is an equivalence. 

Note that $i'$ (resp. $i''$) admits a right adjoint $\phi'$ (resp. $\phi''$) given by the microlocal cutoff construction  such that the unit $\id \to \phi' i'$ (resp. $\id \to \phi'' i''$ is an equivalence, so $i'$ (resp. $i''$) is fully faithful, and    
the counit $i'\phi' \to \id $ (resp. $i'' \phi'' \to \id$) is the canonical map of the microlocal cutoff construction.
Note also $\image(i') = \ker(\phi''), \image(i'') = \ker(\phi')$.

Now
we can see the images of $i', i''$ are orthogonal to each other, for example
$$
\xymatrix{
\Hom(i' \cF' , i''\cF'') \simeq \Hom(\cF' , \phi'i''\cF'') \simeq 0
}
$$
since $ \phi'i''\cF'' \simeq 0$.

Finally, it remains to show $i' \oplus i''$ is essentially surjective. For $\cF\in sh_{\Lambda}/sh_{T_X^*X} |_x$, it suffices to show the canonical map
$$
\xymatrix{
 i'\phi'   \cF \oplus  i''\phi''  \cF\ar[r] &\cF
}
$$
is an equivalence. Let $\cK$ be the cone. Then $\phi' \cK \simeq  0 \simeq \phi'' \cK$ so the microsupport of $\cK$ lies in $T^*_X X$, hence $\cK$ is trivial in $ sh_{\Lambda}/sh_{T_X^*X} |_x$.
\end{proof}

\begin{definition} \label{def: finite position}
We say a closed subset $\Lambda \subset S^*X$ is in {\em finite position} over $x \in X$ if it satisfies the hypothesis of Proposition \ref{prop: local antimicrolocalization} over $x$.  When this is true at all $x$, we say $\Lambda$ is in {\em finite position}.  
If $\Lambda$ may be perturbed to have finite projection by a contact isotopy, we say $\Lambda$ is {\em perturbable to finite
position.}  

Finally, for $\Lambda$ a closed subset equipped with the germ of an embedding into a contact manifold, 
we say $\Lambda$ is {\em perturbable to finite position}
if every codimension zero contact embedding of $\Lambda$ into a cosphere bundle is perturbable to finite position. 
\end{definition}

\subsection{Cusps and doubling} 
Let $M$ be a manifold, and consider $\Lambda \subset S^*M$.  
Fix a contact isotopy $\phi_t$ on $S^* M$, positive in a neighborhood of $\Lambda$.  For $s \ge 0$, let $\widetilde{\phi}_s^{\pm} = \phi_{\pm s^{3/2}}$, and 

$$\Lambda_{\pm s} := \tilde{\phi}^{\pm}_s (\Lambda) \subset S^*M \qquad \qquad \Lambda^{\prec} = \widetilde{\Phi}^+(\Lambda)
\cup \widetilde{\Phi}^-(\Lambda) \subset S^*(M \times \RR)$$ 
Note the contact reduction of $\Lambda^{\prec}$ over $s \ge 0$ is 
 $\Lambda_{s} \cup \Lambda_{-s}$, and is empty for $s < 0$.  
Also observe the germ of $S^*(M \times \R)$ around 
$\Lambda^\prec$ is locally 
contactomorphic to $\Lambda \times \R \subset Nbd(\Lambda) \times T^*\R$.  
In particular, $\mush_{\Lambda^\prec}$ is the pullback of $\mush_{\Lambda}$. 

If $\Lambda$ was a point in $S^* \R$ and we take the Reeb flow, 
then the front projection of  $\Lambda^\prec$ is the standard
cusp $y^2 = x^3$.  This example is studied in some detail by
elementary arguments in \cite[Sec. 3.3]{stz}
and \cite[Sec. 7.3.3]{nrssz}.  

We note the following key fact: 

\begin{lemma} \label{lem: stalk at cusp}
Objects in  $sh_{\Lambda^\prec}(M \times \R)$ are locally constant over $M \times \R_{\le 0}$.
\end{lemma} 
\begin{proof}
This is obvious over $M \times \R_{<  0}$.  
Over $M \times 0$, note that  $\Lambda^\prec$ has no conormals in the $ds$ directions,  so
noncharacteristic propagation shows that stalks agree with those over $M \times \R_{< 0}$.   
\end{proof}

We write $sh_{\Lambda^\prec; 0} \subset sh_{\Lambda^\prec}$ for the sheaf of full subcategories
of sheaves  (on $M \times \RR$, microsupported at infinity in $\Lambda^\prec$) with vanishing stalks near $M \times - \infty$
(and hence over $M \times \R_{\le 0}$).  By Lemma \ref{lem: stalk at cusp}, objects of $sh_{\Lambda^\prec; 0}$
are the $!$-extension of their restrictions to $M \times \R_{> 0}$. 

\begin{lemma} \label{lem: zero at cusp}
Let $j: M \times \R_{> 0} \to M \times \R$ be the inclusion.  Restriction on sheaves
induces a natural map of sheaves of categories on $M \times \R$: 
$$j^* :  sh_{\Lambda^\prec; 0}  \to j_* sh_{(M \times \R_{> 0}) \cap \Lambda^\prec  } $$
This map is fully faithful.   
\end{lemma}
\begin{proof}
There is nothing to prove except along $M \times 0$.  Along this locus, the point is just that
objects of $sh_{\Lambda^\prec; 0}$ are themselves extensions by zero from objects in 
$M \times \R_{> 0}$, as shown in Lemma \ref{lem: stalk at cusp}. 
\end{proof}

\begin{example}
Here is an example of how the map of Lemma \ref{lem: zero at cusp} is not surjective.  
Consider $M = \R$, and $\Lambda$ the union of positive conormals to $-1$ and $1$.  Then
$\Lambda^{\prec}$ is a union of conormals to cusps in $\R^2$ namely those given
by $(y +1)^2 = x^3$ and $(y-1)^2 = x^3$.  Consider in the half-plane $x > 0$ the appropriate
extension of the constant sheaf in the region $-1 -x^{3/2} < y < 1 + x^{3/2}$.  This is an element
of $sh_{\Lambda^{\prec}}(\R_y \times \R_{x > 0})$ which is not the image of any element
of $sh_{\Lambda^\prec; 0}(\R^2)$. 
\end{example}

\begin{proposition} \label{prop: constancy near cusp}
Let $c_1$ be (a certain universal positive function of) the length of the shortest chord for 
the flow of $\Lambda$ under $\phi_t$, and let $0 < s < s' < c_1$.  

Consider the restriction maps 
$$sh_{\Lambda^\prec; 0}(M \times \R_{\le s}) \xrightarrow{(1)} sh_{\Lambda^\prec; 0}(M \times \R_{< s} ) 
\xrightarrow{(2)} 
sh_{\Lambda^\prec; 0}(M \times \R_{< s'} ) \xrightarrow{(3)} 
(sh_{\Lambda^\prec; 0}|_{M \times 0})(M)$$
Then (1) and (2) are isomorphisms.  If $\Lambda$ is relatively compact, then also $(3)$ is an isomorphism.

In addition, restriction at s gives a fully faithful map: 
$$sh_{\Lambda^\prec; 0}(M \times \R_{\le s}) \hookrightarrow sh_{\Lambda_{s} \cup \Lambda_{-s}}(M)$$
\end{proposition}
\begin{proof}
Over $(0, c_1)$ we may 
apply a contact isotopy to move $\Lambda_{\pm s}$ in order to trace out $\Lambda$; 
note this becomes impossible exactly at $c_1$.  Hence (1) and (2) being isomorphisms 
follows from Cor \ref{cor: gks constancy}.   Assuming compactness, 
the fact that (3) is an isomorphism follows from proper base change (of sheaves of categories), 
applied to a neighborhood of the projection of $\Lambda$.  
The final assertion is Lemma \ref{lem: zero at cusp}. 
\end{proof}

\begin{remark}
Other tameness hypotheses can substitute for the assumption that $\Lambda$ is relatively compact.
\end{remark} 

Let $\pi: S^*M \to M$ be the projection, and likewise $\tilde{\pi}: S^*(M \times \R) \to M \times \R$. 

\begin{lemma} \label{lem: cusp antimicrolocalization} 
Assume $\Lambda \subset S^*M$ is closed and in finite position. 
Then the
natural map of sheaves of categories on $M \times \R$
$$sh_{\Lambda^\prec; 0} \to \tilde{\pi}_* \mush_{\Lambda^\prec}$$ 
is an equivalence along $M \times 0$. 
\end{lemma} 
\begin{proof}
It is enough to check at stalks.  
Evidently if $\Lambda$ is in finite position, then $\Lambda^{\prec}$ is in finite position over $0$. 
By Prop \ref{prop: local antimicrolocalization}, 
the stalk of $\mush$ is the orthogonal complement to constant
sheaves in a neighborhood, i.e. those sheaves with no local sections, i.e. with vanishing stalk.  Per Lemma \ref{lem: stalk at cusp} (in turn just an application of noncharacteristic propagation), this is   $sh_{\Lambda^\prec; 0}$.
\end{proof}

\begin{theorem}\label{thm:antimicrolocalization} 
Assume that  $\Lambda$ is closed and perturbable to finite position. 
For $s < c_1$, the composition 
$$sh_{\Lambda^\prec; 0}(M \times \R_{\le s} ) \to  sh_{\Lambda_{s} \cup \Lambda_{-s}}(M) \to   \mush_{\Lambda_{-s}}(\Lambda_{-s})$$
is an equivalence.  In particular, the second map has a right inverse.  
\end{theorem} 
\begin{proof}
By \cite{guillermou-kashiwara-schapira} the statement is invariant under contact perturbation (of $\Lambda$, before
forming $\Lambda^{\prec}$), so we may assume 
$\Lambda$ is in finite position. 

We have the commutative diagram 

\beq\label{eq: antimicrolocalization}
\begin{tikzcd} 
sh_{\Lambda^\prec; 0}(M \times \R_{\le s} ) 
\arrow[hook]{r}[swap]{\text{Prop. \ref{prop: constancy near cusp}}} 
\arrow{d}{\sim}[swap]{\text{Prop. \ref{prop: constancy near cusp}}} & 
sh_{\Lambda_{s} \cup \Lambda_{-s}}(M) 
\arrow{r} & 
\mush_{\Lambda_{-s}}(\Lambda_{-s}) \\
sh_{\Lambda^\prec; 0}|_{M \times 0}(M) 
\arrow{r}{\sim}[swap]{\text{Lemma \ref{lem: cusp antimicrolocalization}}} & 
\tilde{\pi}_* \mush_{\Lambda^\prec}|_{M \times 0}(M)  
\arrow{r}{\sim} & 
\mush_{\Lambda^{\prec}}(\Lambda^{\prec}) 
\arrow{u}{\sim}
\end{tikzcd} 
\eeq

The labelled arrows are fully faithful or equivalences by the noted previous results, or in the case of comparisons
between $\mush$ categories, by Lemma \ref{lem: relative constancy}.  The result follows. 
\end{proof}

It follows from Thm \ref{thm:antimicrolocalization} that the image of the map 
$sh_{\Lambda^\prec; 0}(M \times \R_{\le s}) \to sh_{\Lambda_{s} \cup \Lambda_{-s}}(M)$ is the subcategory generated 
by microstalks on $\Lambda_{-s}$.   We now discuss how to give a more explicit characterization (which however we will
not require in the sequel). 
By inspection of the \cite{guillermou-kashiwara-schapira} construction, the image is 
always contained in the locus of sheaves whose microsupport 
is contained in the locus
$$\U_s := \R_{> 0} \Lambda_{s} \,\, \cup \,\, \R_{> 0} \Lambda_{-s} \,\, \cup \bigcup_{-s \le r \le s} \pi(\Lambda_r) \subset T^*M$$

\begin{example}
It is not always the case that the map to $sh_{\U_s}$ is surjective.  Indeed, consider $M = \R$ and $\Lambda$ {\em both} conormals
to $0$.  Now $\mush_\Lambda$ is two copies of the coefficient category, while $sh_{\U_s}$ is representations of the $A_3$ quiver. 
An object not in the image is the skyscraper at 0.  
\end{example}

The above is essentially the only difficulty, and in its absence we can characterize the image. 

\begin{proposition} \label{prop: image of antimicrolocalization}
Assume that $\Lambda$ is closed and in finite position, and $c_1 > 0$. 

Assume in addition that $sh_{\RR_{\geq 0}\Lambda}(M) = 0$ (which holds, for example, if the front projection of $\Lambda$ is injective on a dense locus). 
Then the map $$sh_{\Lambda^\prec; 0}(M \times \R_{\le s}) \to sh_{\U_s}(M)$$
is an isomorphism for small $s$. 
\end{proposition} 
\begin{proof}
There is a natural candidate for the inverse.  
To construct it, for  $s>0$, set $\U_{(0, s]}=\cup_{r\in (0,s]} \U_r$.
Note that the reduction along $M\times 0$ of the closure of $\U_{(0, s]}$ is precisely $\RR_{\geq 0} \Lambda$.
Under the assumptions on $\Lambda$, for small enough $s>0$, the restriction map
$\rho_r:\sh_{\U_{(0, s]}}(M \times (0, s]) \to  
\sh_{\U_r} (M)$ is an equivalence for all $r\in (0, s]$. 
Now the candidate inverse is the composition with the $!$-extension
 $\phi = j_! \circ \rho_s^{-1}:sh_{\U_s}(M) \stackrel{\sim}{\to} \sh_{\U_{(0, s]}}(M \times (0, s])\to sh(M \times \R_{\le s})$.   Our task is to show that for $\cF\in sh_{\U_s}(M)$,
the microsupport of $\phi(\cF)$  lies in $\Lambda^\prec$. Note by construction,  the  support of $\phi(\cF)$  lies  within the closure of $\U_{(0, s]}$, and the
$*$-restriction of $\phi(\cF)$ to $M \times \RR_{\leq 0}$ vanishes.

It remains to  calculate the microsupport of $\phi(\cF)$ for non-zero covectors along $M \times 0$. First note that, as in the context of relative microsupport, for a  covector of the form $(m; \xi, 0) \in T^*M \times T^*_0 \RR$, with $\xi \not = 0$, to be in the microsupport, we must have $\xi \in \Lambda.$ So it remains to show any 
covector of the form $(m; \xi, u) \in T^*M \times T^*_0 \RR$, with $u\not = 0$, is not in the microsupport.

Consider pairs of a small open ball $B_\epsilon \subset  M \times \RR$ of radius $\epsilon>0$ centered at $(m, 0)$, and a function $f:M \times \RR\to \RR$, with $f(m, 0) = 0$, and $df_{(m, 0)} = (\xi, u)$.  

Set $H_{\epsilon} = \{ b\in B_\epsilon \, |\, f(b) <0\}$. Since the stalk of $\phi(\cF)$ at $(m, 0)$ vanishes, we need to show $\Gamma(H_{\epsilon}, \phi(\cF))$ vanishes for small enough $\epsilon>0$.  

For $u >0$, note that $H_\epsilon$, for small enough $\epsilon>0$, is disjoint from the support of $\phi(\cF)$, so $\Gamma(H_{\epsilon}, \phi(\cF))$ indeed vanishes in this case.

To handle $u <0$, set $K_{\epsilon, \delta} = H_\epsilon \cap M\times \delta$. Note that, by non-characteristic propagation,  the restriction $\Gamma(H_{\epsilon}, \phi(\cF)) \to \Gamma(K_{\epsilon, \delta}, \phi(\cF))$ is an isomorphism for small $\epsilon>0$ and very small $\delta> 0$. Unwinding the definitions, if we set $C_\epsilon\subset M$ to be 
 a small open ball of radius $\epsilon>0$ centered at $m$ in the front of $\Lambda$, 
 we are left to calculate 
$\Gamma(C_{\epsilon}, \cF)$, for $\cF\in sh_{\U_s}(M)$, where
$\epsilon>0$ is small and 
$s>0$ is very  small. This vanishes since it factors through the nearby cycles as $s\to 0$ whose target is the zero category 
$sh_{\RR_{\geq 0}\Lambda}(M)$, by assumption.
\end{proof}

\subsection{Relative doubling} 

Previously we have considered closed $\Lambda \subset S^*M$.  We now generalize this to the case
of locally closed $\Lambda$.  We write $\partial \Lambda := \overline{\Lambda} \setminus \Lambda$
(as opposed to $\overline{\Lambda}$ minus its relative interior). 
The need for a generalization is illustrated by the following example: 

\begin{example}
Recall that the microsupport of a sheaf can never be a manifold with nontrivial boundary.  

Consider $M=\R^2$,
and let $\Lambda \subset J^1 \R \subset S^*(\R^2)$ be the Legendrian with front projection the open interval $(0,1) \times \{0\} \subset \R^2$.  Then $\mush_\Lambda(\Lambda)$ is the coefficient category. 

Fix any positive isotopy; for illustration we take the jet bundle Reeb flow (translation in the $\R$ direction of $J^1 \R$).  Then $\U_s$ will be a union of a rectangle $(0, 1) \times [-s, s]$ in the zero section and the closure to positive conormals of $(0,1) \times \{\pm s\}$.
Anything microsupported in $\U_s$ is certainly microsupported in $\overline{\U_s}$, which, away from the zero section, each component is a manifold with nontrivial boundary, hence any sheaf microsupported in $\overline{\U_s}$ is in fact microsupported in the zero section; but $\U_s$ does not include the whole zero section, so the sheaf is zero.  That is, $sh_{\U_s}(\R^2)=0$. Thus the assertion of Theorem \ref{thm:antimicrolocalization} is false for this $\Lambda$. 

This is no contradiction: Theorem \ref{thm:antimicrolocalization} applies only to closed $\Lambda$, 
ultimately because of the use of Lemma \ref{lem: local antimicrolocalization}. 
We could apply Theorem \ref{thm:antimicrolocalization}
to the closure of $\Lambda$, but this would not be useful as $\mush_{\overline{\Lambda}}(\overline{\Lambda}) = 0$
(again by the recollection above).  
\end{example} 

The solution is to cap off the boundaries of the 
doubled object.  

\begin{definition}
A contact collar of a subset $\Lambda \subset S^*M$ is 
a contact manifold $U$ with a subset $L$, and an embedding
$(L \times (0, \epsilon); U \times T^*(-\epsilon, \epsilon)) \to (\Lambda, S^*M)$ carrying
$L \times 0$ diffeomorphically onto $\partial \Lambda$.

We say a contact flow on $S^*M$ is compatible with the collar if its restriction to 
$U \times T^*(-\epsilon, \epsilon)$ is the product of a flow on $U$ and the constant
flow on $ T^*(-\epsilon, \epsilon)$. 
\end{definition}

For collared $\Lambda$, 
we define $(\Lambda, \partial \Lambda)^{\prec} \subset S^*(M \times \R)$
as the union of $\Lambda^{\prec}$ with a `boundary capping off' component contained entirely inside 
$U \times T^*(-\epsilon, \epsilon) \times T^*\R$.  This component is constructed as follows.  
Note that $\Lambda^{\prec}$ in this region is $L \times [0,\epsilon)$ times a parabola $P \subset T^* \R$.  We fix some standard
capping of the parabola in $T^* \R$ to a Legendrian half-paraboloid $\tilde{P} \subset T^* (-\epsilon, 0] \times T^* \R$, so that 
$\tilde{P} \cup_P ([0, \epsilon) \times P) \subset T^* (-\epsilon, \epsilon) \times T^* \R$ is a manifold.\footnote{It is enough for our purposes
that this be a $C^1$ manifold; thus the construction can be made in the subanalytic category if desired.}  We define 
\beq
(\Lambda, \partial \Lambda)^{\prec} := \Lambda^{\prec} \cup (L \times \tilde{P})
\eeq

Note that by construction $\Lambda^{\prec} \subset (\Lambda, \partial \Lambda)^{\prec}$, and in the fiber over $M \times 0$ the 
latter is the closure of the former. 

\begin{example}
If $M = \R^2$ and $\Lambda$ is a Legendrian open interval whose front projection is an embedding, then 
$(\Lambda, \partial \Lambda)^{\prec}$ is a disk whose front projection is half of a somewhat squashed version of the standard front projection of the two dimensional Legendrian unknot.
\end{example}

\begin{remark}
The reason for the delicate construction of the capping is that we want 
to verify local constancy of $\mush$ along the boundary by simply invoking Lemma \ref{lem: relative constancy}.  
The local constancy would still be true with a less delicate capping, but 
one would have to argue for it directly.
\end{remark}

Let us write $(\Lambda, \partial \Lambda)_s \subset S^*M$ for the contact reduction of 
$(\Lambda, \partial \Lambda)^{\prec} \subset S^*(M \times \R)$ over $s \in \R$.  

\begin{lemma}
The restriction of $(\Lambda, \partial \Lambda)^{\prec}$ to $S^*(M \times \R_{>0})$ is the legendrian movie 
of a contact isotopy acting on $(\Lambda, \partial \Lambda)_s$. 
\end{lemma} 

Analogues of the results of the previous section hold for $(\Lambda, \partial \Lambda)^{\prec}$. 

\begin{proposition} \label{prop: constancy near cusp rel} 
If $\Lambda$ is relatively compact, then for any $0 < s' < s < c_1$,  the following natural restriction maps are all isomorphisms. 
$$sh_{(\Lambda, \partial \Lambda)^\prec; 0}(M \times \R_{\le s}) \xrightarrow{\sim} 
sh_{(\Lambda, \partial \Lambda)^\prec; 0}(M \times  \R_{< s}) \xrightarrow{\sim}
sh_{(\Lambda, \partial \Lambda)^\prec; 0}(M \times \R_{< s'}) \xrightarrow{\sim} 
sh_{(\Lambda, \partial \Lambda)^\prec; 0}|_{M \times 0}(M)$$
In addition, restriction at $s$ induces a fully faithful functor
$$sh_{(\Lambda, \partial \Lambda)^\prec; 0}(M \times \R_{\le s}) \hookrightarrow sh_{(\Lambda, \partial \Lambda)_{s}}(M)$$
\end{proposition}
\begin{proof}
Same as Proposition \ref{prop: constancy near cusp}, with the additional observation that the capping off at the boundary
adds no chords.  
\end{proof}

\begin{lemma} \label{lem: cusp relative antimicrolocalization} 
Assume $\Lambda$ is collared and in finite position. 
Then the
natural maps of sheaves of categories on $M \times \R$
$$sh_{(\Lambda, \partial \Lambda)^\prec; 0} \to \tilde{\pi}_* \mush_{(\Lambda, \partial \Lambda)^\prec} \to 
\tilde{\pi}_* \mush_{\Lambda^\prec}$$ 
are both equivalences along $M \times 0$. 
\end{lemma} 
\begin{proof}
Note that since $\Lambda$ is in finite position, 
so is $(\Lambda, \partial \Lambda)^{\prec}$ over $0$.  Now the 
first map is an isomorphism by the same proof as Lemma  \ref{lem: cusp antimicrolocalization}.  
The second map is an isomorphism because (1) over $M \times 0$ the loci  
$\Lambda^\prec \subset (\Lambda,\partial \Lambda)^\prec$ agree except
along the boundary where one is topologically $L \times (0, \epsilon)$ and the other $L \times [0, \epsilon)$, 
and (2) $\mush$ is locally constant along the $[0, \epsilon)$ direction by Lemma \ref{lem: relative constancy}. 
\end{proof}

\begin{theorem}\label{thm: relative antimicrolocalization} 
Assume $\Lambda$ is collared,  relatively compact, and perturbable to finite position. 
For $s < c_1$,  the composition 
$$sh_{(\Lambda, \partial \Lambda)^\prec; 0}(M \times \R_{\le s} ) \to  sh_{(\Lambda, \partial \Lambda)_s}(M) \to   \mush_{\Lambda_{-s}}(\Lambda_{-s})$$
is an equivalence.  In particular, the second map has a right inverse.
\end{theorem} 
\begin{proof}
Similar to the proof of Theorem \ref{thm:antimicrolocalization}. 
\end{proof}

\begin{proposition} \label{prop: image of relative antimicrolocalization} 
Proposition \ref{prop: image of antimicrolocalization} holds for 
any locally closed, relatively compact,  and collared $\Lambda$, with 
$$\U_s := \R_{> 0} (\Lambda, \partial \Lambda)_{s} \,\, \cup \,\, \R_{> 0} (\Lambda, \partial \Lambda)_{-s} \,\, \cup \bigcup_{-s \le r \le s} \pi((\Lambda, \partial \Lambda)_r) \subset T^*M$$
\end{proposition}
\begin{proof}
Same as Proposition \ref{prop: image of antimicrolocalization}.
\end{proof}

\subsection{Two sided double}
Our previous constructions involved $M \times \R$ or $M \times [0,\epsilon]$,
and microsupport going to infinity or the boundary.  While these suffice for our purposes in this
article, we give here a variant
of the doubling construction which avoids this defect.   
This can be technically convenient when invoking theorems stated for compact manifolds; and will be the version
of antimicrolocalization invoked in \cite{gpsconstructible}.  
This subsection owes its existence to discussions with the authors of that article.

The prototype of the construction is to begin with a Legendrian point in the contact manifold $\R$, and produce the standard
Legendrian unknot, rather than simply half of it as before.  We just use two 
copies of our construction above: first begin forming $\Lambda^{\prec}$ or 
$(\Lambda, \partial \Lambda)^{\prec}$ as before, but near some time $s < c_1$, smoothly cutoff the Reeb pushoff
so that in the region $t \in (s-\epsilon, s]$ one has $(\Lambda, \partial \Lambda)_{\pm t}$ independent of $t$.  
Finally, reverse the process so that the Legendrian for $t > s$ is just the reflection of the Legendrian when $t < s$.  

We denote the resulting Legendrian as $(\Lambda, \partial \Lambda)^{\lunknot} \subset S^*(M \times \R)$.   (Despite the notation,
the corresponding front will be connected and rounded at the top and bottom.) 
Note this is supported over
a compact subset of $\R$; fixing an inclusion $\R \subset S^1$ we may view 
$(\Lambda, \partial \Lambda)^{\lunknot} \subset S^*(M \times S^1)$.  As before we write 
$sh_{(\Lambda, \partial \Lambda)^{\lunknot} }(M \times S^1)_0$ for the full subcategory of sheaves 
with vanishing stalks over $S^1 \setminus \R$. 

\begin{theorem}
The category $sh_{(\Lambda, \partial \Lambda)^{\lunknot} }(M \times S^1)_0$ is the (left and/or right) orthogonal complement
to the category of local systems on $M \times S^1$.  Moreover, the following natural morphisms
are equivalences: 
$$sh_{(\Lambda, \partial \Lambda)^{\lunknot} }(M \times S^1)_0 \xrightarrow{\sim} \mush^{pre}_{(\Lambda, \partial \Lambda)^{\lunknot}}((\Lambda, \partial \Lambda)^{\lunknot}) \xrightarrow{\sim} \mush_{\Lambda}(\Lambda)$$
\end{theorem} 
\begin{proof}
Let us first see that any  $\mathcal{F} \in sh_{(\Lambda, \partial \Lambda)^{\lunknot} }(M \times S^1)_0$ is orthogonal to all local systems.  By construction $\mathcal{F}$ is supported over some interval in $S^1$; thus morphisms between $\mathcal{F}$ 
and a local system will 
factor through some $\mathcal{G}$ which is a local system on $M$ 
times the constant sheaf on an open or closed interval $I$.  Away from the zero section, the microsupport of $\mathcal{G}$
is entirely in the $d t$ direction ($t$ the coordinate on $S^1$), above $\partial I$.  Shrinking $I$ to an interval entirely disjoint
from the support of $\mathcal{F}$ and correspondingly propagating $\mathcal{G}$ is a noncharacteristic homotopy since $(\Lambda, \partial \Lambda)^{\lunknot}$ is by construction
$S^1$-noncharacteristic, i.e. disjoint from $dt$ covectors.  
Thus we see that $sh_{(\Lambda, \partial \Lambda)^{\lunknot} }(M \times S^1)_0$ is contained in both the left and right orthogonal
complements to the local systems.  
For the converse inclusion, consider now any $\mathcal{F} \in sh_{(\Lambda, \partial \Lambda)^{\lunknot} }(M \times S^1)_0$.
Let $\mathcal{L}_\mathcal{F}$ be the local system on $M$ given by restricting $\mathcal{F}$ to $M \times (S^1 \setminus \R)$. 
Then the same noncharacteristic propagation argument shows that the identity induces morphisms in both directions
between $\mathcal{F}$ and $\mathcal{L}_F$, which are zero only if $\mathcal{L}_{\mathcal{F}}$ vanishes.  

Regarding the morphisms in the second assertion of the proposition, the above orthogonality implies that the first 
is an isomorphism.  Meanwhile the composite 
$sh_{(\Lambda, \partial \Lambda)^{\lunknot} }(M \times S^1)_0 \to \mush_{\Lambda}(\Lambda)$ factors through
$$sh_{(\Lambda, \partial \Lambda)^{\lunknot} }(M \times S^1)_0  \to 
sh_{(\Lambda, \partial \Lambda)^{\prec} }(M \times \R_{\le \epsilon})_0 \xrightarrow{\sim} \mush_{\Lambda}(\Lambda)$$
where we learned that the second morphism is an isomorphism in Theorem \ref{thm: relative antimicrolocalization}.  
Thus it suffices to show the first morphism is an isomorphism. 
If  $(\Lambda, \partial \Lambda)^{\lunknot}$ is supported above $[0,T] \subset \R \subset S^1$, then 
the same argument as in Lemma \ref{lem: zero at cusp} shows that $sh_{(\Lambda, \partial \Lambda)^{\lunknot} }(M \times S^1)_0$
is restricts by an equivalence to $[0, T)$, and the same argument as in Proposition \ref{prop:  constancy near cusp} shows
that this is in turn restricts by an equivalence to $[0, \epsilon)$.   
\end{proof}


\section{Gapped specialization of microsheaves} 

We are  interested here in limits of contact isotopies. 

\begin{definition}
Let $Y$ be a manifold and $Z_t$ a family of subsets of $Y$ 
defined for $t \in (0,1]$.  We write 
$$Z_0 := \lim_{t \to 0} Z_t := \overline{\bigcup{t \times Z_t}} \cap (0 \times Y)$$
where the closure and intersection are taken in $[0,1] \times Y$. 
\end{definition}

\begin{remark}
We will be interested in the case when $\phi_t$ is a contact isotopy, and the $Z_t$ are $\phi_t(Z)$. 
In this case we could also form 
a similar construction
on the contact movie of $\phi_t$.  The symplectic reduction of the zero fiber of the closure
of the movie is contained in, but not in general equal to, the limit above. 
\end{remark}

Using the results on antimicrolocalization, we now give a microlocal version of the theorem 
on gapped specialization (Theorem \ref{thm: gapped specialization is fully faithful}). 

\begin{theorem} \label{thm:specialization}
Consider $\Lambda_1 \subset S^*M$, which is either compact or locally closed, relatively compact, 
and collared. 
Let $\phi_t: S^*M \to S^*M$ 
be a contact isotopy for $t \in (0, 1]$, compatible with the collar. 
Let $\widetilde{\Lambda} \subset S^*M \times T^*(0,1] \subset S^*(M \times (0,1])$ be the contact movie, 
and let  $\Lambda_0 = \lim_{t \to 0} \phi_t(\Lambda_1)  \subset S^*M$.  Assume: 
\begin{enumerate}
\item \label{hyp: g} For some contact form on $S^*M$ compatible with the collar,  the family $\Lambda_t$ (including $t=0$) is gapped.
\item  \label{hyp: f}  Both $\Lambda_0$ and $\Lambda_1$ project finitely to $M$. 
\item  \label{hyp: r}  $\Lambda_0$ is  spdff.
\end{enumerate} 

Then antimicrolocalization and nearby cycles induce a fully faithful functor 
$$\mu sh_{\Lambda_1}(\Lambda_1) \to \mu sh_{\Lambda_0}(\Lambda_0)$$
\end{theorem}
\begin{proof}
We will use the gapped hypothesis both in constructing the functor, and in proving full faithfulness.  Let $\eta_s$ 
be the Reeb flow for the given contact form.  Let $\widetilde{\eta}_s$ be the lift of $\eta_s$ to 
$S^*M \times T^*(0,1]$; this flow is positive (as recalled in Lemma \ref{lem: partial cosphere}).

As we assumed $\phi$ compatible with the collar, $\Lambda_0$ is also collared (in a compatible way).
We
form $(\widetilde{\Lambda}, \partial \widetilde{\Lambda})^\prec \subset S^*(M \times (0,1] \times \R)$ 
and $(\Lambda_0, \partial \Lambda_0)^{\prec} \subset S^*(M \times \R)$.  Note that 
$(\Lambda_0, \partial \Lambda_0)^{\prec}$ is the closure at zero of the projection of 
$(\widetilde{\Lambda}, \partial \widetilde{\Lambda})^\prec$ 
to $S^*(M \times (0,1]) \times \R$.  We have the diagram:

\beq\label{eq: microlocal nearby cycles}
\begin{tikzcd}
\mush_{\Lambda_1}(\Lambda_1) &  
sh_{(\Lambda_1, \partial \Lambda_1)^\prec; 0}(M \times \R_{\le \epsilon}) 
\arrow{l}[swap]{\sim} 
\arrow[hook]{r} 
 & 
sh_{(\Lambda_1, \partial \Lambda_1)_{\epsilon}}(M) 
 \\
\mush_{\widetilde{\Lambda}}(\widetilde{\Lambda}) \arrow{u}{\sim} &
sh_{(\widetilde{\Lambda}, \partial \widetilde{\Lambda})^\prec; 0}(M \times (0,1] \times \R_{\le \epsilon}) 
\arrow{l} 
\arrow{r} 
\arrow{d}{\psi}  \arrow{u}{\sim} & 
sh_{(\widetilde{\Lambda}, \partial \widetilde{\Lambda})_{\epsilon}}(M \times (0,1]) 
\arrow[hook]{d}{\psi} \arrow{u}{\sim} \\
\mush_{\Lambda_0}(\Lambda_0) &  
sh_{(\Lambda_0, \partial \Lambda_0)^\prec; 0}(M \times \R_{\le \epsilon})  
\arrow{l}[swap]{\sim} 
\arrow[hook]{r} &
sh_{(\Lambda_0, \partial \Lambda_0)_{\epsilon}}(M) 
\end{tikzcd}
\eeq

The upward arrows are isomorphism by \cite{guillermou-kashiwara-schapira} (the leftmost one alternatively 
by Lemma \ref{lem: relative constancy}).  
The top and bottom 
right horizontal arrows are the restriction at $\epsilon$, and are fully faithful by Prop. \ref{prop: constancy near cusp rel}.  
The top and bottom left horizontal arrows are equivalences by Theorem \ref{thm: relative antimicrolocalization}. 
(If $\Lambda$ is closed we may use the easier Prop. \ref{prop: constancy near cusp} and Theorem \ref{thm:antimicrolocalization}.)  
Note we 
are using gappedness  including at $\Lambda_0$ in order to apply antimicrolocalization to $\Lambda_0$ using the same $\epsilon$
as for the family. 

The downward arrows are induced by nearby cycles, and the image sheaves have the stated microsupports
by the standard estimate (Lemma \ref{nearbyestimate}).  By Theorem \ref{thm: gapped specialization is fully faithful}, 
the right downward arrow is fully faithful;  it follows formally that the middle is as well.  
Following around the diagram we find the 
desired fully faithful functor $\mush_{\widetilde{\Lambda}}(\widetilde{\Lambda}) \hookrightarrow \mush_{\Lambda_0}(\Lambda_0)$.
 \end{proof} 

\begin{remark}
    In practice, Theorem \ref{thm:specialization} may be applied after conjugation by a 1-parameter family of contactomorphisms $\phi_t :S^*M \to S^*M$, in order to weaken the hypotheses to hold after suitable perturbation.
\end{remark}

\begin{remark}
In the above diagram, 
we use the top row to avoid quoting theorems for the middle row due to its noncompactness and also because we do
not want to check hypotheses for $\tilde{\Lambda}$.  We use the middle column 
to avoid explicitly characterizing the images in the right column; 
alternatively we could use  Proposition \ref{prop: image of relative antimicrolocalization}.  We use the right column because 
the gappedness is more evident there than in the middle column.  
\end{remark}

\begin{remark}
Note that gappedness of $\Lambda_t$ including at $t=0$ is equivalent to gappedness of the family $\Lambda_t$ for $t \ne 0$
plus $\epsilon$-chordlessness of $\Lambda_0$ by itself, these all being considered for the same isotopy. 
\end{remark}

\begin{remark} (On the isotropicity hypotheses.) \label{lament}
The hypotheses (\ref{hyp: f}) and (\ref{hyp: r}) on $\Lambda_0$ roughly mean it is isotropic.  As a consequence we will be later restricted to those Liouville manifolds with isotropic skeleton, 
and also only be able to prove invariance under homotopies through such manifolds. Not all Liouville manifolds have isotropic skeleton \cite{mcduff-nonweinstein} (although the ones which appear in applications to mirror symmetry or geometric representation theory do). On the other hand, the wrapped Fukaya category is defined for all Liouville manifolds, functorial under exact symplectomorphism of such, and (after much work) equivalent to the microsheaf categories studied here, when the latter are defined \cite{gpssectorsoc, gpsdescent, gpsconstructible}.  In short it is natural to expect that it should be possible to relax the isotropicity hypotheses in our work.  
We now recall
precisely how the hypotheses arise, and speculate on how they may be relaxed. {\em This discussion will occupy the remainder of this section, and nothing in the present
article depends upon it}.  

Before we begin let us note that in contrast to hypotheses  (\ref{hyp: f}) and (\ref{hyp: r}), the gappedness hypothesis  (\ref{hyp: g}) does
{\em not} imply isotropicity (in particular, a Liouville hypersurface in a contact manifold is gapped) and seems to be more fundamental. 

The hypothesis (\ref{hyp: r}) is inherited ultimately from Lemma \ref{lem: reasonable} and Lemma \ref{lem:lower square}, 
where it is used to control the limiting behavior of a family using only facts about the limit geometry.  However,
in our setting we are not fully ignorant of the nature of the family; plausibly this could be used in place of 
the hypothesis on the limit. 

We will discuss hypothesis (\ref{hyp: f}) in more detail.  It 
comes from the fact that we use antimicrolocalization in order to define the specialization functor
on microsheaves, and antimicrolocalization ultimately relies on Lemma \ref{lem: local antimicrolocalization} which
requires a finite front projection.  On the one hand, perhaps this lemma may be improved -- it derives from the refined
microlocal cutoff
\cite[Proposition 6.1.4]{kashiwara-schapira}, which has no explicit isotropicity hypotheses.  

On the other hand, it may seem odd
that we use antimicrolocalization to define the specialization functor at all: why not simply argue that nearby cycles respects
microsupports hence factors through microlocalization?  Let us develop this idea somewhat in order to expose a subtle
difficulty.  Let $\pi: E \to B$ be a smooth fiber bundle.  The results of Section \ref{sec: relative microsupport} make it 
natural to define a sheaf of categories of {\em relative microsheaves} on $T^*\pi$ as follows.  For an open 
$U\subset T^*\pi$, take 
$$Null(U) = \{\cF \in sh(E)\, |\, ss_\pi(\cF) \cap U = \emptyset \} \subset sh(E)$$
and define $\mush_{T^*\pi}$ as the sheafification of the presheaf given by
$$\mush^{pre}_{T^* \pi}(U) = sh(E) / Null(U)$$
One virtue of this construction is that for a submanifold $A \subset B$, it follows
from the estimates
Lemma \ref{lem:bundlerestrictionestimate} and Remark \ref{pointrestrictionestimate}. 
that there is a restriction map $\mush_{T^* \pi}|_{T^* \pi_A} \to \mush_{T^* \pi_A}$.  The same 
would not be true if we simply took the usual microsheaves $\mush_{T^* E}$ on $T^*E$ and pushed this sheaf of categories
to  $\Pi_* \mush_{T^*E}$ on $T^*\pi$. 
Indeed, while one has the equality of presheaves $\mush_{T^* \pi}^{pre} = \Pi_* \mush^{pre}_{T^*E}$, the sheafification
does not commute with the pushforward: i.e. the natural map  
$$\mush_{T^* \pi} = (\Pi_* \mush^{pre}_{T^*E})^{shf} \to
\Pi_* ((\mush^{pre}_{T^*E})^{shf})= \Pi_* \mush_{{T^*E}}$$  is
 is not usually an isomorphism, and there is no obvious map  between the left and right terms in 
$$(\Pi_* \mush)|_{T^* \pi_A}  \leftarrow \mush_{T^*\pi}|_{T^*\pi_A} \to  \mush_{T^* \pi_A}$$

To understand
the interaction of relative microsheaves with nearby cycles, it therefore remains to consider 
an open $j: B' \subset B$ and contemplate inducing a functor
on relative microsheaves from $j_*$.  
 
Let $\pi': E' \to B'$ be the restriction, and similarly denote the various analogous maps with~$'$.  
We want to consider the relationship between $\mush_{T^*\pi'}$ and $\mush_{T^* \pi}$.  
 We write $J_*$ for pushing forward sheaves or presheaves 
of categories by $j$.  We claim there always exists the following diagram:

\begin{equation}
\xymatrix{
\mush_{T^*\pi}^{pre}    \ar[d]  & J_* \mush_{T^* \pi'}^{pre} \ar[d] \ar[l]^{j_*} & \\
\mush_{T^*\pi} & (J_* \mush_{T^* \pi'}^{pre})^{sh} \ar[l]^{j_*} \ar[r] & J_* \mush_{T^* \pi'}
}
\end{equation}

Above, $j_*$ on sheaves induces the top line by microsupport estimates.  The lower line is induced 
because sheafification is a functor, and a map from a presheaf to a sheaf factors through the sheafification. 
However note that {\em we have not} succeeded in constructing a functor 
``$j_* : J_* \mush_{T^* \pi'} \to \mush_{T^*\pi}$'', since
the natural map $(J_* \mush_{T^* \pi'}^{pre})^{sh} \to J_* \mush_{T^* \pi'}$ need not be an isomorphism.  
One difficulty arises by considering the possibility of sheaves with 
components of the microsupport accumulating near the boundary of $B'$.   

Plausibly by imposing appropriate hypotheses for the 
microsupport near $\partial B'$ (e.g. fixing some $\Lambda' \subset T^*\pi'$), one can ensure 
that $(J_* \mush_{\Lambda'}^{pre})^{sh} \to J_* \mush_{\Lambda'}$ is an isomorphism, and
then use this construction to define the microlocal specialization functor.  It would be of interest
to determine for what $\Lambda'$ this holds, and more generally to develop the relative microlocal sheaf theory. 
\end{remark} 

\section{Microsheaves on polarizable contact manifolds} \label{sec:polarizable}

Following \cite{shende-microlocal}, we explain how a category of microsheaves can be associated
to polarizable contact manifolds or their closed subsets, such as the skeleton of a Weinstein manifold.  
We then apply Theorem \ref{thm:specialization} to show that Lagrangians define objects, and that 
the category associated to (the skeleton of a) Weinstein manifold is invariant under deformation of
the symplectic primitive.  Note such deformations cause dramatic changes in the geometry of the skeleton. 

The polarizability hypothesis is a very strong version of asking for Maslov obstructions to vanish.  We will
later relax the former hypothesis to the latter.

\subsection{The embedding trick} \label{sec: embedding}

Following  \cite{shende-microlocal}, we use high codimension embeddings to define global categories of microsheaves.  
The existence of the requisite embeddings follows from
Gromov's h-principle for contact embeddings, which implies in particular 
that for any contact manifold $\mathcal{U}$, 
there's a nonempty space of embeddings $\mathcal{U} \subset \R^{2N + 1 \gg 0}$, which
can be made as connected as desired by increasing $N$.

Such an embedding gives $\mathcal{U}$ its {\em stable symplectic normal bundle} $\nu_\mathcal{U}$; 
as a stable symplectic
bundle, it is the negative of the tangent bundle. We (contactomorphically) identify 
a tubular neighborhood of $\mathcal{U}$ in its
embedding with neighborhood of the zero section in this normal bundle.  
By a {\em thickening}, we mean the preimage under this identification of 
the total space of a Lagrangian subbundle of $\nu_\mathcal{U}$.  Evidently 
this is determined by a section $\sigma$ of $LGr(\nu_\mathcal{U})$; 
we term such a section a ``stable normal polarization'', and denote
the corresponding thickening by $\mathcal{U}^\sigma$.   

Note the following facts about stable normal polarizations: 

\begin{lemma} \label{lem:normalsec}
Let $(M, d\lambda)$ be an exact symplectic manifold, and $(M \times \R, \lambda + dt)$ its contactization.  The following are equivalent: 
\begin{itemize}
\item A section of  $LGr(TM \oplus \R^{2n}) \to M$ for $n \gg 0$.
\item A section of  $LGr(\ker(\lambda+ dt) \oplus \R^{2n}) \to M \times \R$ for $n \gg 0$
\end{itemize}

Let $(V, \eta)$ be any co-oriented contact manifold.  Then the following are equivalent: 
\begin{itemize}
\item A section of  $LGr(\ker(\eta) \oplus \R^{2n}) \to V$ for $n \gg 0$
\item A section of the Lagrangian Grassmannian bundle of the stable normal bundle of $V$
\end{itemize}
\end{lemma}
\begin{proof}
From $M$ one forms $M \to BU$ classifying its stable symplectic tangent bundle, and composes with the projection $BU \to B(U/O)$
to get the map $\phi: M \to B(U/O)$ classifying
the Lagrangian Grassmannian.  The Lagrangian Grassmanian of $\ker(\lambda+ dt)$ over $M \times \R$ is likewise classified
by a map $\widetilde{\phi}: M \times \R \to B(U/O)$, which moreover satisfies $\widetilde{\phi}|_{M \times 0} = \phi$.  

Now if $(V, \eta)$ is any contact manifold, and $\widetilde{\phi}: V \to B(U/O)$ classifies $LGr(\ker(\eta) \oplus \R^{2n})$, then
the Lagrangian Grassmannian of the stable normal bundle to $V$ is classified by $-\widetilde{\phi}$.  

As $U/O$ is an infinite loop space, hence a group (homotopically speaking), a section of the stable Lagrangian 
Grassmannian is the same as a trivialization
of it, hence given by a null-homotopy of, respectively, $\phi, \widetilde{\phi}, -\widetilde{\phi}$ in the above cases.  
Evidently these are equivalent.  
\end{proof}

In any case, fixing a stable normal polarization, we may define microsheaves on any contact manifold. 

\begin{definition}
We write $\mush_{\mathcal{U}; \sigma} := \mush_{\mathcal{U}^\sigma}|_{\mathcal{U}}$.
\end{definition}

\begin{remark}
Note that thickening by a lagrangian in the normal bundle ensures local constancy along the normal directions 
(Lemma \ref{lem: relative constancy}). 
\end{remark}

Let us consider on what this invariant may depend.  
The $h$-principle implies that different embeddings are themselves isotopic, so 
by another application of Lemma  \ref{lem: relative constancy} we see that 
$\mush_{\mathcal{U}; \sigma}$ does not depend on the embedding of $\mathcal{U}$, 
save perhaps through the dimension $2N+1$
of the target.  Similarly, homotopic choices of normal polarization give equivalent categories.
(We will see in Section \ref{sec:descent} below that in fact 
the category depends only on the image of the normal polarization under a certain map.)

To show independence of the embedding dimension, it suffices to observe that the category we defined is preserved under 
replacing $\mathcal{U}$ by $\mathcal{U} \times T^* [0,1]$.  The section of the Lagrangian Grassmannian is promoted by choosing
a constant section in the $T^*[0,1]$ direction, e.g. the zero section.

The objects representing objects in 
$\mush_{\mathcal{U}}$ have a (micro)support in $\mathcal{U}^\sigma$, constant in the normal bundle directions. 
We define their corresponding microsupport in $\mathcal{U}$ to be the restriction to $\mathcal{U}$. 
This newly defined microsupport  is evidently co-isotropic (since the original was by \cite{kashiwara-schapira}). 

\begin{definition}
For $\Lambda \subset \mathcal{U}$, we write $\mush_{\Lambda, \sigma}$ for the
subsheaf of $\mush_{\mathcal{U}, \sigma}$ consisting of full subcategories on the objects microsupported in $\Lambda$. 
\end{definition}

Obviously $\mush_{\Lambda, \sigma}$ only depends on the ambient $\mathcal{U}$ through its
germ along $\Lambda$. 

\begin{remark} 
When $\mathcal{U} = S^*M$ is a cosphere bundle, there are now two notions of microsheaves on $\mathcal{U}$.  The first
is the original $\mush_{S^*M}$.  The second is obtained from what we have just introduced, by observing that 
the cosphere fibers induce a polarization of the cosphere bundle, so in particular a stable normal
polarization. 
It is not difficult to check that the original $\mush_{S^*M}$ is canonically isomorphic to the new 
$\mush_{S^*M, \mathrm{fibers}}$, as sheaves of categories on $S^*M$. 
(For a detailed argument, see \cite[Cor. 4.13]{perverse-microsheaves}.)
\end{remark} 

When, as for $S^*M$, the polarization is understood, or e.g. all polarizations in question are 
obtained by restricting some (possibly unspecified) fixed polarization on an ambient space, 
we may omit the polarization from the notation.

\subsection{Quantization of Legendrians and Lagrangians} \label{sec:objects}

Given a set $X$ equipped with the germ of a contact embedding $X \to \hat{\mathcal{U}}_X$,
by a stabilized embedding of $X$ we mean a codimension zero embedding of
$(X \times (0,1)^n, \hat{\mathcal{U}}_X \times T^*(0,1)^n)$ into some contact manifold.

\begin{definition} \label{def: sufficiently isotropic}
We say a subset $X$ of a contact manifold $U$ is {\em sufficiently isotropic} if, after  any high codimension embedding $U \hookrightarrow \R^{2n+1}$, and thickening along any polarization $\sigma$, the resulting $X^\sigma$ is, after some further contact perturbation, spdff (Def. \ref{def: pdff}) and in finite position (Def. \ref{def: finite position}). 

We say a subset of $X$ of an exact symplectic manifold $W$ is sufficiently isotropic if its embedding in the contactization $X \times 0 \subset W \times \R$ is sufficiently isotropic.
\end{definition}

\begin{example}
It follows from \cite[Lemma 3.3.9]{li-nadler-shende} that the core of a Weinstein manifold, when Whitney stratifiable, will be sufficiently isotropic.  We recall also that the Liouville form on a Weinstein manifold can always be perturbed  so the core is Whitney stratifiable \cite{Laudenbach-TS} (and in fact real analytic \cite[Cor. 7.27]{gpsconstructible}).  
\end{example}

\begin{theorem} \label{thm: quantization} 
Let $\mathcal{U}$ be a contact manifold equipped with a stable normal polarization $\tau$. 
Consider a compactly supported contact isotopy
$\phi_t: \mathcal{U} \times (0,1] \to \mathcal{U}$ with  
$\phi_1 = id_U$.  Consider a compact subset $\Lambda_1 \subset \cU$, and let $\Lambda_t := \phi_t(\Lambda)$ 
and let $\Lambda_0 := \lim_{t \to 0} \Lambda_t$. 
Assume: 
\begin{itemize}
\item The family $\Lambda_t$ (including at $t=0$) is gapped (Def. \ref{def:gapped}) for some contact form on $\mathcal{U}$. 
\item $\Lambda_0 \cup \Lambda_1 \subset \mathcal{U}$ is sufficiently isotropic.  
\end{itemize}

Then there is a fully faithful functor
$$\psi: \mush_{\Lambda_1, \tau }(\Lambda_1) \to \mush_{\Lambda_0, \tau} (\Lambda_0)$$

More generally, the same holds for $\Lambda_1$ relatively compact and collared along its boundary, and $\phi_t$
and the contact form witnessing gappedness compatible with the collar. 
\end{theorem} 
\begin{proof}
We will embed the problem and then appeal to Theorem \ref{thm:specialization}.  More precisely, 
fix an (high codimension) embedding 
$\cU \subset J^1(\R^N) \subset S^*(\R^{N+1})$.  Let us write $\pi: \cN_\cU \to \cU$ for the symplectic normal bundle.  Since $N \gg 0$, we may choose a Lagrangian sub-bundle $\mathcal{L}_{\cU} \subset \cN_\cU$.  For any subset $Z \subset \cU$, we have the thickening
$Z^\tau := \pi^{-1}(Z) \cap \mathcal{L}_{\cU}$ in the symplectic normal bundle; we use the same notation for the corresponding locus 
in $J^1(\R^N)$ obtained by transporting along some fixed symplectomorphism of a tubular neighborhood $\cT_\cU$ of $\cU$ with a neighborhood of the zero section in $\cN_\cU$.

Fix any contact form on $U$ compatible with the co-orientation, and a contact form on $\cN_\cU$ with the properties that (1) if $h$ is a contact Hamiltonian on $\cU$ then the projection $\pi$ intertwines the contact flows of $h$ and $h \circ \pi$ and (2) the flow for $h \circ \pi$ preserves $\cL_\cU$.  Such a form on $\cN_\cU$ can be constructed e.g. by identifying $\cN_\cU$ with the fiberwise cotangent bundle of  $\mathfrak{f}: \cL_\cU \to \cU$ and applying Lemma \ref{lem:contrelcot}.  We then extend flows from $U$ to $\cT_\cU$ by pullback of Hamiltonians $h \mapsto h \circ \pi$; this preserves the notion of positivity.  The key feature of this extension procedure is that chords of $Z^\tau$ for any given extended positive isotopy must map to chords of $Z$ for the original isotopy. 
This map of flows depends on the choices of contact forms, but, as the space of such choices is contractible, and the effect on sheaf categories is locally constant by the usual quantization of contact isotopy argument, the choice has no effect on our categorical constructions.

In particular $\Lambda_t^\tau$ remains gapped for an appropriate extended flow, and $\Lambda_0^\tau$ is its limit.  Our `sufficiently isotropic' assumptions mean, by definition, that we may choose the embedding and thickening so  $\Lambda_0^\tau$ is spdff and
$\Lambda_1^\tau, \Lambda_0^\tau$ are in finite position.  

Thus we may appeal to Theorem \ref{thm:specialization} 
to deduce the existence and full faithfulness of some map 
$\psi^\tau: \mush_{\Lambda_1^\tau}(\Lambda_1^\tau) \to \mush_{\Lambda_0^\tau}(\Lambda_0^\tau)$.  
As the natural restriction maps $\mush_{\Lambda_1^\tau}(\Lambda_1^\tau) \to \mush_{\Lambda_1; \tau}(\Lambda_1)$
and $\mush_{\Lambda_0^\tau}(\Lambda_0^\tau) \to  \mush_{\Lambda_0; \tau}(\Lambda_0)$
are equivalences (essentially by definition, and indeed before passing to global sections), we obtain a fully faithful functor
$\psi: \mush_{\Lambda_1; \tau}(\Lambda_1) \to \mush_{\Lambda_0; \tau}(\Lambda_0)$. 

Note the same proof works when $\Lambda_1$ is relatively compact and collared along its boundary,  and $\phi_t$
and the contact form witnessing gappedness are compatible with the collar, using that Theorem \ref{thm:specialization} also holds in this generality.
\end{proof}

\begin{remark}
For any $X \supset \Lambda_0$, 
we may compose with the natural inclusion to obtain a fully faithful map 
$ \psi: \mush_{\Lambda_1, \tau }(\Lambda_1) \to \mush_{X, \tau} (X)$. 
\end{remark}

An exact symplectic manifold $(W, \lambda)$ (or any subset thereof) is canonically embedded 
in its contactization $W \times \R_z$, which we take with contact form $dz - \lambda$.  In particular
we may naturally speak of microsheaves on subsets of exact symplectic manifolds (with stable normal polarizations). 
Note also that if $v_\lambda$ is the Liouville flow, then $v_\lambda + z \frac{d}{dz}$ is a contact vector field
lifting $v_\lambda$; in case $W$ is Liouville, this contact vector field retracts the contactization to the core
$\frc(W)$.

\begin{corollary} \label{cor: weinstein quantization} 
Let $W$ be Liouville with stable normal polarization $\tau$, and let $\Lambda_1 \subset W \times \R$.   Assume the projection $\Lambda_1 \to W$ is an embedding
onto a subset $L$ which is compact. 
Let $L_0$ be the limit of $L$ under the Liouville flow. 
Assume $\Lambda_1 \cup (L_0 \times 0) \subset W \times \R$ is sufficiently isotropic. 
Then there is a fully faithful functor:
$$\mush_{\Lambda_1; \tau}(\Lambda_1) \to \mush_{L_0; \tau}(L_0) \subset \mush_{\mathfrak{c}(W); \tau}(\mathfrak{c}(W))$$
\end{corollary}
\begin{proof} 
Since $L_0$ is invariant under the Liouville flow and compact, we have  $L_0 \subset \mathfrak{c}(W)$, giving the 
asserted inclusion.  
Let  $\Lambda_t$ be the flow of $\Lambda_1$ under $v_\lambda + z \frac{d}{dz}$, parameterized to live over t in $(0,1]$.
Then $\lim_{t \to 0} \Lambda_t = L_0$.  We want to apply Theorem 
\ref{thm: quantization}.  Gappedness is automatic: 
the $\Lambda_t$ have no self-chords
at all, as these would project to self-intersections of (the image under Liouville flow of) $\Lambda_1$, which we have assumed
do not exist.  The other hypotheses of the theorem hold by assumption. 
\end{proof}

In the corollary, we did not require that $L$ was smooth.  
In particular, consider two symplectic primitives $\lambda$ and $\lambda' = \lambda + df$.    Denote the respective cores by 
$\frc$ and $\frc'$.  Note by translating $f \to f + N$ for $n \gg 0$, we may ensure that $\mathfrak{c}$ and $\mathfrak{c}'$ are disjoint. 
If $\mathfrak{c} \sqcup \mathfrak{c}'$ is sufficiently isotropic, 
then using Corollary \ref{cor: weinstein quantization} for $(W, \lambda)$, 
we obtain a fully faithful functor $\mush_{\frc',\tau}(\frc') \hookrightarrow \mush_{\frc,\tau}(\frc)$;
likewise using Corollary \ref{cor: weinstein quantization} for $(W, \lambda')$ we obtain
${\mush}_{\frc, \tau}(\frc) \hookrightarrow {\mush}_{\frc',\tau}(\frc')$.

\begin{remark}
    We expect that sufficient isotropicity of $\mathfrak{c} \sqcup \mathfrak{c}'$ follows from the separate sufficient isotropicity of $\mathfrak{c}$ and $\mathfrak{c}'$, but in any case, since both $\mathfrak{c}$ and $\mathfrak{c}'$ admit ribbons by their very definition, so does the disjoint union, hence if they are (separately) Whitney stratifiable, sufficient isotropicity follows from \cite[Lemma 3.3.9]{li-nadler-shende}. 
\end{remark}

\begin{definition} \label{def: interpolation}
We say a Liouville form $\lambda$ on $W$ is {\em sufficiently Weinstein} if the resulting core is sufficiently isotropic. 
We say $\lambda, \lambda'$ are {\em sufficiently Weinstein cobordant} if there exists
a symplectic primitive $\eta$ for the stabilization $W \times T^*\R$ with sufficiently
isotropic core,\footnote{Note that $(T^*\R, pdq)$ is not a Liouville manifold,
but rather an ``open Liouville sector'' in the sense of \cite{gpssectorsoc}.  
The core remains well defined: $\frc_{pdq}(T^*\R)$ is the zero section.  
Likewise, $(W \times T^* \R, \eta)$ is an open Liouville sector with a well defined core.} such that $\eta$ restricts to 
$\lambda + pdq$ near $-\infty$ and to $\lambda' + pdq$ near $+\infty$.   

More generally, for $\Lambda \subset W^\infty$, we say $\lambda$ is {\em sufficiently Weinstein} for the 
pair $(W, \Lambda)$ if the relative core is sufficiently isotropic; correspondingly we speak of 
sufficiently Weinstein cobordisms.  
\end{definition}

\begin{theorem} \label{thm: invariance} 
Let $\lambda, \lambda'$ be two sufficiently Weinstein forms for $W$ which are sufficiently Weinstein cobordant. 
Then 
the functor ${\mush}_{\frc_{\lambda}(W)}(\frc_{\lambda}(W))  \to {\mush}_{\frc_{\lambda'}(W)}(\frc_{\lambda'}(W))$
is an equivalence. 
\end{theorem} 
\begin{proof}
For notational clarity we write $X = \frc_{\lambda}(W)$ and $X' = \frc_{\lambda'}(W)$.  
We work inside $W \times T^*\R$.   We use $q$ for the coordinate on $\R$, and $p$ for the cotangent
coordinate.  On $W \times T^* \R$, we consider the three Liouville forms  
$$\widetilde{\lambda} = \lambda + pdq \qquad \qquad 
\widetilde{\lambda}' = \lambda' + pdq \qquad \qquad \eta$$ 

The cores for these Liouville forms are evidently 
$X \times T_\R^*\R$, $X' \times T_\R^*\R$, and some $Y$ which projects to $T_\R^* \R$ under the projection
$W \times T^* \R \to T^*\R$  and interpolates between $X$ for $q \ll 0$ and $X'$ for $q \gg 0$. 
 
Flowing by the appropriate Liouville forms give fully faithful maps 

$$ \mush_{X \times T_\R^*\R} (X \times T_\R^*\R)  \leftrightarrows 
\mush_{Y} (Y) \leftrightarrows \mush_{X' \times T_\R^* \R}    (X' \times T_\R^*\R)$$

Restriction at $q = -\infty$ gives canonical maps 
$$\mush_{X \times T_\R^*\R} (X \times T_\R^*\R) \xrightarrow{\sim} \mush_X(X) 
\leftarrow \mush_{Y}(Y) $$  
These maps evidently commute with the $\mush_{X \times T_\R^*\R} (X \times T_\R^*\R)  \leftrightarrows 
\mush_{Y} (Y)$, since those maps are induced by flows constant near $-\infty$.   
It follows formally that the composition 
$\mush_{X \times T_\R^*\R} (X \times T_\R^*\R)  \to  
\mush_{Y} (Y) \to \mush_{X \times T_\R^*\R} (X \times T_\R^*\R)$ is the identity, hence that
the second map $\mush_{Y} (Y) \to \mush_{X \times T_\R^*\R} (X \times T_\R^*\R)$ is essentially
surjective.  As we already knew it was fully faithful, it must be an equivalence; hence
so must its left inverse.  Likewise we learn that $\mush_{Y} (Y) \leftrightarrows \mush_{X' \times T_\R^* \R}    (X' \times T_\R^*\R)$ are equivalences.

Finally, consider the commutative diagram 

$$
\xymatrix{
\mush_{X \times T_\R^*\R} (X \times T_\R^*\R)  \ar[r] \ar[d]_{|_{+\infty}} &  
\mush_{Y} (Y) \ar[d]_{|_{+\infty}} \\
\mush_{X}(X) \ar[r] & \mush_{X'}(X')
}
$$

Here, the vertical maps are the restriction at $q= +\infty$.  The top horizontal map is the nearby 
cycles for flowing $X \times T_\R^*\R$ to $Y$ using the Liouville form $\eta$, and the lower
horizontal map is the restriction of this to $q= +\infty$.  Note this agrees with the nearby cycle
for flowing $X$ to $X'$ using $\lambda'$, by construction.  The diagram commutes because 
nearby cycles is a functor of sheaves of categories, hence commutes with restriction. 
Finally, we have seen that the top, left, and right arrows are equivalences, hence so is the bottom one. 
\end{proof} 

\begin{remark}
It is not completely obvious that the functor obtained by composing the flowdowns
$\mush_{X \times T_\R^*\R} (X \times T_\R^*\R)  \to 
\mush_{Y} (Y) \to \mush_{X' \times T_\R^* \R}    (X' \times T_\R^*\R)$  
agrees with what would be obtained directly by flowing
$\mush_{X \times T_\R^*\R} (X \times T_\R^*\R) \to \mush_{X' \times T_\R^*\R} (X' \times T_\R^*\R)$.
We did not use this assertion in the above argument, though it does follow from what we have shown. 
\end{remark} 

\begin{remark}
As mentioned in the introduction,
 any Liouville homotopy of Weinstein domains  can  in fact be deformed to a sufficiently 
 Weinstein homotopy \cite[Prop. 2.42]{lazarev-sylvan-tanaka}.  This was pointed out to us by 
 Wenyuan Li, who has also shown~\cite{wenyuan}  by another argument that ``sufficiently Weinstein homotopy''
 can be weakened to ``Liouville homotopy'' in the theorem.  This means that the hypothesis is removed entirely, 
 since the linear interpolation $\lambda + sdf$ is such a homotopy.
\end{remark}

\subsection{Collars, ribbons, and Weinstein pairs}
Here we develop a `relative' version of Corollary \ref{cor: weinstein quantization}, suitable for the study of
Weinstein pairs.  For this purpose we need to understand when it is possible to construct contact collars 
for $\Lambda \subset W \times \R$ when $\Lambda$ projects to an eventually conic Lagrangian in the Liouville 
manifold $W$.

Suppose $(D, \lambda)$ is a Liouville domain with contact boundary $\partial D$ with contact form $\alpha = \lambda|_{\partial D}$. 
Note since $\partial D$ comes with a  contact form $\alpha$, it also comes with a contact flow given by integrating the Reeb vector field $v$ satisfying  $\alpha(v) = 1, d\alpha(v) = 0$.

One can find a collared neighborhood of $\partial D \subset D$ of the form  $(1-\epsilon, 1] \times \partial D \subset D$  on which $\lambda = e^r \alpha$. In other words, the  collared neighborhood is the symplectization. Let  $(W, \lambda)$ be the completed Liouville manifold 
obtained by attaching to $D$ the conic end $[1, \infty) \times \partial D$.  

Let $L\subset D$ be a closed exact Lagrangian conic near $\partial D$.  So its end $\partial L := L \cap  \partial D$ is a closed Legendrian.
 Let $f:L\to \RR$ be a primitive so that $df = \lambda|_L$; note that $f$ is locally constant near  $\partial L$.

Let $N = W \times \RR$ be the contactification of $W$ with contact form $dt - \lambda$.  Let $\Lambda\subset N$ be the Legendrian lift of $L$ given by the graph of the primitive $f$. We  write $\partial \Lambda$ for the lift of $\partial L$.

By a {\em weak ribbon} for $\partial L$, we mean a hypersurface $M \subset \partial D$ with $\partial L \subset M$ such that $d \lambda |_M$ is symplectic
and $\partial L$ is invariant for the Liouville flow of $\lambda|_M$.  
In fact, being/having a ribbon depends only on the contact structure on $\partial D$, and not the contact form.  Here, we say 
``weak'' because we do not require that $M$ is in fact a Liouville manifold with skeleton $\partial L$. 

Ribbons allow the construction of collars: 

\begin{prop} \label{prop: collar from ribbon}
Given a weak ribbon $M \subset \partial D$ around $\partial L$, there exists a contact collar of $\Lambda \subset N$. Moreover, the standard translation contact flow on $N$ generated by $\partial_t$ is compatible with the collar.
\end{prop}

\begin{proof}
Our aims are local near $\partial L$. So without loss of generality, we may assume $\partial L$ is connected and the primitive $f$ is zero (as opposed to some other constant) near $\partial L$.

By construction, locally near $M$, the contact manifold $N = W \times \R_t$ results from the contactification of the symplectization of the (negative by our conventions) contactification of $(M, \lambda|_M)$, so in other words
$N = ((M \times \R_z) \times \R_r) \times \R_t$ with contact form
$dt - e^r(dz + \lambda|_M)$.

On the other hand, consider the contactification $M \times \R_w$ with contact form $dw - \lambda|_M$, and the collar $M \times \R_w \times T^*\R $ with contact form $dw - \lambda|_M + p dq$. 

Then near $M$ we have the local contactomorphism
\beq
N \isom M \times \R_w \times T^*\R \qquad
(m, z, r, t) \mapsto (m', w, q, p) = (e^{rZ_M}(m), e^r z + t, r, -e^r z)  
\eeq
where $Z_M$ is the Liouville field of $\lambda|_M$, and $e^{rZ}(m)$ is the time $r$ flow of $m$ under $Z_M$. Indeed, we find
$dw - \lambda|_M + p dq$ pulls back to 
$$d(e^r z + t) - e^r \lambda|_M  -e^r z dr =   e^r dz + dt - e^r \lambda|_M = dt - e^r(dz + \lambda|_M)
$$

Now this local contactomorphism takes $\Lambda$, which lies in $z=0, t = 0$, to points $(\ell, 0, q, 0)$ with $\ell\in \partial L$, exhibiting the collar.
\end{proof}

\begin{corollary} \label{cor: relative weinstein quantization}
In the situation of Corollary \ref{cor: weinstein quantization}, suppose now that $L \subset W$ has conic end.  
Assume, in addition to the hypotheses of Corollary \ref{cor: weinstein quantization}, that the
$\partial^\infty L_t \subset \partial^\infty W$ admit a family of weak ribbons, including at $L_0$.  Then 
the conclusion of Corollary \ref{cor: weinstein quantization} holds: there is a fully faithful functor
$$\mush_{\Lambda_1; \tau}(\Lambda_1) \to \mush_{L_0; \tau}(L_0) \subset \mush_{\mathfrak{c}(W) \cup L_0; \tau}(\mathfrak{c}(W) \cup L_0)$$
\end{corollary} 
\begin{proof}
By Proposition \ref{prop: collar from ribbon}, in this situation we have collars and appropriately compatible structures.  
Thus we may invoke Theorem 
\ref{thm: quantization}. 
\end{proof}

\begin{remark}
Using Corollary \ref{cor: relative weinstein quantization}, 
it is straightforward to extend Theorem \ref{thm: invariance} to the case of Weinstein pairs. 
\end{remark}

\subsection{Conicity on symplectic hypersurfaces} \label{sec: conic}

As the construction of Cor. \ref{cor: weinstein quantization} is built from the Liouville flow, the microsheaves it produces
are automatically conic. 
Let us give some evidence that in fact all microsheaves on exact symplectic manifolds are automatically conic. 

Recall the following stabilizations: 

\begin{itemize}
\item If $(W, \lambda)$ is exact symplectic, then it canonically embeds 
as the zero section of its 
contactization $(W \times \R_z, dz - \lambda )$.  
Conversely, any codimension one 
symplectic hypersurface in a contact manifold has a local neighborhood with this local model. 

\item A contact manifold with contact form $(V, \alpha)$ determines an exact symplectic manifold 
$(V \times \R_t, e^t \alpha)$.  
\end{itemize}

Recall we say a submanifold of a symplectic manifold $Z \subset (W, \omega)$ is co-isotropic if 
$T_z Z$ contains its orthogonal under the symplectic form, i.e. $T_z Z^\perp_{\omega} \subset T_z Z$,
and we say a smooth subset
of a contact manifold $Y \subset (V, \alpha)$ is co-isotropic if $T_y Y^\perp_{d \alpha} \cap \mathrm{ker}\,\alpha \subset T_y Y$.   (This notion depends on the choice of the contact form $\alpha$ only
through the contact distribution $\mathrm{ker}\, \alpha$.) 

\begin{lemma} \label{lem:conic} 
Let $(W, \lambda)$ be an exact symplectic manifold, and $X \subset W$ a smooth 
submanifold.   Then the following are equivalent.
\begin{itemize}
\item If $\tau := TX^\perp_\omega \subset TW|_X$, then $\tau \subset TX$ and $\lambda|_\tau = 0$.
\item $X$ is conic coisotropic. 
\item $X \times 0 \subset (W \times \R_z, dz -  \lambda)$ is contact coisotropic.
\item $X \times 0 \times \R_t \subset (W \times \R_z \times \R_t, e^t(dz - \lambda))$ is coisotropic. 
\end{itemize}
\end{lemma}
\begin{proof}
We will show the first property is equivalent to all others. 

The submanifold $X \subset W$ is coisotropic iff $\tau \subset TX$.  
Conicity of $X \subset W$ is equivalent to asking that the Liouville vector field 
$v = (d \lambda)^{-1} \lambda$ is contained in $TX$ along $X$.  Being contained
in $TX$ is the same as being orthogonal to $\tau$, so a coisotropic is conic iff 
$0 = d\lambda(v, \tau) = \lambda|_\tau$. 

That $X \times 0$ is contact coisotropic in $ (W \times \R_z, dz -  \lambda)$ is asking that 
$\ker( dz -\lambda) \cap \ker(d\lambda |_X)^{\perp, W\times \R}_{d\lambda} \subset TX$. 
Evidently 
$\ker(d\lambda |_{X\times 0})^{\perp, W\times \R}_{d\lambda} = \tau \oplus \frac{\partial}{\partial z}$. 
We are asking when 
$$(\tau \oplus \frac{\partial}{\partial z}) \cap \ker( dz -\lambda) \subset TX$$
As already $\tau \subset TX$ and no vector with a nonzero component in $\frac{\partial}{\partial z}$ 
will be contained in $TX$, this happens if and only if $\lambda|_{\tau} = 0$. 

Finally let us see what it means for 
$X \times 0 \times \R_t \subset (W \times \R_z \times \R_t, e^t(dz - \lambda))$ to be coisotropic. 
We must study the orthogonal complement of $TX \oplus \R \frac{\partial}{\partial t}$ 
with respect to the form $e^t(dt dz - dt \lambda - d\lambda)$.  The orthogonal complement
of $TX$ is $\tau \oplus \R \frac{\partial}{\partial z} \oplus \R \frac{\partial}{\partial t}$.   The orthogonal complement of 
$\R \frac{\partial}{\partial t}$ is the kernel of $dz - \lambda$.  Thus we are asking when 
$$ (\tau \oplus \R \frac{\partial}{\partial z} \oplus \R \frac{\partial}{\partial t} ) \cap \ker(dz - \lambda) \subset TX \oplus \R \frac{\partial}{\partial t}$$
This happens iff $\ker(dz - \lambda) \cap (\tau \oplus \R \frac{\partial}{\partial z})
\subset \tau$, which in turn happens iff $\lambda|_\tau = 0$.
\end{proof}

\begin{remark}
In particular, note that for $X$ a Lagrangian in $W$, we have that 
$X$ is conic Lagrangian iff $d\lambda$ vanishes on $X$ iff $X \times 0$ is
a Legendrian in $(W \times \R_z, dz - \lambda)$. 
\end{remark}

\begin{remark}
Recall from \cite{kashiwara-schapira} that for a closed subset $X$ of a smooth manifold $W$, there 
are two notions of tangent cone, $C(X) \subset C(X, X) \subset TW|_X$.   
When $X$ is smooth, $C(X) = C(X,X) = TX$.  
If $W$ is symplectic, with form $\omega$, one 
says $X$ coisotropic (involutive in \cite{kashiwara-schapira}) if
$$C(X,X)^\perp_\omega \subset C(X)$$

Meanwhile, conicity interacts well with $C(X)$: according to \cite[Lemma 6.5.3]{kashiwara-schapira}, a vector
field preserves $X$ if it is contained in $C(X)$ along $X$. 

However, at least the proof of Lemma \ref{lem:conic}  {\em does not}  
appear to generalize to arbitrary coisotropic subsets. 
More precisely, if we substitute $\tau \to C(X,X)^\perp_{d \lambda}$ and elsewhere $TX \to C(X)$, 
then the first and fourth conditions remain equivalent.  However, $\lambda|_{\tau} = 0$ no longer
seems to imply that $X$ is conic.  Indeed, $\lambda|_{\tau} = 0$ ensures that 
the Liouville vector field is in $\tau^\perp = (C(X,X)^{\perp})^{\perp}$.  But
in general $(C(X,X)^{\perp})^{\perp} \supsetneq C(X)$, and it is in the latter space where 
we need the Liouville vector field to live.
\end{remark}

\begin{corollary} \label{cor: conicity}
For a Liouville manifold $W$ with normal polarization $\tau$, and any object $\cF \in \mush_{W, \tau}(W)$, 
the locus $ss(\cF)$ is conic coisotropic whenever it is smooth.  
\end{corollary}
\begin{proof}
In an embedding $W \times \R \to S^*M$, a local representative for $\cF$ has by \cite[Theorem 6.5.4]{kashiwara-schapira} 
conic co-isotropic microsupport in $T^*M$.  By Lemma \ref{lem:conic}, since this microsupport 
is contained in $W \times 0$ at infinity, it must be conic in $W$. 
\end{proof}

\begin{remark}
One could ask for even a stronger version of conicity, as follows.  On the contactization 
$(W \times \R_z, dz-\lambda)$ of $(W, \lambda)$, one has the contact vector field 
$v + z \frac{d}{dz}$ which lifts the Liouville vector field $v$.   By quantization of contact
isotopy, one has an isomorphism $\Phi_t: e^{tv}_* \mush \cong \mush$.  One could ask for objects
$\cF$ 
which are conic in the sense that there exist isomorphisms $\cF \cong \Phi_t(\cF)$.  
Note that
the full subcategory spanned by these objects is the same the Liouville-equivariant microsheaves,
since $\R$ is contractible. 

In fact however any object with conic microsupport is also conic in this sense.  Indeed, the
quantization of the isotopy would produce an object in $W \times T^* \R$ whose microsupport
was contained in $(\Phi_t(ss(\cF)),t,0)$, the last factor being constant because the microsupport
was already conic.  Such an object is necessarily locally constant in the $t$ direction. 
\end{remark}

\begin{remark} When $W = T^*M$ is a cotangent bundle, there are now two notions of microsheaves
on $W$.  The first is the original $\mush_{T^*M}$.  The second is what we have just defined, using
the fiber polarization.  In fact these are the same: $\mush_{T^*M} = \mush_{T^*M, \mathrm{fibers}}$ as
sheaves of categories.  One can see this by considering the image of the natural map 
$sh(M) \xrightarrow{\pi^*} sh(M \times (0,\infty)) \xrightarrow{i_*} sh(M \times \R)$. 
In particular all objects of the ``new" $\mush$ are indeed conic. 
\end{remark}

\section{Maslov data and descent} \label{sec:descent}


In Section~\ref{sec:polarizable}, we defined and studied microsheaves on contact manifolds $\calU$, and from there exact symplectic manifolds, equipped with a {\em stable normal polarization} $\sigma$. This structure was used at the outset in Section~\ref{sec: embedding} to construct an embedding of a {\em thickening} $\calU^\sigma$  into the contact manifold $\RR^{2N+1\gg 0}$. Here we could use sheaf theory to construct microsheaves on the thickening $\calU^\sigma$ and then prove theorems about them.

In this section, we  take on the independent task of showing microsheaves on  the thickening $\calU^\sigma$ do not depend on the choice of stable normal polarization, but only on its induced Maslov data 
as formulated in Definition~\ref{def:maslov data}. This is a purely homotopical assertion, and we prove its ``universal" version as formulated in the descent of Theorem~\ref{thm: descent}. The main challenge in doing so is the care we must take to show such descent holds for the ``higher" homotopical structures involved.    Thus in contrast with previous sections, the arguments here are $(\infty-)$categorical rather than geometric in nature. 
In particular, only sheaves which are local systems play any essential role.

To say a word about the specific subsections: In \ref{ss:mon univ}, we  restate the dependence of microsheaves on stable normal polarizations in terms of  unstable frame bundles; this is technically convenient in that it allows us to work with finite-dimensional principal bundles. In  \ref{ss: rel indep}, we record choices that allow one to trivialize the dependence of microsheaves on polarizations; in particular, we explain how functors between categories of microsheaves are independent of 
 polarizations. 
 In \ref{sec: morphism}, we use the prior results to construct a universal theory of microsheaves over the classyifying space of stable polarizations; we organize this using $2$-categorical language since key to the later developments is the compatibilty of this universal theory of microsheaves with the group structure of the classifying stack. In \ref{sec: equivariance}, we spell out this compatibility in the language of equivariant geometry;  here one could say we show that microsheaves form a character sheaf over the classifying stack in a suitable sense. 
 Finally, in \ref{ss: forget pol}, we show how this established equivariance provides the sought-after descent.

Throughout, recall we fix at the outset some stable compactly generated symmetric monoidal $\infty$-category $\cC$. 
We write $Pic(\cC)$ for the spectrum of $\cC$-module automorphisms of $\cC$.  For example, 
$Pic(Mod-\Z) = \Z \oplus B\Z^*$.  Here, $\Z^*$ is the group of units in $\Z$, namely $\{1, -1\}$.

We will write $U$ for the stable unitary group, $O$ for the stable orthgonal group, and 
$U/O$ for the stable Lagrangian Grassmannian.

\subsection{Working relative to universal polarization} \label{lgr microsheaves}

As  $\mu sh$ is local in nature, it is natural to expect that, rather than 
demanding the existence of a global polarization, one may work over the space of local 
polarizations, i.e. the Lagrangian Grassmannian of the contact distribution.  
The results of Section \ref{cotangent lagrangian grassmannian} allow us to do this
in a straightforward way. 

Let $(\cU, \xi)$ be a contact manifold, let $\frf: LGr(\xi) \to \cU$ be the Lagrangian Grassmannian of its contact
distribution, and consider the relative cotangent bundle $T^* \frf$.  Per Lemma \ref{lem:contrelcot} 
and Lemma \ref{lem:contcansec},
the space $T^* \frf$ is a contact manifold whose contact distribution carries a Lagrangian distribution
(canonically up to the contractible choice of almost complex structure).  Thus we may define microsheaves
on $T^* \frf$ by the methods of Section \ref{sec: embedding}, and in particular consider 
 the sheaf $\mu sh_{LGr(\xi)}$ 
of microsheaves on $T^* \frf$ supported along the zero section $LGr(\xi) \subset T^*\frf$.   

A Lagrangian distribution $\rho \subset \xi$ is equivalent data to a section $\sigma_\rho: \cU \to LGr(\xi)$ of $\frf$. 
A neighborhood of $\sigma_\rho(\cU)$ inside of $LGr(\xi)$ provides a Lagrangian thickening of $\sigma_\rho(\cU) \subset T^*\frf$. So, by definition, 
there is a stable normal polarization $\eta$ such that 
\begin{equation} \label{polarization pullback} 
\mu sh_{\cU, \eta} = \sigma_\rho^* \mu sh_{LGr(\xi)}
\end{equation}
It is an exercise to show that $\eta$ and $\rho$ are related by Lemma \ref{lem:normalsec}, which we 
henceforth express by writing $\eta = -\rho$.

\subsubsection{Monodromy along Lagrangian Grassmannian}
By  Lemma~\ref{lem: relative constancy},  $\mu sh_{LGr(\xi)}$ is locally constant along the fibers of $\frf: LGr(\xi) \to \cU$.

To isolate the monodromy of  $\mu sh_{LGr(\xi)}$ along the fibers of  $\frf: LGr(\xi) \to \cU$, let us simplify the situation by
 restricting the support of objects to the preimage of 
some smooth Legendrian $L \subset \cU$. 
Denote by $\wt L \subset \frf^{-1}(L) \subset  LGr(\xi)$  the preimage of $L$, and let $\mu sh_{\widetilde{L}} \subset \mu sh_{LGr(\xi)}$ be the subsheaf of microsheaves on $T^*\frf$ supported along $\wt L$.
By Lemma \ref{cor: locally local systems}, $\mu sh_{\widetilde{L}}$ is a local system of categories with stalk 
noncanonically isomorphic to $\cC$.  The monodromies are $\cC$-linear because all sheaf operations are;
e.g.~in principle they can be  computed using the construction of Lemma \ref{lem: isotopy constancy}.
Thus, the sheaf of local isomorphisms  from the constant sheaf
\begin{equation} \frM_{\widetilde{L}} := Isom(\cC_{\widetilde{L}},  \mu sh_{\widetilde{L}}) \end{equation}
is a $Pic(\cC)$-torsor  such that
$\mu sh_{\widetilde{L}} =\cC \times^{Pic(\cC)} \frM_{\widetilde{L}}  $.  

Let us further restrict to the situation where $L$ retracts to a point $\ell \in L$, and we use the tangent polarization $T_\ell L \subset \xi|_\ell$ to give an isomorphism $LGr(\xi)_\ell \simeq U(n)/O(n)$, where  $n = \dim L$.  Then
$\widetilde{L}$ deformation retracts to the fiber, and restricting $\mu sh_{\widetilde{L}}$  to the fiber gives 
a local system of categories $\mush_{U(n)/O(n)}$ and corresponding 
$Pic(\cC)$-torsor
$
\frM_{U(n) / O(n)}.
$

\subsubsection{Stabilization}

Observe that all of the above constructions are  compatible with stabilization. If we fix some $N \gg 0$,
we can similarly consider $\frf_N: LGr(\xi \oplus \CC^N) \to \cU$, and
the corresponding sheaf of categories  $\mush_{LGr(\xi, N)}$, local system of categories $\mush_{U(n+N)/O(n+N)}$, and $Pic(\cC)$-torsor $\frM_{U(n+N)/O(n+N)}$.

Taking $N\to \infty$, 
we write $(U/O)(\xi) := \lim_{N \to \infty} LGr(\xi \oplus \CC^N)$ for the stable Lagrangian Grassmannian bundle, and $\frf_\infty:(U/O)(\xi)\to \cU$ for its projection. 
 Note $U/O:= \lim_{N \to \infty} U(n+ N)/O(n+N)$, and that $\frf_\infty$ is a principle $U/O$-bundle.
We have the corresponding sheaf of categories  $\mush_{(U/O)(\xi)}$, local system of categories $\mush_{U/O}$, and $Pic(\cC)$-torsor $\frM_{U/O}$, whose classifying map we denote by 
\begin{equation} \label{maslov map} 
\frM:  U/O \to BPic(\cC)
\end{equation}

 \subsection{The monodromy is universal} \label{ss:mon univ}

\subsubsection{Working relative to universal frame} 
It will be useful to work relative to the universal frame bundle (rather than polarization bundle), since this is a principle bundle already
before stabilization.\footnote{Indeed, $U(n)/O(n)$ is not a group, although its stabilization $U/O$ is.  
While our final results concern the stabilized setting, our technical foundations are in the finite dimensional 
microsheaf theory.  This makes it technically difficult to directly prove that various statements about $U/O$-equivariance
of microsheaf categories.  Instead we check $U(n)$ equivariance compatibly with a trivialization of the $O(n)$ action 
and then pass to the stabilization.}

Let $\frg: Fr(\xi) \to \cU$ be the frame bundle of the contact
distribution with its natural factorization
\begin{equation} 
\xymatrix{
\frg: Fr(\xi) \ar[r]^-{p} & LGr(\xi) \ar[r]^-\frf & \cU
}
\end{equation}
where $p$ is an $O(n)$-bundle.

Set $\mu sh_{Fr(\xi)} = p^*\mu sh_{LGr(\xi)}$. 
A frame $\tau$ of $\xi$ is equivalent data to a section $\sigma_\tau: \cU \to Fr(\xi)$ of $\frg$. 
By standard identities, we have
\begin{equation}
\mu sh_{\cU, - \ol \tau} = \sigma_\tau^* \mu sh_{Fr(\xi)}
\end{equation}
where we write $- \ol \tau$ for the stable normal polarization associated to  the  polarization  $\ol \tau \subset \xi$ induced by $\tau$, i.e. given by the section $\sigma_{\ol \tau} = p_\cU\circ \sigma_\tau$ of $\frf$.
Similarly, set 
$\mu sh_{U(n)} = p^*\mu sh_{U(n)/O(n)}$ and
$
\frM_{U(n)}= p^*\frM_{U(n) / O(n)}
$, where $p:U(n) \to U(n)/O(n)$ is the quotient map.
Thus $\mu sh_{U(n)}$ is   a local system of categories, $\frM_{U(n)}$ is a 
$Pic(\cC)$-torsor, and we have
$\mu sh_{U(n)} =\cC \times^{Pic(\cC)} \frM_{U(n)}  $.

Note all of the above are naturally $O(n)$-equivariant expressing that they are pulled back from the respective quotient by $O(n)$.

\subsubsection{Local factorization} 

Now we consider the ``local" situation where $\frg: Fr(\xi) \to \cU$ admits a section.

\begin{lemma} \label{lem: local structure of mush} 
Suppose given a frame $\tau$ of $\xi$ with section  $\sigma_\tau$ of $\frg$,
 and denote by $\ol \tau \subset \xi$ the induced polarization  with section  $\sigma_{\ol \tau}$ of $\frf$,
 with  stable normal polarization $- \ol \tau$.
 
Denote by $\widetilde{\sigma} :  \cU\times U(n) \xrightarrow{\sim} Fr(\xi)$ the 
 $U(n)$-equivariant isomorphism satisfying 
$\sigma_\tau = \widetilde{\sigma}|_{\cU \times 1}$. 
Then there is a resulting canonical   isomorphism of $O(n)$-equivariant sheaves 
$$\widetilde \sigma^* \mush_{Fr(\xi)} \simeq \mush_{\cU, -\ol \tau} \boxtimes_\cC \mush_{U(n)}$$ 
\end{lemma}

\begin{proof} To simplify the notation, we may use $\wt \sigma$ to trivialize $\xi$ and  $Fr(\xi)$, and thus from the start work with the diagram of projections
$$
\xymatrix{
\frg: \cU\times U(n)  \ar[r]^-{p} & \cU\times U(n)/O(n)  \ar[r]^-\frf & \cU
}
$$
So we seek an
$O(n)$-equivariant  isomorphism 
 $$ \mush_{\cU\times U(n)} \simeq \mush_{\cU, -\ol \tau} \boxtimes_\cC \mush_{U(n)}$$ 
where $ \mush_{\cU\times U(n)} $ is defined with respect to the stable normal   polarization 
provided by the polarization: 
$$\ol g \boxtimes \nu \subset \CC^n \boxtimes T_g^* U(n),
\qquad (u, g) \in \cU\times U(n), 
$$ where $\ol g \subset \CC^n\simeq \xi$ is the polarization provided by the frame $g\in U(n)$, and $\nu \subset T_g^* U(n)$ is the tautological polarization.

Consider the contact manifold $\cU\times \CC^n \times  T^* U(n)$ with support condition $\cU \times \RR^n \times U(n)$, and polarization
$$\ol g \boxtimes \RR^n \boxtimes \nu \subset \CC^n \boxtimes T_g^* U(n),
\qquad (u, z, g) \in \cU\times \CC^n\times U(n).
$$
Let $\mush_{\cU\times \RR^n \times U(n)}$ be the corresponding sheaf of categories, and note over
$\cU \times U(n)$
the evident isomorphism 
$$ 
\mush_{\cU\times \RR^n \times U(n)} \simeq  \mush_{\cU\times U(n)} $$

Observe the corresponding stable normal   polarization is also represented by 
$$
\ol 1 \boxtimes g (\RR^n) \boxtimes \nu \subset \CC^n \boxtimes T_g^* U(n),
\qquad (u, z, g) \in \cU\times \CC^n\times U(n).
$$
 where  $\ol 1 \subset \CC^n\simeq \xi$ is the polarization provided by the tautological frame $1\in U(n)$.
Thus parsing the last two terms of the product together,  over
$\cU \times U(n)$ we have
the canonical isomorphism 
$$ 
\mush_{\cU\times \RR^n \times U(n)} \simeq  \mush_{\cU, -\ol \tau} \boxtimes_\cC \mush_{U(n)}
$$
The composition of the above isomorphisms gives the asserted isomorphism. Note all of the constructions are evidently $O(n)$-equivariant. 
\end{proof}

\subsubsection{Return to stabilization}

Observe that all of the above constructions are  compatible with stabilization. If we fix some $N \gg 0$,
we can similarly consider $\frg_N: Fr(\xi \oplus \CC^N) \to \cU$, and
the corresponding sheaf of categories  $\mush_{Fr(\xi, N)}$, local system of categories $\mush_{U(n+N)}$, and $Pic(\cC)$-torsor $\frM_{U(n+N)}$.

Taking $N\to \infty$, 
we write $U(\xi) := \lim_{N \to \infty} Fr(\xi \oplus \CC^N)$ for the stable frame bundle, and $\frg_\infty:U(\xi)\to \cU$ for its projection. 
 Note $U:= \lim_{N \to \infty} U(n+ N)$, and that $\frg_\infty$ is a principle $U$-bundle
with factorization \begin{equation} 
\xymatrix{
\frg: U(\xi) \ar[r]^-{p_\infty} & (U/O)(\xi) \ar[r]^-{\frf_\infty} & \cU
}
\end{equation}
We have the corresponding sheaf of categories  $\mush_{U(\xi)} \simeq  p_\infty^*\mush_{(U/O)(\xi)}$, local system of categories $\mush_{U} \simeq  p_\infty^*\mush_{U/O} $, and $Pic(\cC)$-torsor $\frM_{U} 
 \simeq  p_\infty^*\frM_{U/O}$. 

Now  Lemma~\ref{lem: local structure of mush} immediately implies:

\begin{cor} \label{cor: local structure of mush}  
Suppose for some $N\gg 0$, given a frame $\tau$ of $\xi \oplus \CC^N$ with section  $\sigma_\tau$ of $\frg_N$,
 and denote by $\ol \tau \subset \xi$ the induced polarization of $\xi \oplus \CC^N$ with section  $\sigma_{\ol \tau}$ of $\frf_N$.
 
Denote by $\widetilde{\sigma} :  \cU\times U \xrightarrow{\sim} U(\xi)$ the 
 $U$-equivariant isomorphism satisfying 
$\sigma_\tau = \widetilde{\sigma}|_{\cU \times 1}$. 
Then there is a resulting canonical   isomorphism of $O$-equivariant sheaves 
$$\widetilde \sigma^* \mush_{U(\xi)} \simeq \mush_{\cU, -\ol \tau} \boxtimes_\cC \mush_{U}$$ 
\end{cor}

\subsection{Relative independence of polarization} \label{ss: rel indep} Fix a symplectic vector space $V$.  
We regard it as an exact symplectic manifold carrying the radial Liouville vector field, i.e.
with exact symplectic form $\sum x_i d y_i - y_i d x_i$.  In particular, all linear Lagrangians
are conic.
Note that a polarization on $V$ is determined (up to contractible choice) 
by its value at the origin, which is again a single linear Lagrangian in $V$. 

Given linear Lagrangians $L, P \subset V$, we
write $\mush(L; P) := \mush_{L, P}(L) = (\mush_{L, P})_0$ for the category of microsheaves
on $V \times \R$ defined using the polarization $P$ and supported on $L \times 0 \subset V \times \R$. 
As we discussed in the previous section, the 
category $\mush(L; P)$ is noncanonically isomorphic to $\cC$, and can have monodromy 
as $P$ varies in the Lagrangian Grassmannian.  However:

 \begin{lemma}\label{lem:pol indep} Let $V$ be a finite-dimensional hermitian vector space. 
 Given Lagrangians $L_0, L_1, P \subset V$, the category of ($\cC$-linear and colimit preserving) functors 
 \beq
 \Fun^L_{\cC}(\mush(L_0; P), \mush(L_1; P))
 \eeq
 is  independent of $P$. More precisely, the natural locally constant sheaf of categories $\cF$ over $Lag(V)$ with stalk  
 \beq
\cF|_P =  \Fun^L_{\cC}(\mush(L_0; P), \mush(L_1; P)), \qquad P\in Lag(V),
\eeq
   is in fact constant.
 \end{lemma}

We need the following general fact:

\begin{lemma} \label{closed monoidal homs}
 Let $T$ be a connected topological space, $\cQ$ a closed monoidal category with unit $1_\cQ \in \cQ$, 
 and $\cE$  a locally constant sheaf on $T$ valued in $\cQ$ with invertible stalks.
 Then for any $t, s \in T$, there's a non-canonical isomorphism
 $\Hom(\cE_t, \cE_s) \cong 1_{\cQ}$. 
\end{lemma} 
\begin{proof} 
Recall that a closed monoidal category has by definition internal hom objects and a monoidal structure
satisfying the tensor hom adjunction.  
 Let
 $T \xrightarrow{\pi} t \xrightarrow{i} T$ be the projection to and inclusion of the point $t$. 
 Then 
 $\underline{\Hom}(\cE_t, \cE_s) = \underline{\cH om}(\pi^* i^* \cE, \cE)_s = (\cE_t^\vee \otimes \cE)_s$
 is the stalk of a locally constant sheaf.  The stalk at $s = t$ is evidently $1_{\cQ}$. 
\end{proof}

 \begin{proof}[Proof of Lemma \ref{lem:pol indep}]
 We apply Lemma \ref{closed monoidal homs} with:
 \begin{itemize}
 \item $T = Lag(V)$ with $t = L_0$ and $s=L_1$,
 \item $\cQ$ is the category of locally constant sheaves on $Lag(V)$ valued in $\cC$, with closed monoidal 
 structure inherited from $\cC$.   
 (This copy of $Lag(V)$ will parameterize the polarization $P$, so should not be confused with $T = Lag(V)$.)
 Note that 
 locally constant sheaves on $T$ valued in $\cQ$ are the same as locally 
 constant sheaves on $T \times Lag(V)$ valued in $\cC$. 
 \item $\cE$ is the locally constant sheaf on $T \times Lag(V)$ with stalk at $(L, P)$ given by $\mush(L; P)$
\end{itemize} 
The conclusion of the lemma now follows from the fact that the monoidal unit in $\cQ$ is the constant sheaf on 
$Lag(V)$ with stalk the monoidal unit $1_\cC \in \cC$. 
\end{proof}

As always, we have compatibility with stabilization, as we record in:

 \begin{lemma}\label{lem:real indep} Let $V$ be a finite-dimensional symplectic vector space, and $W_\RR$ a 
  finite-dimensional real vector space with complexification $W = W_\RR \otimes_\RR \CC$, carrying the standard
  symplectic form.
  
 Given Lagrangians $L_1, L_2 \subset V$, there is a natural equivalence of  functor categories compatible with composition
 \beq
 \Fun^L_{\cC}(\mush(L_1), \mush(L_2 ))
\simeq  \Fun^L_{\cC}(\mush(L_1 \oplus W_\RR), \mush(L_2  \oplus W_\RR))
 \eeq
 from microsheaves on $L_1 \oplus W_\RR \subset V \oplus W$ to microsheaves on on $L_2 \oplus W_\RR \subset V \oplus W$ with respect to any polarization as in Lemma~\ref{lem:pol indep}. 
 \end{lemma}
 
\subsection{A local system of 2-categories over $B(U/O)$}  \label{sec: morphism}

 \begin{definition} 
Let $V$ be a symplectic vector space.  We write 
$\frL_V$ for the $2$-category whose objects are linear Lagrangians $L \subset V$, 
and whose morphisms are the $\cC$-linear colimit preserving functors 
$
 \Fun^L_{\cC}(\mush(L_1), \mush(L_2)).
 $
Here we use Lemma~\ref{lem:pol indep} to suppress any choice of polarization.
\end{definition} 

We write $B\cC$ for the $2$-category on one object, whose morphisms are given by the monoidal category $\cC$. 
Observe that the full subcategory on given Lagrangian $L \in \frL_V$ is canonically isomorphic to $B \cC$, and 
correspondingly $\frL_V$ is noncanonically isomorphic to $B\cC$.

Recall that a model for $BU(n)$ is given by taking the $k$-simplices to be $(C^\infty)$  
symplectic vector bundles of dimension $2n$ over $(C^\infty)$  $k$-simplices, with the 
obvious face and degeneracy maps.  Then $V \mapsto \frL_V$ can be regarded as the $0$-simplex data 
of a local system of $2$-categories
over $BSp(2n, \R)$.   We extend it to the higher simplices as follows.
Consider a $k$-simplex of $BU(n)$ given by a symplectic vector bundle $V_{\Delta_k}$ over $\Delta_k$. 
Objects in $\frL_{V_{\Delta_k}}^\circ$ are smooth Lagrangian sub-bundles in $V_{\Delta_k}$.  To define
morphisms, consider the pullback of $V_{\Delta_k}$ to the total space of $T^* \Delta_k$; denote it $V_{T^*\Delta_k}$. 
On the total space of $V_{\Delta_k}$, there is a $1$-form $\lambda_V$ which vanishes on any vector pulled back
from the base, and whose restriction to any fiber is the $1$-form giving rise to the radial Liouville structure used in 
the previous subsection.  The 
total space of $V_{T^*\Delta_k}$ carries a $1$-form $\lambda = \lambda_V + \lambda_{T^*\Delta_k}$
(we suppress the notation for pullbacks).  Consider the natural inclusion 
$i: V_{\Delta_k} \subset V_{T^*\Delta_k}$.
  We claim that if $L \subset V_{\Delta_k}$ is a Lagrangian sub-bundle, 
then $i(L)$ is a conic Lagrangian in $V_{T^*\Delta_k}$.
Using these conic Lagrangians, we proceed
as before with the definition of $\frL$.  The face maps are defined by pullback along inclusion of fibers, and
are isomorphisms by Lemma \ref{lem: relative constancy}.  Denote the resulting 
local system of $2$-categories by $\frL(n)$.  
By compatibility with stabilization (Lemma \ref{lem:real indep}), this is the pullback
of a local system of $2$-categories $\frL$ on $BU$.  

We recall that local systems of categories locally isomorphic to $B\cC$ are classified by maps to $B^2 Pic(\cC)$.  
Thus we have defined maps $\frL(n): BU(n) \to B^2 Pic(\cC)$ which stabilize to $\frL: BU \to B^2 Pic(\cC)$. 

One defect of the above description of $\frL: BU \to B^2 Pic(\cC)$ as a limit is that, while the source and target
are commutative $(E_\infty)$, it is not obvious that $\frL$ respects these structures.  In particular, we would like to 
argue descent to $B(U/O)$ by trivializing $\frL|_{BO}$, and for this argument it is helpful to know that $\frL$ is
itself a group homomorphism.  Here we resolve this issue.  

Consider the spaces

$$M_\CC = \coprod_{n\geq 0} BU(n) \qquad \qquad M_\RR = \coprod_{n\geq 0} BO(n)$$
These classify respectively finite-rank complex and  real vector bundles. The direct sum of vector bundles equips  $M_\CC$ and $M_\RR$ 
with the structure of commutative monoids in spaces.  Note  that $\pi_0$ of either is the commutative monoid of natural numbers $\NN = \{0, 1, 2,\ldots, \}$. 

\begin{proposition}
$\bigoplus \frL(n) : \coprod_{n\geq 0} BU(n) \to B^2 Pic(\cC)$ is a map of commutative monoids. 
\end{proposition} 
\begin{proof}
Recall that the map $\frL(n)$ was defined in terms of as the classifying map of a certain sheaf of $2$-categories, also called $\frL(n)$, 
over $BU(n)$.  
Recall that simplices of $BU(n)$ are symplectic vector bundles over simplices.  The monoidal structure
$a_{n,m}: BU(n) \times BU(m) \to BU(n + m)$ is determined by sending 
$V_{\Delta} \times V'_{\Delta'} \mapsto (V \oplus V')_{\Delta \times \Delta'}$.  
Thus our task is to give isomorphisms $i_{V, V'}: \frL_{V} \otimes \frL_{V'} \mapsto \frL_{V\oplus V'}$, compatible with 
the commutativity and associativity of direct sums of vector spaces.  

Here, $\frL_{V} \otimes \frL_{V'}$ is the 2-category whose objects are formal symbols $L \otimes L'$ for $L \in \frL_V$ and $L' \in \frL_{V'}$, 
and whose morphisms are 
\begin{equation} \label{tensor morphisms} 
\Hom_{\frL_{V} \otimes \frL_{V'}}(L \otimes L', M \otimes M') := \Hom_{\frL_V}(L, M) \otimes \Hom_{\frL_{V'}}(L', M')
\end{equation} 
where the tensor product here 
is of $\cC$ module categories. 

We define $i_{V, V'}(L \otimes L') = L \oplus L'$.  On morphisms, we have a canonical identification 
$$\Hom_{\frL_{V}}(L, M) \otimes \Hom_{ \frL_{V'}}(L', M') = \Fun_{\cC}^L( \mush(L), \mush(M) ) \otimes  \Fun_{\cC}^L( \mush(L'), \mush(M') )$$
$$ = \Fun_{\cC}^L(\mush(L \oplus L'), \mush(M \oplus M')) = \Hom_{\frL_{V \oplus V'}}(L \oplus L, M \oplus M') $$

The only commutativity in question is the commutativity of the tensor product on the RHS of \eqref{tensor morphisms}; 
recall that if $\cC$ is $E_{n}$ then the monoidal structure on $\cC$-modules is $E_{n-1}$.  
(We have assumed throughout that our $\cC$ is $E_\infty$.) 
\end{proof}

One calls a commutative monoid  in spaces $M$ group-like if the commutative monoid $\pi_0(M)$ has inverses. The inclusion of group-like commutative monoids into all commutative monoids has a left adjoint called the group-completion denoted by $M\mapsto M^{gp}$. Note the natural isomorphism $\pi_0(M^{gp}) \simeq \pi_0(M)^{gp}$.

\begin{lemma}\label{lem:gp}
There are natural equivalences of group-like commutative monoids $M_\CC^{gp} \simeq BU \times \ZZ$ and $M_\RR^{gp} \simeq BO \times \ZZ$.  
\end{lemma}

\begin{proof}
The proof is the same in both cases and we will write $M$ for either $M_\CC$ or $M_\RR$.  We seek to show 
 $M^{gp} \simeq BGL_\oo \times \ZZ$ using that $GL_\oo(\CC) \simeq U$ and   $GL_\oo(\RR) \simeq O$.

Let $L \in M$ classify the trivial rank one vector space. 
Consider the filtered colimit
\beq
\xymatrix{
 M[-L] := \colim (M \ar[r]^-{\oplus L} & M \ar[r]^-{\oplus L}  &  M \ar[r]^-{\oplus L} &\cdots )
}\eeq
Note $ M[-L]$ is the colimit of a diagram of $M$-spaces so is itself an $M$-space.
By the Group Completion Theorem\footnote{See for example Corollary 7 of \url{https://www.uni-muenster.de/IVV5WS/WebHop/user/nikolaus/Papers/Group_completion.pdf}}, the natural map is an equivalence of $M$-spaces
\beq
\xymatrix{
 M[-L] \ar[r]^-\sim & M^{gp}
}\eeq
Note $\pi_0( M[-L]^{gp} ) \simeq  \pi_0( M^{gp} ) \simeq  \pi_0( M )^{gp} \simeq \NN^{gp} \simeq \ZZ$.

Recall filtered colimits commute with fiber products. Thus we can calculate the based loop spaces at any component
\beq
\Omega_{L^{\oplus m}} M[-L] \simeq \colim_n GL(L^{\oplus m} \oplus L^{\oplus n})  \simeq GL_\oo
\eeq
where the  maps are induced by $L^{\oplus m} \oplus L^{\oplus n} \simeq L^{\oplus m} \oplus L^{\oplus n} \oplus \{0\} \subset 
L^{\oplus m} \oplus L^{\oplus n} \oplus L$. 

We conclude  $M^{gp} \simeq BGL_\oo \times \ZZ$ as asserted.
\end{proof}

\begin{corollary}
 $\frL : M_\CC \to B^2\Pic(\cC)$ descends to an $E_\infty$ map
$\frL :BU \to B^2\Pic(\cC)$.  
\end{corollary}
\begin{proof}
Being a map from a monoid to a group, $\frL: M_\CC \to B^2\Pic(\cC)$ factors through the group completion.  
\end{proof} 

\begin{theorem} \label{delooped maslov map}
$\frL : BU \to B^2\Pic(\cC)$ descends to an $E_\infty$ map $\frL :B(U/O) \to B^2\Pic(\cC)$.  
\end{theorem}
\begin{proof}
We should give 
a trivialization of $M_\RR \xrightarrow{\otimes \CC} M_\CC \xrightarrow{{\frL}} B^2\Pic(\cC)$.  This is provided by 
Lemma \ref{lem:real indep}. 
\end{proof}

Recall the map $\frM: U/O \to B\Pic(\cC)$ defined in Section \ref{lgr microsheaves}.  
\begin{proposition} \label{loop deloop maslov map}
$\frM = \Omega \frL$. 
\end{proposition} 
\begin{proof}
The map $\frM$ classified a certain local system of categories $\mush_{U/O}$ locally equivalent to $\cC$.  
Returning to its definition, we find $\Omega \frL$ also classifies such a local system $F_{\Omega \frL}$ : 
at a point  $L \in  U/O$ represented by $A \in U_n$,  we have 
$(F_{\Omega \frL})_A = \Fun^L_{\cC}(\mush(\RR^n), \mush(A(\RR^n)))$. 
We have suppressed the polarization by Lemma~\ref{lem:pol indep}; now let us choose any fixed $P$ and moreover
fix any isomorphism $\cC \cong \mush(\RR^n; P)$.  
This identifies our stalks as simply $(F_{\Omega \frL})_A = \mush(A(\RR^n); P)$. 
Thus $F_{\Omega \frL}$ agrees with $\mush_{U/O}$, giving the result. 
\end{proof} 

\begin{remark}
In particular, we see that $\frM$ itself can be given an $E_\infty$ structure. 
One could contemplate 
showing this directly, and then just defining $\frL = B\frM$, rather than go through the above discussion of $2$-categories
and group completions.  This is presumably possible, but one must face the difficulty that the commutative structure
on $U$ is nontrivial to describe, being in particular not the limit of commutative structures on the finite $U(n)$.  One such description
begins: given $A \in U(n)$ and $B \in U(m)$, there's a path carrying $B \oplus \R^\infty \rightsquigarrow 1_n \oplus B \oplus \R^\infty$,
and $A \oplus \R^\infty$ visibly commutes with $1_n \oplus B \oplus \R^\infty$.  One would have to contemplate all these
paths and the corresponding higher homotopies.  

By contrast, our approach above accesses the monoidal structure on $BU$ directly through the manifestly commutative direct sum of 
vector spaces.  As the above description makes clear, this  is anyway the source of the commutative structure on $U$. 
\end{remark}

\subsection{Equivariance} \label{sec: equivariance}
Let $F_{can}$ be the local system of $\cC$-linear categories  on $BPic(\cC)$
classified by the identity map $BPic(\cC) \to BPic(\cC)$.  So $F_{can} = EPic(\cC) \times^{Pic(\cC)} \cC$.

First, consider the addition map $a_{BPic(\cC)}: BPic(\cC) \times BPic(\cC) \to BPic(\cC)$. Recall -- as with any such addition map -- it classifies the 
diagonal quotient 
 $$
 a_{BPic(\cC)}^*EPic(\cC) \simeq EPic(\cC)\times^{Pic(\cC)} EPic(\cC)
 $$
 and hence pulls back the  local system $F_{can}$ of $\cC$-linear categories  by
  $$
 a_{BPic(\cC)}^*F_{can} =  a_{BPic(\cC)}^*(EPic(\cC) \times^{Pic(\cC)} \cC) \simeq
   a_{BPic(\cC)}^*(EPic(\cC)) \times^{Pic(\cC)} \cC \simeq
 $$
 $$
( EPic(\cC)\times^{Pic(\cC)} EPic(\cC)) \times^{Pic(\cC)} \cC 
\simeq 
( EPic(\cC)\times^{Pic(\cC)} \cC) \boxtimes_{\cC} (EPic(\cC) \times^{Pic(\cC)} \cC)
 \simeq  F_{can} \boxtimes_\cC F_{can}
  $$

(Here we write $\boxtimes_\cC$ rather than just $\boxtimes$ only for emphasis: all occurrences in this
article of the $\boxtimes$ of $\cC$-linear sheaves of categories have been $\boxtimes_{\cC}$.)

Similarly, for any composite addition $a_n: BPic(\cC)^{\times n}  \to BPic(\cC)$, with any parenthesizing or  ordering,
there is a tautological isomorphism $a_n^* F_{can} \simeq \boxtimes_\cC^{n} F_{can}$, compatible with 
 any re-parenthesizing or  re-ordering. In particular, we have the following:

\begin{proposition} \label{canonical equivariance} 
 The $\cC$-linear box products
 $ a_n^* F_{can} \simeq\boxtimes_\cC^{n} F_{can}$ 
provide a sheaf on the augmented
simplicial space
$$
\xymatrix{
BPic(\cC) & \ar[l]_-{a_{BPic(\cC)}}  BPic(\cC)  \times BPic(\cC) & \ar@<-0.5ex>[l]  \ar@<0.5ex>[l]  BPic(\cC) \times BPic(\cC) \times  BPic(\cC)  & \ar@<0ex>[l]  \ar@<-0.75ex>[l]  \ar@<0.75ex>[l]      \cdots 
}
$$
\end{proposition} 

Next, recall from \eqref{maslov map} the
map $\frM: U/O \to BPic(C)$; per its definition, it provides a canonical equivalence 
$\frM^* F_{can} \simeq \mush_{U/O}$.  
We have also learned that $\frM$ is a group homomorphism, and thus can regard it as 
giving a  $U/O$-action on $BPic(\cC)$.

Consider the action map $a: U/O \times BPic(\cC) \to BPic(\cC)$ and note -- as with any such action map -- it fits into a commutative diagram
\beq
\xymatrix{
\ar[d]_-{\frM \times \id}  U/O \times BPic(\cC) \ar[r]^-a &  BPic(\cC)\ar[d]^-\id\\
BPic(\cC) \times BPic(\cC) \ar[r]^-{ a_{BPic(\cC)}} &  BPic(\cC)
}
\eeq
Hence
 it pulls back the  local system $F_{can}$ of $\cC$-linear categories  by
  $$
 a^*F_{can} 
 \simeq  (\frM \times \id)^* a_{BPic(\cC)}^* F_{can} \simeq  \frM^*F_{can} \boxtimes_\cC F_{can} \simeq  \mush_{U/O} \boxtimes_\cC F_{can}
  $$
  
Similarly, for  for any composite addition and action $a_{m, n}: U/O^{\times m} \times BPic(\cC)^{\times n}  \to BPic(\cC)$, with any parenthesizing or  ordering,
there is a tautological isomorphism $$
a_{m, n}^* F_{can} \simeq (\boxtimes_\cC^{m} \frM) \boxtimes_\cC  (\boxtimes_\cC^{n} F_{can})
$$ compatible with 
 any re-parenthesizing or  re-ordering. In particular, we have the following:

\begin{corollary} \label{mush character}
 The $\cC$-linear box products
 $a_{m, 1}^* F_{can} \simeq (\boxtimes_\cC^{m} \frM) \boxtimes_\cC F_{can}$ 
provide a sheaf on the augmented
simplicial space
$$
\xymatrix{
BPic(\cC) & \ar[l]_-{a}  U/O \times BPic(\cC) & \ar@<-0.5ex>[l]  \ar@<0.5ex>[l]  U/O \times U/O \times  BPic(\cC)  & \ar@<0ex>[l]  \ar@<-0.75ex>[l]  \ar@<0.75ex>[l]      \cdots 
}
$$

The $\cC$-linear box products
 $a_{m, 0}^* F_{can} \simeq \boxtimes_\cC^{m} \frM $ 
provide a sheaf on the augmented
simplicial space
$$
\xymatrix{
U/O & \ar[l]_-{a_{U/O}}  U/O \times U/O & \ar@<-0.5ex>[l]  \ar@<0.5ex>[l]  U/O \times U/O\times  U/O & \ar@<0ex>[l]  \ar@<-0.75ex>[l]  \ar@<0.75ex>[l]      \cdots 
}
$$

If we take the sheaf on the first diagram and forget the action maps on each term, then we recover the sheaf on the second diagram.  
\end{corollary}

Now let $(\cU, \xi)$ be a contact manifold, and $(U/O)(\xi)$ the stable Lagrangian Grassmannian of the contact distribution.  Consider the action map $a_\xi: U/O \times (U/O)(\xi) \to (U/O)(\xi)$.

For the moment, let us work locally in $\cU$ and choose a local trivialization $\sigma: \cU \to (U/O)(\xi)$. This provides a tautological $U/O$-equivariant identification  
$(U/O)(\xi) \simeq_\sigma U/O \times  \cU$ and a commutative diagram
$$
\xymatrix{
\ar[d]_-{a_\xi} U/O \times (U/O)(\xi)  & \simeq_{\id \times \sigma} &  U/O \times U/O \times  \cU \ar[d]^-{a_{U/O} \times \id} \\
(U/O)(\xi)  & \simeq_\sigma &  U/O \times  \cU
}
$$
By Corollary \ref{cor: local structure of mush},  this provides 
  an  equivalence
 $
\mush_{U/O(\xi)} 
 \simeq \mush_{U/O} \boxtimes_\cC\mush_{\cU, \sigma}
$ and a resulting composite  equivalence
$$
a_\xi^*\mush_{U/O(\xi)} 
 \simeq a_{U/O}^*\mush_{U/O}  \boxtimes_\cC\mush_{\cU, \sigma} \simeq
\mush_{U/O}   \boxtimes_\cC \mush_{U/O}  \boxtimes_\cC\mush_{\cU, \sigma}
\simeq 
\mush_{U/O}   \boxtimes_\cC \mush_{U/O(\xi)}
  $$
 Unwinding the constructions, for any map $g:\cU \to U/O$, we obtain a canonically isomorphic composite  equivalence if we consider the trivialization $g\sigma: \cU \to (U/O)(\xi)$. In particular, there are commutative diagrams in which we have a canonical isomorphism of the factored terms
 $$
 \mush_{U/O}   \boxtimes_\cC g^*\mush_{U/O}  \boxtimes_\cC\mush_{\cU, g\sigma}
\simeq
$$
$$
\mush_{U/O}   \boxtimes_\cC (\mush_{U/O} \otimes_{\cC} \mush_{U/O}|_{g^{-1}})  \boxtimes_\cC (\mush_{\cU, \sigma} \otimes_{\cC} \mush_{U/O}|_g)  
$$
$$\simeq
 \mush_{U/O}   \boxtimes_\cC \mush_{U/O}  \boxtimes_\cC\mush_{\cU, \sigma}
 $$
 Similar considerations hold for coherences for higher simplices of sections.

Hence altogether
  the pullback of  the  local system $\mush_{U/O(\xi)}$ along the  action map $a_\xi: U/O \times (U/O)(\xi) \to (U/O)(\xi)$ gives a canonical equivalence of
  of $\cC$-linear categories
  $$
 a_\xi^*\mush_{U/O(\xi)} 
 \simeq \mush_{U/O} \boxtimes_\cC\mush_{U/O(\xi)}
  $$
  
Similarly, for  for any composite action $a_{m}: U/O^{\times m} \times U/O(\xi)  \to U/O(\xi)$, with any parenthesizing,
there is a tautological isomorphism $$
a_{m}^*\mush_{U/O(\xi)} \simeq (\boxtimes_\cC^{m}  \mush_{U/O} )   \boxtimes_\cC \mush_{U/O(\xi)} 
$$ compatible with 
 any re-parenthesizing. In particular, we have the following:

\begin{proposition} \label{polarization equivariance} 
The $\cC$-linear box products
 $a_{m}^*\mush_{U/O(\xi)} \simeq (\boxtimes_\cC^{m}  \mush_{U/O} )   \boxtimes_\cC \mush_{U/O(\xi)} $ 
provide a sheaf on the augmented
simplicial space
$$
\xymatrix{
U/O(\xi) & \ar[l]_-{a_{\xi}}  U/O \times U/O(\xi) & \ar@<-0.5ex>[l]  \ar@<0.5ex>[l]  U/O \times U/O \times  U/O(\xi)  & \ar@<0ex>[l]  \ar@<-0.75ex>[l]  \ar@<0.75ex>[l]      \cdots 
}
$$

If we take the sheaf on the  diagram and forget the action maps on each term, then we recover the sheaf on the second diagram
$$
\xymatrix{
U/O & \ar[l]_-{a_{U/O}}  U/O \times U/O & \ar@<-0.5ex>[l]  \ar@<0.5ex>[l]  U/O \times U/O\times  U/O & \ar@<0ex>[l]  \ar@<-0.75ex>[l]  \ar@<0.75ex>[l]      \cdots 
}
$$
of Corollary~\ref{mush character}. 
\end{proposition}

\subsection{Forgetting the polarization} \label{ss: forget pol}
\begin{definition}\label{def:maslov data}
Fix a contact manifold $(\mathcal{U}, \xi)$,  
and consider 
\begin{equation} \label{Maslov obstruction}
\mathcal{U} \xrightarrow{\xi} BU \to B (U/O) \xrightarrow{\frL} B^2 Pic(\mathcal{C})
\end{equation} 
We call the composite map $\mathcal{U} \to B^2 Pic(\mathcal{C})$ the ($\mathcal{C}$-valued) {\em Maslov obstruction}.
By {\em Maslov data}, we mean a choice of a null-homotopy of this map. 
\end{definition} 

The composite $\mathcal{U} \to B (U/O) $ classifies the stable Lagrangian Grassmannian bundle previously
denoted $(U/O)(\xi)$.  The further composite $\mathcal{U} \to B^2 Pic(\mathcal{C})$ classifies
a $B Pic(\cC)$ bundle we will similarly denote $B Pic(\cC)(\xi)$.

Let $K = \ker(\frM:U/O \to BPic(\cC))$ be the kernel of the group homomorphism. We have  
$$B Pic(\cC)(\xi) 
\simeq K\bs (U/O)(\xi) $$  
where we take the quotient for the action along $K \to U/O$. 
Note that Maslov data is equivalent to a section of the bundle $B Pic(\cC)(\xi)  \to \cU$. 

Consider  the  natural quotient map
map
$$\frM_\xi :  (U/O)(\xi) \to K\bs (U/O)(\xi)   \simeq BPic(\cC)(\xi)$$
Note the map $\frM$ is recovered by taking contractible $\cU$.

A stable polarization on $\mathcal{U}$ is a null-homotopy of $\mathcal{U} \to BU \to B (U/O)$, so determines by composition
a choice of Maslov data.  We have  already seen that stable polarizations can be used to define microsheaves on $\cU$; now we will  observe that in fact Maslov data can play the same role: 

\begin{theorem} \label{thm: descent} 
$\mu sh_{(U/O)(\xi)}$ is the pullback along  $\frM_{\xi}$ of a sheaf of categories on $ BPic(\cC)(\xi)$,
which we will denote $\mu sh_{BPic(\cC)(\xi)}$. 
\end{theorem} 

\begin{proof}
Consider the sheaf  $(\boxtimes_\cC^{m}  \mush_{U/O} )   \boxtimes_\cC \mush_{U/O(\xi)} $ 
on the diagram 
\beq\label{desc diag}
\xymatrix{
U/O(\xi) & \ar[l]_-{a_{\xi}}  U/O \times U/O(\xi) & \ar@<-0.5ex>[l]  \ar@<0.5ex>[l]  U/O \times U/O \times  U/O(\xi)  & \ar@<0ex>[l]  \ar@<-0.75ex>[l]  \ar@<0.75ex>[l]      \cdots 
}
\eeq
of Proposition~\ref{polarization equivariance}. 
Let us analyze its restriction to the subdiagram
\beq\label{desc subdiag}
\xymatrix{
U/O(\xi) & \ar[l]_-{a_{\xi}}  K \times U/O(\xi) & \ar@<-0.5ex>[l]  \ar@<0.5ex>[l]  K \times K \times  U/O(\xi)  & \ar@<0ex>[l]  \ar@<-0.75ex>[l]  \ar@<0.75ex>[l]      \cdots 
}
\eeq

From the canonical isomorphism $\mush_{U/O} \simeq \frM^* F_{can}$, and the construction $K = \ker(\frM)$,  we have a canonical identification $\mush_{U/O}|_K \simeq \cC_K$ where we write $ \cC_K$ for the constant $\cC$-valued sheaf on $K$. Similarly, the restriction of the sheaf
$ (\boxtimes_\cC^{m}  \mush_{U/O} )   \boxtimes_\cC \mush_{U/O(\xi)}$ on \eqref{desc diag}  
is canonically isomorphic to the sheaf $(\boxtimes_\cC^{m}  \cC_K)   \boxtimes_\cC \mush_{U/O(\xi)} $
on  \eqref{desc subdiag}. This in turn canonically extends to the action diagram
\beq
\xymatrix{
  U/O(\xi) & \ar@<-0.5ex>[l]  \ar@<0.5ex>[l]  K  \times  U/O(\xi)  & \ar@<0ex>[l]  \ar@<-0.75ex>[l]  \ar@<0.75ex>[l]    K \times K \times  U/O(\xi)    \cdots 
}
\eeq
of $K$ on $  U/O(\xi)$
where the additional maps are the evident projections. Thus it descends to the colimit
$B Pic(\cC)(\xi) 
\simeq K\bs (U/O)(\xi) $.
\end{proof}

\begin{definition}
Given a choice of Maslov data $\tau$, i.e.~a section of  the bundle $BPic(\mathcal{C})_{\mathcal{U}} \to \cU$,
we define $\mush_{\tau} := \tau^* \mush_{BPic(\mathcal{C})(\xi)}$.
\end{definition}

Evidently, if Maslov data $\tau$ came from a stable polarization $\sigma$, then we have a canonical equivalence 
$\mush_{\Lambda; \tau} \simeq \mush_{\Lambda; \sigma}$. In particular,  if two (possibly different) polarizations determine equivalent Maslov data, then they determine the same sheaf.

\begin{proposition}\label{prop:gen version}
Theorem \ref{thm: quantization}, Corollary \ref{cor: weinstein quantization} and Theorem \ref{thm: invariance}
hold as stated with $\tau$ understood as Maslov data rather than as a polarization. 
\end{proposition}
\begin{proof}
In the proof of those theorems, replace all objects with the Lagrangian Grassmannian bundles over them. 
\end{proof}

\begin{remark} \label{rem: secondary maslov data}
Consider in particular the case of a Weinstein manifold $W$ and a smooth exact Lagrangian $L\subset W$. 
Fixing Maslov data $\tau$ on $W$, we have by Cor. \ref{cor: weinstein quantization} a fully faithful functor
$\mush_{L, \tau}(L) \to \mush_{W, \tau}(W)$.  Evidently this is similar to the statement that a Lagrangian (equipped
with appropriate structures) determines an object in the Fukaya category.  To make a more direct comparison, 
note that in a neighborhood (of the contact lift of) $L$, there is a canonical polarization $\ell$, coming from the structure 
as a jet bundle.  We know that $\mush_{L, \ell}(L)$ is nothing other than the category of local systems on $L$.  Thus
an isomorphism $\ell \cong \tau$ of Maslov data determines 

$$Loc(L) \cong \mush_{L, \ell}(L) \cong \mush_{L, \tau}(L) \to \mush_{W, \tau}(W)$$

We term such an isomorphism ``secondary Maslov data''.  (It has elsewhere been called `brane data' \cite{jin-treumann}.)  
Note the obstruction to the existence of this data is the difference $\ell - \tau \in [L, BPic(\mathcal{C})]$. 
In the case
when the coefficient category $\mathcal{C}$ is the category of $\Z$-modules, $Pic(\mathcal{C}) = \Z \oplus B(\Z/2)$,
and so $[L, BPic(\mathcal{C})] = H^1(L, \Z) \oplus H^2(L, \Z/2)$.  When in addition $W = T^*M$ and $\tau$ is the fiber
polarization, then the obstruction to the existence of secondary
Maslov data would be the Maslov class of $L$ in the first factor, and $w_2(L, M)$ in the second, though in fact these
are known to always vanish.  However, the same calculation applies for Legendrians in $J^1(M)$, which can have
nontrivial such classes. 
\end{remark}

\begin{remark}
Work of Jin \cite{jin-BO, jin-J} shows that the map $\chi: LGr\to BPic(\mathcal{C})$ is in fact the topologists' J-homomorphism
in the universal case when $\mathcal{C}$ is the category of spectra.  The specialization to the case of $\mathcal{C}$ the category of 
$\Z$ modules  is essentially in \cite{guillermou}. 
We do {\em not} depend on these results for the abstract setup discussed above, but they are of course invaluable for
actually constructing or characterizing Maslov obstructions and Maslov data in practice. 
\end{remark}

\appendix

\section{Symplectic structures on relative cotangent bundles}  \label{cotangent lagrangian grassmannian}

Here we discuss the elementary differential geometry of constructing symplectic (or contact) forms on the total spaces of fiberwise cotangent bundles of fibrations over symplectic (or contact) manifolds.

Let $\frf:E\to M$ be a fiber bundle of smooth manifolds.
We will denote points of $E$ by pairs $(m, q)$ where $m\in M$, and $q\in \frf^{-1}(m)$.

Let $\pi:T^*E \to E$ denote the cotangent bundle. We will be interested in the relative
tangent and cotangent bundles, which are bundles on $E$ characterized by the  short exact sequences
$$
\xymatrix{
0 \ar[r] &  T\frf \ar[r]^-{\iota} & TE \ar[r]^-{\frf_*}  & \frf^* TM \ar[r] &  0 
}
$$
$$
\xymatrix{
0 \ar[r] &  \frf^* T^*M \ar[r]^-{\frf^*} & T^*E \ar[r]^-{\iota^\vee} & T^* \frf \ar[r] &  0
}
$$

We write $\pi_\frf:T^*\frf \to E$ for the natural projection, and introduce the composition
$$
\xymatrix{
\frF = \frf\circ \pi_\frf: T^*\frf \ar[r] &  M
}
$$
We
denote points of $T^*\frf$ by triples $(m, q, p)$ where $(m, q)\in E$, and $p\in \pi_\frf^{-1}(m, q)$,
and identify $E$ with the zero-section $\{p = 0\}\subset T^*\frf$. 

A connection on $\frf$ is given by a splitting  
\begin{equation} \label{connection tangent split}
\xymatrix{
0 \ar[r] &  T\frf \ar[r]_-{\iota} & \ar@/_1em/[l]_-s TE \ar[r]_-{\frf_*}  & \frf^* TM  \ar@/_1em/[l]_-t   \ar[r] &  0 
}
\end{equation}
or equivalently, as a dual splitting  of the second exact sequence
\begin{equation} \label{connection cotangent split}
\xymatrix{
0 \ar[r] &  \frf^* T^*M \ar[r]_-{\frf^*} & \ar@/_1em/[l]_-{t^\vee}  T^*E \ar[r]_-{\iota^\vee} & T^* \frf \ar[r]   \ar@/_1em/[l]_-{s^\vee}&  0
}
\end{equation}
Note that,
by construction, the natural pairing splits: 
for $\xi \in T^* E$ and $v \in T E$, we have
$$\langle \xi, v \rangle_E = \langle \xi, (\iota s + t \frf_*)(v) \rangle_E = \langle \iota^\vee(\xi) , s(v)  \rangle_\frf + \langle  t^\vee(\xi), \frf_* v \rangle_M$$
where brackets with subscripts indicate the canonical pairings between dual  bundles. 

{\bf For the remainder of the appendix, we fix some $\frf:E\to M$ and the data of a connection, 
which we usually interact with through $s^\vee$.  
All constructions will depend on the choice of connection.  The space of such choices is contractible.} 

\subsection{Fiberwise canonical one-form} 
Let $\theta_{can} \in \Omega^1(T^*E)$ denote the canonical one-form on the cotangent bundle $T^*E$.
We define: 
$$
\theta_\frf:=(s^\vee)^*\theta_{can}
$$
Note that $\theta_{can}$ vanishes on all vectors based along the zero-section $E\subset T^*E$, and so 
$\theta_{\frf}$ vanishes on all vectors based along the zero-section $E\subset T^*\frf$.
In fact, $\theta_\frf$ is a ``fiberwise canonical one-form'': 

\begin{lemma} \label{lem:fiberwisecanonical}
For any $m \in M$, the restriction $\theta_\frf|_{T^* E_m}$ is the canonical one-form on $T^* E_m$. 
\end{lemma}
\begin{proof}
Consider the natural commutative diagram with
compatible  sections
$$
\xymatrix{
 T^*E \ar[r]_-{\iota^\vee} & T^* \frf    \ar@/_1em/[l]_-{s^\vee}\\
\ar@{->>}[d]_-q\ar@{^(->}[u]^-i T^*E|_{E_m}  \ar[r]_-{\iota^\vee} & T^* \frf|_{E_m}    \ar@/_1em/[l]_-{s^\vee}\ar@{^(->}[u]^-i\\
T^*E_m \ar@{=}[ur] & 
}
$$

The canonical one-form $\theta_{can, m} \in \Omega^1(T^*E_m)$ is characterized by $q^*\theta_{can, m} = i^*\theta_{can}$. Using the section $s^\vee$, it can be calculated as $\theta_{can, m} = (s^\vee)^*i^*\theta_{can}$.
By commutativity of the diagram,  
$(s^\vee)^*i^*\theta_{can} = i^*(s^\vee)^*\theta_{can}$.
\end{proof}

\begin{example}
When $M=pt$, we have $T^*\frf = T^*E$ is the absolute cotangent bundle, and $\theta_{\frf} = \theta_{can}$ is the usual canonical one-form. 
\end{example}

Next, consider the short exact sequence
\begin{equation}\label{eq:ses}
\xymatrix{
0 \ar[r] &  T\frF \ar[r]^-{I} & T (T^*\frf \ar[r]^-{\frF_*})  & \frF^* TM \ar[r] &  0 
}
\end{equation}
of bundles on $T^*\frf$, where by definition $T\frF := \ker(\frF_*)$.

\begin{lemma}
The two-form $d\theta_\frf \in \Omega^2(T^*\frf)$ is non-degenerate on $T\frF \subset T (T^*\frf)$.
Its kernel $K :=\ker(d\theta_\frf) \subset T(T^*\frf)$ provides a splitting of \eqref{eq:ses} in the sense that
$\frF_*$ restricts to an isomorphism
\begin{equation} \label{eq: f star k} 
\xymatrix{
\frF_*|_K :K \ar[r]^-\sim & \frF^* TM
}
\end{equation}
\end{lemma}

\begin{proof}
We have $d\theta_\frf =d ( (s^\vee)^*\theta_{can} )= (s^\vee)^*d\theta_{can} = (s^\vee)^*\omega_{can}$. 
We have  $T \frF|_{E_m} = T(T^*\frf |_{E_m}) = T(T^* E_m)$.  
From Lemma \ref{lem:fiberwisecanonical}, the restriction of $(s^\vee)^*\omega_{can}$ to each
such fiber is the canonical symplectic form, hence non-degenerate.
On the other hand, $d\theta_\frf$ clearly vanishes on $\frf^*TM$.
\end{proof}

\begin{lemma} \label{lem:symprelcot}
If $\omega_M \in \Omega^2(M)$ is a symplectic form,
then $$\omega_\frf:= \frF^* \omega_M + d \theta_\frf \in \Omega^2(T^*\frf)$$ is a symplectic form.   
If $L \subset M$ is $\omega_M$-Lagrangian, then  $E|_L \subset T^*\frf$ is $\omega_\frf$-Lagrangian.
\end{lemma}

\begin{proof}
We have the splitting $T (T^*\frf) \simeq    \frF^* TM \oplus K$
where $K =\ker(d\theta_\frf)$. Non-degeneracy of $\omega_M$ implies $\frF^* TM = \ker(\frF^*\omega_M)$. Since $\omega_M$ is non-degenerate on $K =\ker(d\theta_\frf)$, and $d\theta_\frf$ is non-degenerate on $\frF^* TM =\ker(\frF^*\omega_M)$, we conclude that $\omega_\frf$ is non-degenerate.

Finally, if $L \subset M$ is Lagrangian, then $\omega_M|_L = 0$; on the other hand $\theta_\frf |_E = 0$, hence $\omega_\frf |_{E|_L} = 0$. 
\end{proof}

\begin{lemma} \label{lem:contrelcot}
If $\lambda_M\in \Omega^1(M)$ is a contact form, 
then $\lambda_\frf:=  \frF^* \lambda_M + \theta_\frf \in \Omega^1(T^*\frf)$ is a contact form.
 If $\Lambda \subset M$ is Legendrian, then so is $E|_\Lambda \subset T^*\frf$.
\end{lemma}

\begin{proof} We must show 
$d\lambda_\frf= \frF^* d \lambda_M + d \theta_\frf \in \Omega^2(T^*\frf)$ is non-degenerate on $\xi = \ker(\lambda_\frf) \subset T(T^*\frf)$, or other words, that $\ker(d\lambda_\frf) \cap \xi = \{0\}$.
Observe that $\ker(d\lambda_\frf)  \subset K = \ker(d\theta_\frf)$ since $d \theta_\frf$ is non-degenerate on 
$T\frF \subset T(T^*\frf)$. Recall also \eqref{eq: f star k}:  $\frF_*|_K:K\stackrel{\sim}{\to} \frf^*TM$, hence $\ker(d\lambda_\frf)   \subset K \cap \ker(\frf^*\lambda_M)$. But since $d\lambda_M$ is non-degenerate on $\ker(\lambda_M)$, we obtain the assertion.
\end{proof} 

\subsection{Liouville structures}

\begin{lemma}
Suppose $\lambda_M \in \Omega^1(M)$ is such that $(M, d \lambda_M)$ is symplectic.  Then
$$\lambda_\frf := \frF^* \lambda_M + \theta_\frf \in \Omega^1(T^*\frf)$$ 
makes $(T^*\frf, d \lambda_\frf)$ symplectic.   If $L \subset M$ is (exact) Lagrangian, then so is $E|_L \subset T^* \frf$.
\end{lemma} 
\begin{proof}
Symplecticness of  $d\lambda_\frf$ is 
immediate from Lemma \ref{lem:symprelcot}.  
Exactness of $E|_L$ follows from the fact that $\theta_\frf|_E = 0$.  Alternatively, 
a Lagrangian is exact iff it is the image of a Legendrian in the contactization; so we may apply 
Lemma \ref{lem:contrelcot}. 
\end{proof}

For an exact symplectic manifold $(X, \omega_X = d \lambda_X)$, we write $v_X$ for the Liouville
vector field, which is characterized by $i_{v_x}(\omega_X) = \lambda_X$.  
For $X = T^*E$, we have denoted $\omega_{can} = d \theta_{can} \in \Omega^2(T^*E)$ for the symplectic
form and so write $v_{can}$ for the Liouville vector field.

\begin{lemma} \label{lem: rel cot liouville vector field}
The Liouville vector field $v_\frf\in \Vect(T^*\frf)$ for 
$(T^*\frf, \omega_\frf = d \lambda_\frf)$ is  given by the formula
$$
\xymatrix{
v_\frf = \tilde v_M + \tilde v_{can}
}
$$
where $\tilde v_M$ is the preimage of $v_M \in TM$ under the 
isomorphism \eqref{eq: f star k} (and lies in particular in $ K = \ker(d\theta_\frf) \subset T T^*\frf$); and 
$\tilde v_{can}$ generates dilations of $ T^*\frf$ (and lies in particular in $\ker (\pi_{\frf*})$). 
\end{lemma}

\begin{proof}
Let us first explain the relationship of $\tilde v_{can}$ to $v_{can}$.
The latter is the Euler vector field generating the dilation on $T^*E$. 
As dilations  of $T^*E$ preserve the image $s^\vee(T^*\frf) \subset T^*E$ of the
bundle splitting $s^\vee: T^*\frf \to T^*E$, we see that 
$(s^\vee)_* \tilde v_{can} = v_{can}|_{s^\vee(T^*\frf)}$;
as $s^\vee$ is injective, this characterizes $\tilde{v}_{can}$.  
Note that $\tilde v_{can}$ lies in the kernel of $\pi_{\frf*}:T(T^*\frf) \to TE$, hence  
the kernel of $\frF_* = \frf_* \circ \pi_{\frf*}:T(T^*\frf) \to TM$.

Note the  identities
$$
\xymatrix{
i_{\tilde v_M}(\frF^*d\lambda_M) = \frF^* i_{\frF_* \tilde v_M}(d\lambda_M) = \frF^* i_{v_M}(d\lambda_M) = \frF^*\lambda_M
&
i_{\tilde v_M}(d\theta_\frf) = 0  
}
$$
with the latter due to the fact that  $\tilde v_M\in K = \ker(d\theta_\frf)$.

Note the additional identity
$$
\xymatrix{
i_{\tilde v_{can}}(\frF^*d\lambda_M) =  0 
}
$$
due to the fact that $\frF_*\tilde v_{can} =0$. 

Thus we must show
$$
\xymatrix{
i_{\tilde v_{can}}(d\theta_\frf) =  \theta_\frf 
}
$$
We calculate
$$
\xymatrix{
i_{\tilde v_{can}}(d\theta_\frf) =   i_{\tilde v_{can}}((s^\vee)^*d\theta_{can}) =  
(s^\vee)^* i_{(s^\vee)_*\tilde v_{can}}(d\theta_{can}) =  
(s^\vee)^* i_{v_{can}}(d\theta_{can}) =  
(s^\vee)^* \theta_{can} = \theta_\frf  
}
$$
and we're done.
\end{proof}

\begin{corollary}
If $v_M\in  \Vect(M)$ satisfies $d\varphi (v_M) >0$ over an open $U\subset M$ (resp. is gradient-like) for a function $\varphi:M\to \BR$,
 then  $v_\frf \in \Vect(T^*\frf)$  satisfies $d\Phi (v_\frf) >0$ over the open $\frF^{-1}(U)\subset T^*\frf$ (resp. is gradient-like) for
the function $\Phi = \frF^*\varphi + Q:T^*\frf\to \BR$, for any positive definite quadratic form $Q$ on the bundle $\frf:T^*\frf \to E$.
\end{corollary} 

\begin{corollary}
If $(M, \lambda_M)$ be a Liouville (resp. Weinstein) manifold, 
then  so is $(T^*\frf, \lambda_{\frf})$.  Moreover, the core satisfies
$$\mathfrak{c}(T^* \frf, \lambda_\frf) =  E|_{\mathfrak{c}(M, \lambda_M)}$$
\end{corollary} 
\begin{proof}
Follows from Lemma \ref{lem: rel cot liouville vector field} and the previous corollary.  
Note that since we assume $\frf:E\to M$ is proper, we may lift any integral curve of $v_M$ to an integral curve of $\tilde v_M$. 
\end{proof}

\begin{example}
Let $(M, \omega_M = d\lambda_M)$ be a Liouville (resp. Weinstein) manifold with a compatible almost complex structure $J_M$.
Then we can take $E$ to be the unitary frame bundle of $M$ and choose any principal connection. Furthermore, for any compact manifold $F$
with a unitary action, 
 in particular any homogenous space, 
 we can take $E$ to be the associated bundle. For example, we can take $E$ to be the Lagrangian Grassmannian bundle of  $M$. 
\end{example}

\subsection{Polarizations and the Lagrangian Grassmannian} 

\begin{lemma} \label{lem:sympcansec}
Let $(M, \omega)$ be a symplectic manifold, $\frf: LGr(TM) \to M$ the bundle of Lagrangian Grassmannians of its symplectic
tangent bundle, and $T^* \frf$ the total space of the relative cotangent bundle.  A choice of almost complex structure 
on $M$, and connection as in \eqref{connection tangent split},  determines a symplectic structure on $T^* \frf$ along with a canonical Lagrangian distribution, i.e. 
section of $LGr(T T^*\frf) \to T^* \frf$. 
\end{lemma}
\begin{proof}
We saw above that such a complex structure determines a connection on $\frf$ and thence a symplectic structure 
on $T^* \frf$.  We also have a  splitting of bundles on $T^*\frf$ (induced by the connection) $T (T^* \frf) = T \frF \oplus \frF^* TM$, where $\frF: T^*\frf \to M$ 
was the bundle of relative cotangent bundles.  Note $T \frF|_{E_m} = T T^* E_m$, and that $T \frF$ contains as a sub-bundle
$\pi_f^* T^* \frf$.

Recall we describe points of $T^* \frf$ as $(m, q, p)$ where $m$ is a point in $M$, $q$ is a point in the Lagrangian Grassmannian of $T_m M$, and $p \in 
\pi_\frf^{-1}(m, q)$ is a relative cotangent vector.  Now over a point $(m, q, p)$, we can take the direct sum of the subspace of $\frF^*TM$ named by $q$
with the fibre of $\pi_\frf^* T^* \frf$ summand.  Evidently this is isotropic for $\omega_\frf$, and half-dimensional, hence Lagrangian. 
\end{proof} 

\begin{lemma} \label{lem:contcansec}
Let $V$ be a contact manifold with contact form $\lambda$, let $\frf: LGr(ker\, \lambda) \to M$ the bundle of Lagrangian Grassmannians of the contact
distribution, and $T^* \frf$ the total space of the relative cotangent bundle.  A choice of almost complex structure 
on $\ker(\lambda)$, and connection as in \eqref{connection tangent split}, determines a contact form $\lambda_\frf$ on $T^* \frf$ along with a canonical Lagrangian distribution in $\ker(\lambda_\frf)$.
\end{lemma}
\begin{proof}
Similar to the previous lemma. 
\end{proof}

\newpage


\bibliographystyle{plain}
\bibliography{sqesc}

\end{document}